\documentclass[reqno]{amsart}

\usepackage[english]{babel}
\usepackage{amsmath}
\usepackage{amsthm}
\usepackage{amsfonts}
\usepackage{amssymb}
\usepackage{anysize}
\usepackage{hyperref}
\usepackage{appendix}
\usepackage{mathtools}
\usepackage{graphicx}
\usepackage{color}
\usepackage{forest} 
\usepackage{subcaption} 
\usepackage{cancel}
\usepackage[normalem]{ulem}
\usepackage{enumitem}

\newtheorem{defi}{Definition}[section]
\newtheorem{lem}[defi]{Lemma}
\newtheorem{theo}[defi]{Theorem}
\newtheorem{cor}[defi]{Corollary}
\newtheorem{pro}[defi]{Proposition}

\newtheorem{rem}[defi]{Remark}

\DeclareMathOperator{\divop}{div}

\DeclareMathOperator*{\esssup}{ess~sup}

\DeclareMathOperator{\Law}{Law}
\DeclareMathOperator{\Lip}{Lip}

\DeclareMathOperator{\RR}{\mathbb{R}}

\DeclareMathOperator{\R}{\mathbb{R}}

\newcommand{\E}{{\mathbb E}}
\DeclareMathOperator{\Tree}{\mathsf{T}}
\DeclareMathOperator{\Words}{\mathsf{W}}

\allowdisplaybreaks

\title[Mean-field limit of
non-exchangeable systems]{Mean-field limit of
non-exchangeable systems}

\author[Pierre-Emmanuel Jabin]{Pierre-Emmanuel Jabin}
\address[Pierre-Emmanuel Jabin]{\newline Dept. of Mathematics, Huck institutes, and Excellence Research Unit ``Modeling Nature'', Pennsylvania State University, University Park, PA 16803, USA} \email{pejabin@psu.edu}

\author[David Poyato]{David Poyato}
\address[David Poyato]{\newline Institut Camille Jordan (ICJ), UMR 5208 CNRS \& Universit\'e Claude Bernard Lyon 1, 69100 Villeurbanne, France, and Research Unit ``Modeling Nature'' (MNat), Universidad de Granada, Granada, 18071, Spain} 
\email{poyato@math.univ-lyon1.fr \ davidpoyato@ugr.es}

\author[Juan Soler]{Juan Soler}
\address[Juan Soler]{\newline Departamento de Matem\'atica Aplicada and Research Unit ``Modeling Nature'' (MNat), Facultad de Ciencias, Universidad de Granada, 18071 Granada, Spain}
\email{jsoler@ugr.es}

\begin{document}


\thanks{\textbf{Acknowledgment.}  This work has been partially supported by the grants: DMS Grant  1908739, 2049020, 2205694 and 2219397 by the NSF (USA);  by the State Research Agency of the Spanish Ministry of Science and FEDER-EU, project PID2022-137228OB-I00 (MICIU/AEI /10.13039/501100011033); by Modeling Nature Research Unit, Grant QUAL21-011 funded by Consejería de Universidad, Investigaci\'on e Innovaci\'on (Junta de Andalucía);  by the European Union's Horizon Europe research and innovation program under the Marie Sk{\l}odowska-Curie grant agreement No 101064402, and by the European Union’s Horizon 2020 research and innovation program (grant agreement No 865711).}

\begin{abstract}
This paper deals with the derivation of the mean-field limit for multi-agent systems on a large class of sparse graphs. More specifically, the case of non-exchangeable multi-agent systems consisting of non-identical agents is addressed. The analysis does not only involve PDEs and stochastic analysis but also graph theory through a new concept of limits of sparse graphs (extended graphons) that reflect the structure of the connectivities in the network and has critical effects on the collective dynamics. In this article some of the main restrictive hypothesis in the previous literature on the connectivities between the agents (dense graphs) and the cooperation between them (symmetric interactions) are removed.
\end{abstract}

\maketitle

\tableofcontents

\section{Introduction. Multi-agent systems on graphs}\label{sec:introduction}
We derive in this article  the mean-field limit for multi-agent systems posed on a large class of graphs. More precisely, we consider the dynamics of $N$ agents, interacting pairwise through a complex map of connections between them. The information about this connectivity map is encompassed through a weighted, {\it a priori} non-symmetric graph $G_N$ where each vertex represents an agent. Denoting by  $X_i(t)\in \R^d$ 
the state of the $i$-th agent, and by $K$ the interaction kernel, this leads to the system
\begin{equation}\label{eq1}
\left\{
\begin{array}{rll}
\displaystyle \frac{d X_i}{dt} &=&
\displaystyle   \sum_{j=1}^N w_{ij}\,K(X_i-X_j) , \\ 
  X_i (0) &=&X_i^0,
\end{array}
\right.
\end{equation}
where the  $w_{ij}=w_{ij}^N\not=w_{ji}^N$ are the weight of the edge $(i,j)$ of the graph.

While we consider for simplicity deterministic systems like~\eqref{eq1}, our analysis would extend in a straightforward manner to stochastic multi-agent systems with additive noise such as
\begin{equation}\label{eq1stochastic}
dX_i =  \sum_{j=1}^N w_{ij}\,K(X_i-X_j)\,dt+\sigma_N\,dW_i, \\ 
\end{equation}
where the $W_i^t$ are independent Wiener processes.

This paper focuses on the analysis of the behavior of the system as $N\to \infty$, for smooth Lipschitz kernels~$K$ but with minimal assumptions on the connectivities~$w_{ij}$. We only impose scalings that naturally extend the classical notion of mean-field scaling:  First, the total interaction felt by an agent should be of order $1$, that is,
\begin{equation}
\max_{1\leq i\leq N}\sum_{j=1}^N |w_{ij}|=O(1),\quad \max_{1\leq j\leq N}\sum_{i=1}^N |w_{ij}|=O(1),\quad \mbox{as}\  N\rightarrow\infty,\label{meanfieldscaling}
  \end{equation}
ensuring that complex dynamics emerge at time scales of order $1$ as well. Secondly, no single agent is allowed to have a dominant role, that is,
\begin{equation}
\max_{1\leq i,j\leq N} |w_{ij}|=o(1),\quad \mbox{as}\ N\rightarrow\infty,\label{vanishingweights}
  \end{equation}
which, together with \eqref{meanfieldscaling}, implies that a given agent is influenced by the average of a large number of other agents, so that we expect some concept of mean-field limit to persist.

System~\eqref{eq1} is a canonical example of so-called non-exchangeable or non-interchangeable systems. This refers to the fact that \eqref{eq1} is not symmetric under permutation: If $(X_i(t))_{1\leq i\leq N}$ solves \eqref{eq1} and $\sigma\in \mathcal{S}_N$ is any permutation, then $(X_{\sigma(i)}(t))_{1\leq i\leq N}$ is not in general a solution to~\eqref{eq1}. Instead the same permutation~$\sigma$ would need to be applied to the connectivities~$w_{ij}$, thus modifying the system. Another way of formulating this is that agents are not identical because the effect of the $j$-th on the $i$-th agent may be very different from the effect on the $k$-the agent: One could have $w_{ij}\gg 1$ but $w_{ik}\approx 0$ for example.

Classically, studies of many-particle systems have often focused on identical agents and exchangeable systems which corresponds to choosing~$w_{ij}=1/N$ in~\eqref{eq1}. The question of the mean-field limits for exchangeable systems where all agents are identical is now rather well understood for Lipschitz interaction kernels, with even significant progress on singular kernels; see for example the reviews~\cite{golse,Jab,JW2}. The case of non-exchangeable systems and non-identical agents had remained comparatively less explored but recently started to receive a lot of attention because of their critical importance for many applications and their connections to diverse areas of Mathematics. The analysis of non-exchangeable multi-agent systems indeed not only involves PDEs and stochastic analysis but also graph theory. On the other hand, multi-agent systems are also found in a large number of fields from their classical applications in Physics to the Bio-sciences, Social Sciences or Economics. In many of those settings, agents naturally appear as non-identical: The structure of the connectivities in the network has proved to have critical effects on the dynamics, see among many~\cite{PhPaChVi,PJ, PiShPaShLiChSi,W-S,WC}, and the role of the network of interactions is often critical. We present one example of such application in Appendix \ref{neurons-a} which is based on models for the dynamics of biological neurons for which we briefly refer for example to
\cite{Compte,FPZ,H,MG1,MG2,sporns,WJ}.

The main result of this paper is to incorporate the complex non-interchangeable graph structures defined by $w_{ij}$ into a simple limiting Vlasov equation
\begin{equation} \label{vlasov-eq-lim}
\partial_t f(t,x,\xi)+\divop_x\,\left(f(t,x,\xi) \,\int_0^1 w(\xi,d\zeta)\int_{\R^d} K(x-y)\, f(t,dy,\zeta)\,\right)=0,
\end{equation}
through the extended graphon $w(\xi,d\zeta)$.

In \eqref{vlasov-eq-lim}, $f$ describes the probability of finding an agent at position $x \in \R^d $, in a certain state of interaction with the network $\xi \in [0,\ 1]$, at time $ t> 0$. This means that the role of an agent can be represented by adding an extra one-dimensional variable. In this way, we associate to the discrete connectivities $w_{ij}$ a function $w$, or some more general object, which accounts for the effect of the architecture of the complex original system \eqref{eq1} as the number $N$ of agents is large. This type of graphon-like representation had previously been explored in particular in~\cite{CM1,CM2,CMM,KalMed,Medvedev2,Medvedev} or  in~\cite{GkoKueh,KueXu}. However those results required more stringent assumptions on the connectivities, so that for example \cite{CM1,CM2,CMM,KalMed} only apply to dense graphs or symmetric interactions, whilst \cite{GkoKueh,KueXu} allow for some sparse graphs, all of them requiring an {\it a priori} knowledge on how weights are generated as to discretize a prescribed limiting object. Of course, this way of prescribing weights provides some additional convergence, which is essential in these results. The case of sparse random graphs such as the famous Erd\"os-R\'enyi graphs, has also been treated for example in~\cite{Cop22,CLP23,Medvedev}. In these references one finds the common assumption that $Np_N\to \infty$, where $p_N\in(0,1)$ represents the probability of drawing an edge. This condition characterizes the generalized mean-field scaling \eqref{meanfieldscaling}-\eqref{vanishingweights} on this class of sparse random graphs. Some even sparser Erd\"os-R\'enyi graphs with $Np_N=O(1)$ have also been handled in~\cite{LRW23, ORS20}, but this time they do not satisfy the generalized mean-field scaling  \eqref{meanfieldscaling}-\eqref{vanishingweights} and lead instead to diffusive systems. While graphons are capable of handling limits of some sparse graphs \cite{BCCZ1,BCCZ2}, this usually requires some sort of renormalization of $w_{ij}$.

One of the major contributions of the present article is to  offer a unified framework to derive the mean-field limit without any assumptions, structural or not, outside of the extended mean-field scaling provided by the assumptions~\eqref{meanfieldscaling}-\eqref{vanishingweights}. This comes with considerable difficulties and impose a completely new strategy to obtain the limit equation.
\subsection{Mean--field limits for identical weights}
Before going further in the technical aspects of our work, we recall here the main notions for the idea of classical mean--field limit in exchangeable systems, corresponding to the simple case~$w_{ij}=\bar w/N$, 
\begin{equation} \label{eqm1}
\left\{\begin{array}{rll}
\displaystyle\frac{d {X}_i}{dt} &=&\displaystyle  \frac{\bar w}{N}\sum_{j=1}^N {K}( {X}_i - {X}_j),\\
\displaystyle  {X}_i(0)&=& {X}_{i}^{0}.
\end{array}\right.
\end{equation}
A critical question for any multi-agent system is how to choose initial data for~\eqref{eqm1}, since it is not feasible in almost any applications to individually identify or measure every single $X_i^0$. Classical assumptions rely on the notion of molecular chaos by imposing that the $X_i^0$ be independent and identically distributed ({\it i.i.d.}) random variables (or almost {\it i.i.d.} in some appropriate sense) according to an initial probability distribution~$f^0\in \mathcal{P}(\mathbb{R}^d)$.

This leads to the key notion of propagation of chaos, which consists in proving that the $X_i(t)$ solving~\eqref{eqm1} still remain approximately {\it i.i.d.} (again in some appropriate measure). For exchangeable systems such as~\eqref{eqm1}, propagation of chaos often then directly implies that the limit as $N\to\infty$ of the 1-particle distribution of the system solves the mean-field equation
\begin{equation}\label{eqm4-1}
\partial_t f(t,x)+\divop_x\left( \bar w \, f(t,x)\, \int_{\RR^d} {K}( {x} - y)\,f(t,dy)\right)=0.
\end{equation}
We refer to the seminal results in~\cite{BH,Do,NW} and later~\cite{spohn} for the derivation of~\eqref{eqm4-1} for $K\in W^{1,\infty}(\RR^d)$. As mentioned earlier, the case of non Lipschitz kernels corresponds to many realistic applications, it is much more difficult and still poorly understood to some degree. An extensive discussion would however carry us too far from the main goals of this article and we only briefly refer again to~\cite{golse,Jab,JW2}.

A useful notion for system~\eqref{eqm1} is the so-called empirical measure which is given by,
\begin{eqnarray*}\label{eqm2}
\mu^N(t, {x}):=\frac{1}{N}\sum_{i=1}^N\delta_{{X}_i(t)}(x),
\end{eqnarray*}
for every $t\geq 0$. In the deterministic case, the empirical measure is itself already a solution in distributional sense to the mean-field equation \eqref{eqm4-1}, which offers a straightforward approach to deriving the mean-field limit as $N\rightarrow\infty$, for Lipschitz kernels~$K$, through a stability analysis of Eq.~\eqref{eqm4-1} in the space of measures.  Specifically, the empirical measure offers a very simple way to formulate and understand the mean-field limit in terms of the weak convergence of the random measure~$\mu^N$ to the deterministic limit~$f$. Classical results for Lipschitz kernels for example imply that
  \[
\|\mu_N-f\|_{W^{-1,1}}\leq \|\mu_N^0-f\|_{W^{-1,1}}\,e^{t\,\bar w \|\nabla K\|_{L^\infty}}.
  \]
Some major differences are immediately apparent for non-exchangeable systems like~\eqref{eq1}. First of all, because agents are non-identical, it may not be realistic to assume that they are initially identically distributed.  We will still assume that the $X_i^0$ are independent though as this is essential for the reduction in complexity leading to~\eqref{eqm4-1}. Instead of propagation of chaos, we consequently have to deal with the weaker notion of {\em propagation of independence}.

It is possible to extend some of the classical approaches to prove this propagation of independence (as we do in the proof later). Unfortunately propagation of independence no longer directly implies the limit~\eqref{eqm4-1} for general non-exchangeable systems, creating major difficulties. The one exception occurs in the special case where the degree of each node of the graph is independent of $i$, that is, 
\[
\sum_{j=1}^N w_{ij}=\bar w,
\]
for all $i=1,\ldots,N$. In that setting, it is possible to still derive the classical mean-field limit~\eqref{eqm4-1}; see \cite{DGL} for {\it i.i.d} initial positions $X_i^0$, and to \cite{CDG} for the stability for this  initial distribution.
For general connectivities~$w_{ij}$ however, new ideas and new methods seem to be needed.
\subsection{The use of graphons for multi-agent systems}
%
Graphons offer a non-parametric method of modeling and estimating large networks and are constructed as limits of sequences of dense graphs. There now exists a large literature dedicated to graphons for which we refer to the seminal \cite{BCKL,BCLSV06,BCLSV2,BCLSV3,Lo}.

To give a rough idea of how graphons play a critical role in non-exchangeable systems, consider a sequences of dense graphs $G_N$ with nodes indexed from $1$ to $N$ and adjacency matrix $(w_{ij}^N)_{1\leq i,j\leq N}$ with the key assumption that
\begin{equation}
w_{ij}^N=\frac{\bar w_{ij}^N}{N}\quad \mbox{with}\quad  \max_{1\leq i,j\leq N}\bar w_{ij}^N=O(1),\quad \mbox{as}\  N\rightarrow\infty.\label{densegraph}
  \end{equation}
For a fixed $N$, it is straightforward to define the graphon over $[0,\ 1]^2$
\begin{equation}
w_N(\xi,\zeta)=\sum_{i,j=1}^N N w_{ij}^N\, \mathbb{I}_{[\frac{i-1}{N},\frac{i}{N})}(\xi)\,\mathbb{I}_{[\frac{j-1}{N},\frac{j}{N})}(\zeta).\label{graphon}
  \end{equation}
One directly observes that $w_N$ is uniformly bounded in $L^\infty([0,\ 1]^2)$ but a fundamental property behind graphons is that $w_N$ can be made compact in the appropriate distance after re-arrangements. More precisely, there exists a measure-preserving map $\phi_N$ on $[0,\ 1]$ such that the rearranged object $w_N(\phi_N(\xi),\phi_N(\zeta))$ is compact in the norms of operators from $L^\infty([0,\ 1])$ to $L^1([0,\ 1])$; this norm is equivalent to the so-called {\it cut metric}, see \cite{Lo}.  Identifying $w$ with its associated adjacency operator $\phi\in L^\infty([0,\ 1])\mapsto w(\phi)=\int_0^1 \phi(\zeta)\,w(\cdot,\zeta)\,d\zeta\in L^1([0,\ 1])$, its operator norm \(\| \cdot \|_{L^\infty \to L^1}\) is defined by 
$$\|w\|_{L^\infty \to L^1} = \sup_{\|\phi\|_{L^\infty} \leq 1} \|w(\phi)\|_{L^1}=\sup_{\Vert \phi\Vert_{L^\infty},\Vert \psi\Vert_{L^\infty}\leq 1}\int_{[0,\ 1]^2}\psi(\xi)\,\phi(\zeta)\,w(\xi,\zeta)\,d\xi\,d\zeta.$$ 
The previous compactness property of $w_N$ in the operator norm under the assumption that $w_N$ are uniformly bounded in $L^\infty([0,\ 1]^2)$  was initially established in~\cite{LS} using a weak version of the Szemer\'edi Regularity Lemma~\cite{sze} and a martingale argument.

This turns out to be a key property for the purpose of deriving the mean-field limit. For example, it provides a natural way to obtain weak stability for Eq.~\eqref{vlasov-eq-lim}. Namely, consider a sequence of weak solutions $f_N$ to \eqref{vlasov-eq-lim} with weights $w_N$, where $w_N$ are uniformly bounded in $L^\infty([0,\ 1]^2)$ and $f_N$ are uniformly bounded in $L^\infty([0,\ t_*]\times [0,\ 1],\ W^{1,1}\cap W^{1,\infty}(\R^d))$, for all $t_*>0$. Assume further that $f$ is any other solution to \eqref{vlasov-eq-lim} with weights $w$ lying in the same regularity class as above. We refer to Proposition \ref{existenceweak} for the precise well-posedness result of \eqref{vlasov-eq-lim}. Then, we can infer the stability estimate
\begin{equation}\label{E-weak-stability-graphons}
\frac{d}{dt}\int_0^1\int_{\mathbb{R}^d}|f_N-f|\,dx\,d\xi\leq C_1\int_0^1\int_{\mathbb{R}^d}|f_N-f|\,dx\,d\xi+C_2\,\Vert w_N-w\Vert_{L^\infty\to L^1},
\end{equation}
where $C_1,C_2$ are constants depending on the norms of $f$, the uniform bound of the norm of $w_N$ and the $L^1$ and Lipschitz norms of $K$. This estimate can easily be derived formally and how to obtain it rigorously goes beyond the limited technical scope of this introduction. We just mention that, since $K$ is Lipschitz, \eqref{E-weak-stability-graphons} can also follow from the notion of solutions in Proposition \ref{existenceweak} as long as $f_N$ and $f$ are smooth in $t$ and $x$: if $f_N,\,f\in W^{1,\infty}_{t,x} L^\infty_\xi$ for example, then we can use any $\chi(f_N-f)$ for $\chi$ smooth as a test function on \eqref{vlasov-eq-lim}.

Consequently, using the above theory of dense graph limits, one can set the measure-preserving maps $\phi_N$ on $[0,\ 1]$ so that the re-arranged objects $w_N(\phi_N(\xi),\phi_N(\zeta))$ converge (up to subsequence) in the operator norm, and also set $w$ to be the obtained limit. By doing so, the second term in the right hand side of \eqref{E-weak-stability-graphons} converges to zero and, assuming well-prepared initial data, we can pass to the limit as $N\rightarrow\infty$ in $f_N(t,x,\phi_N(\xi))$ and obtain a solution $f(t,x,\xi)$ to~\eqref{vlasov-eq-lim}. As a consequence, we can determine the limit of the $1$-particle distribution given by~$\int_0^1 f_N(t,x,\xi)\,d\xi$:
\[
\int_0^1 f_N(t,x,\xi)\,d\xi=\int_0^1 f_N(t,x,\phi_N(\xi))\,d\xi\to \int_0^1 f(t,x,\xi)\,d\xi, \quad \mbox{as}\ N\to\infty.
\]
We refer to~\cite{CM1,CM2,CMM,KalMed} in particular for a complete analysis that this brief sketch cannot do justice to. An alternative and more primitive approach to derive macroscopic limits for \eqref{eq1} was obtained in \cite{Medvedev2}, again for graphons $w\in L^\infty([0,\ 1]^2)$ and Lipschitz kernesl $K$, in terms of the graph-limit equation
\begin{equation}\label{E-graph-limit-equation}
\partial_t X(t,\xi)=\int_0^1 w(\xi,\zeta)\,K(X(t,\xi)-X(t,\zeta))\,d\zeta.
\end{equation}
Here, $X\in C^1(\mathbb{R}_+,L^\infty([0,\ 1])$ can be interpreted as a time-evolving parametrization of the continuum of agents by a continuous index $\xi\in [0,\ 1]$. Whilst a large portion of the literature sticks to this simpler formulation, we remark that \eqref{E-graph-limit-equation} corresponds to a special class of solutions to \eqref{vlasov-eq-lim} having the form $f(t,x,\xi)=\delta_{X(t,\xi)}(x)$, but clearly not all solutions of \eqref{vlasov-eq-lim} come from solutions to \eqref{E-graph-limit-equation}.

The critical issue  to implement a similar stability estimate like \eqref{E-weak-stability-graphons} in our case is that our assumptions~\eqref{meanfieldscaling}-\eqref{vanishingweights} of course cannot ensure that the weights $w_{ij}$ satisfy~\eqref{densegraph}. This concerns both dense and sparse graphs: sparse graph sequences in general will not satisfy~\eqref{densegraph}, but also many dense graph sequences may fail to satisfy~\eqref{densegraph}. As we noted above, one may still use graphons to characterize limits of some sparse graph sequences. For example, for sparse deterministic graphs \cite[Theorem 2.8]{BCCZ1} shows that any $C$-upper $L^p$ regular sequence of
weighted graphs with $p>1$ converges to some $w\in L^p([0,\ 1])$ in the cut distance. For sparse random graphs, \cite[Theorem 2.14, Corollary 2.15]{BCCZ1} derives a general convergence result of the sequence of $w$-random sparse graphs generated by any $w \in L^1([0,\ 1])$ toward $w$ itlsef in the cut distance. In both cases, a suitable renormalization on the graphon is needed. One fundamental advantage of the theory developed in this paper is that it applies to all sparse graphs which satisfy~\eqref{meanfieldscaling}-\eqref{vanishingweights}. However, our analysis and motivation are not limited to extend previous results to sparse graph sequences.

Various extensions of graphons have been proposed to handle less stringent assumptions than~\eqref{densegraph}. We mention in particular the aforementioned extension to {\it $L^p$ graphons} in~\cite{BCCZ1,BCCZ2}, the more general operator-based extension in~\cite{BacSze} based on {\it graphops}, and the measure-theoretic extension in \cite{KunLovSze} based on {\it s-graphons}. Each of the three classes of continuum objects contain the family of finite graphs, and they can be endowed with a compact metric space structure under a certain topology, similarly to what it was done for graphons. We emphasize that the above assumption~\eqref{meanfieldscaling} represents a scaling of weights $w_{ij}$ and only provides uniform bounds for $w_N$ in the mixed spaces $L^\infty_\xi\mathcal{M}_\zeta=L^\infty_\xi([0,\ 1],\ \mathcal{M}_\zeta([0,\ 1])$ and $L^\infty_\zeta\mathcal{M}_\xi=L^\infty_\zeta([0,\ 1],\ \mathcal{M}_\xi([0,\ 1])$. The scaling of weights in $L^\infty_\xi\mathcal{M}_\zeta \cap L^\infty_\zeta\mathcal{M}_\xi$ corresponds to a continuous versions of the same discrete version~\eqref{meanfieldscaling} of the scaling and suggests defining a new class of continuum objects which we shall call {\it extended graphons}, see next section and Definition \ref{D-extended-graphons} below for further details. A similarly scaled extension was developed in~\cite{KunLovSze}, based on {\it digraph measures} or bounded everywhere defined families of measures, and is discussed more at length after Theorem~\ref{maintheorem}. Closely related is also the scaling in~\cite{Lucon}, where the author considered limits of random graphs sampled from graphons not necessarily bounded but belonging to the mixed space $L^\infty_\xi([0,\ 1],L^2_\zeta([0,\ 1]))$. In contrast with classical graphons, our measure-valued definition of extended graphons reflects, in part, the sparsity of the network (as they are intimately related to graphops, s-graphons, and digraph measures), but also, the microscopic inhomogeneity of the network. A fundamental example of this lies in constant graphons, which are associated with homogeneous graphs. We refer to Section \ref{subsec:example-convergence} for some examples sparse and inhomogeneous graph sequences which converge to extended graphons under our notion of convergence. 

\begin{figure}[t]
\def\svgwidth{0.6\textwidth}
\centering
\begingroup%
  \makeatletter%
  \providecommand\color[2][]{%
    \errmessage{(Inkscape) Color is used for the text in Inkscape, but the package 'color.sty' is not loaded}%
    \renewcommand\color[2][]{}%
  }%
  \providecommand\transparent[1]{%
    \errmessage{(Inkscape) Transparency is used (non-zero) for the text in Inkscape, but the package 'transparent.sty' is not loaded}%
    \renewcommand\transparent[1]{}%
  }%
  \providecommand\rotatebox[2]{#2}%
  \newcommand*\fsize{\dimexpr\f@size pt\relax}%
  \newcommand*\lineheight[1]{\fontsize{\fsize}{#1\fsize}\selectfont}%
  \ifx\svgwidth\undefined%
    \setlength{\unitlength}{408.64993928bp}%
    \ifx\svgscale\undefined%
      \relax%
    \else%
      \setlength{\unitlength}{\unitlength * \real{\svgscale}}%
    \fi%
  \else%
    \setlength{\unitlength}{\svgwidth}%
  \fi%
  \global\let\svgwidth\undefined%
  \global\let\svgscale\undefined%
  \makeatother%
  \begin{picture}(1,0.77642154)%
    \lineheight{1}%
    \setlength\tabcolsep{0pt}%
    \put(0,0){\includegraphics[width=\unitlength,page=1]{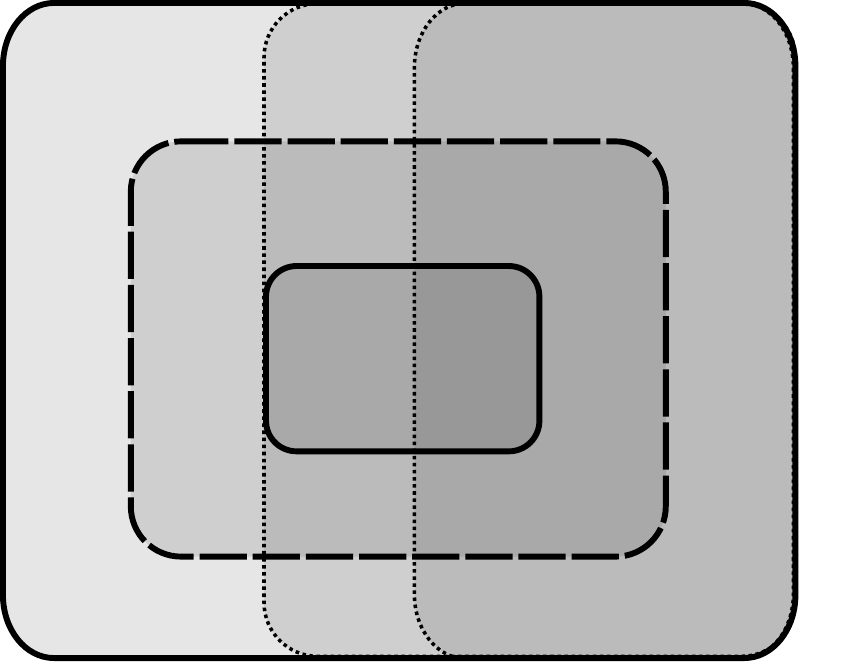}}%
    \put(0.37854123,0.35241791){\color[rgb]{0,0,0}\makebox(0,0)[lt]{\lineheight{1.25}\smash{\begin{tabular}[t]{l}graphons \cite{Lo}\end{tabular}}}}%
    \put(0.67156341,0.05300539){\color[rgb]{0,0,0}\makebox(0,0)[lt]{\lineheight{1.25}\smash{\begin{tabular}[t]{l}s-graphons \cite{KunLovSze}\end{tabular}}}}%
    \put(0.05039916,0.6809039){\color[rgb]{0,0,0}\makebox(0,0)[lt]{\lineheight{1.25}\smash{\begin{tabular}[t]{l}graphops \cite{BacSze}\end{tabular}}}}%
    \put(0.182568,0.52040951){\color[rgb]{0,0,0}\makebox(0,0)[lt]{\lineheight{1.25}\smash{\begin{tabular}[t]{l}{\bf (symmetric) extended graphons}\end{tabular}}}}%
    \put(0.35720001,0.68129722){\color[rgb]{0,0,0}\makebox(0,0)[lt]{\lineheight{1.25}\smash{\begin{tabular}[t]{l}(symmetric) digraph measures \cite{KueXu}\end{tabular}}}}%
  \end{picture}%
\endgroup%

\caption{Comparison of the various graph limit theories.}
\label{fig:graphons-comparison}
\end{figure}
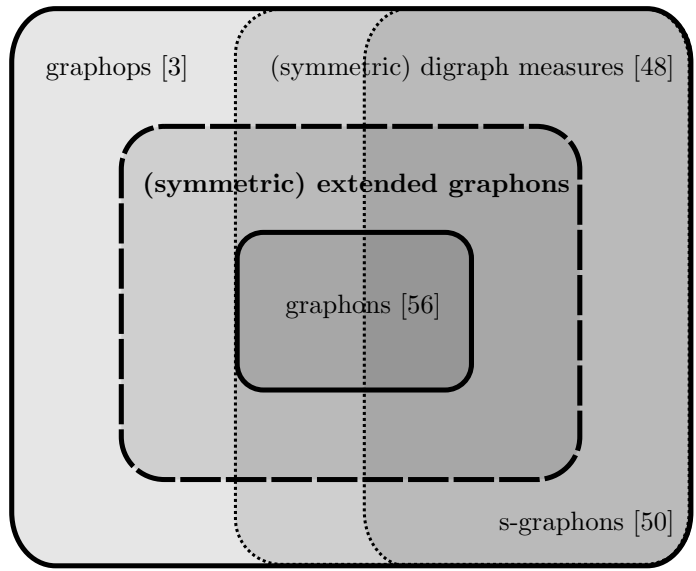

The goal of this paper is not to provide a complete classification of the existing theories of graph limits. However, the above classes of objects can be compared, and so has been done in more dedicated literature, {\it e.g.} \cite{BacSze,KunLovSze} and also \cite{GkoKueh,KueXu}. Without any intention of comparing their topologies, the previous classes of objects are related as sketched in Figure \ref{fig:graphons-comparison}, at least when restricted to symmetric weights. We note that extended graphons account for an intermediate degree of sparsity as they occupy an intermediate position between the class of graphons and graphops. This fact will become clearer later in Section \ref{sec:hierarchy}, and more particularly in the operator representation of extended graphons in Lemma \ref{bilinearbound}. However, we do emphasize the topology that we use on those extended graphons is very distinct from other known objects.
\subsection{Our new result}
The main contribution of this paper is to provide the mean-field limit with only assumptions~\eqref{meanfieldscaling}-\eqref{vanishingweights}. This first requires a careful definition of the space for limiting kernel which we call {\it extended graphons} and we denote by~$L^\infty_\xi \mathcal{M}_\zeta\cap L^\infty_\zeta \mathcal{M}_\xi$; namely~$w\in L^\infty_\xi \mathcal{M}_\zeta\cap L^\infty_\zeta \mathcal{M}_\xi$ iff
\begin{equation}
w(\xi,d\zeta)\in L^\infty([0,\ 1],\ \mathcal{M}([0,\ 1])),\quad w(d\xi,\zeta)\in L^\infty([0,\ 1],\ \mathcal{M}([0,\ 1])),\label{deflimitingspace}
\end{equation}
where we define the space~$L^\infty([0,\ 1],\ \mathcal{M}([0,\ 1]))$ as the topological dual to the Banach space~$L^1([0,\ 1], $ $C([0,\ 1]))$, endowed with the corresponding weak-* topology.

We are now ready to state our main result,
  \begin{theo} \label{maintheorem}
Assume that the interaction kernel $K$ belongs to $W^{1,1}\cap W^{1,\infty}(\R^d)$ and consider a sequence $(X_i(t))_{1\leq i\leq N}$, solving the multi-agent system~\eqref{eq1} for connection weights satisfying only~\eqref{meanfieldscaling}-\eqref{vanishingweights}.  Assume moreover that the initial data $X_i^0$ are independent random variables, but not necessarily identically distributed,  and denote their laws by $f_i^0$. Finally assume that 
\[
\sup_{N\in \mathbb{N}} \sup_{1\leq i\leq N}\mathbb{E}[\vert X_i^0\vert^2]<\infty,\quad \sup_{N\in \mathbb{N}} \sup_{1\leq i\leq N} \|f_i^0\|_{W^{1,1}\cap W^{1,\infty}(\R^d)}<\infty.
\]
Then, there exist $w\in L^\infty_\xi \mathcal{M}_\zeta\cap L^\infty_\zeta \mathcal{M}_\xi$ and $f\in L^\infty([0,\ t_*]\times[0,\ 1],\;W^{1,1}\cap W^{1,\infty}(\R^d))$, for any $t_*>0$, such that $f$ solves~\eqref{vlasov-eq-lim} with $w$ and, up to the extraction of a subsequence,
\[
\sup_{0\leq t\leq T}\,\mathbb{E}\,W_1\left(\int_0^1 f(t,\cdot,\xi)\,d\xi,\;\mu_N(t,\cdot)\right)\to 0,
\]
as $N\to \infty$, where $W_1$ is the usual Wasserstein distance acting on the subset of $\mathcal{P}(\mathbb{R}^d)$ with finite first order moments, and $\mu^N(t,x):=\frac{1}{N}\sum_{i=1}^N \delta_{X_i(t)}(x)$ are the empirical measures associated to $(X_i(t))_{1\leq i\leq N}$.
 \end{theo}

Compared to the existing literature, Theorem~\ref{maintheorem} is the only result so far capable of handling the mean-field limit where we only assume assumptions~\eqref{meanfieldscaling}-\eqref{vanishingweights} on the connections $w_{ij}$ without any other structural assumptions. As mentioned earlier in the introduction, when the graph has some specific structure, Erd\"os-R\'enyi for example, other approaches exist that only require the scaling~\eqref{meanfieldscaling}-\eqref{vanishingweights} \cite{Cop22,CLP23,LRW23,Medvedev,ORS20}. We do not obtain classical graphons since  the limiting kernel $w$ only belongs to $L^\infty_\xi \mathcal{M}_\zeta\cap L^\infty_\zeta \mathcal{M}_\xi$, which is a natural space corresponding to the scaling in~\eqref{meanfieldscaling}. As mentioned above, a space with a similar scaling is the space  of  digraph measures $B([0,\ 1], \mathcal{M}([0,\ 1])$. This space was inspired in the s-graphons~\cite{KunLovSze}, specifically introduced to handle sparse graphs, and it was used in~\cite{KueXu} for mean-field limits. There are a few key differences with the present analysis, as  digraph measures $B([0,\ 1], \mathcal{M}([0,\ 1])$ consists of everywhere defined in $\xi$ bounded measures in $\zeta$ while our extended graphons in $L^\infty_\xi \mathcal{M}_\zeta\cap L^\infty_\zeta \mathcal{M}_\xi$ are only almost everywhere defined in $\xi$ (or in $\zeta$). The use of extended graphons offers several advantages when passing to the limit. In particular, in comparison~\cite{KueXu} appears to require an improved continuous dependence in $\xi$ of the digraph measures that further restricts the type of connectivities that can be used with respect to~\eqref{meanfieldscaling}-\eqref{vanishingweights}. 

But our notion of extended graphons in~$L^\infty_\xi \mathcal{M}_\zeta\cap L^\infty_\zeta \mathcal{M}_\xi$ does lead to some major issues, as it does not seem possible to pass to the limit in Eq.~\eqref{vlasov-eq-lim} directly. That imposes a new strategy briefly outlined in the sketch of proof in the next subsection. Some of that technical difficulty can be glimpsed from the fact that it is not yet clear in which sense Eq.~\eqref{vlasov-eq-lim} is posed with only $w\in L^\infty_\xi \mathcal{M}_\zeta\cap L^\infty_\zeta \mathcal{M}_\xi$. The main issue is the non-linear term of course since $f$ is only  essentially bounded in $\xi$ with no additional regularity, and only {\it a.e.} defined. More specifically, it is not immediately obvious how to make sense of terms like
\[
\xi\in [0,\ 1]\mapsto \int_0^1 \phi(\zeta)\,w(\xi,d\zeta),
\]
for $\phi\in L^\infty([0,\ 1])$. This is where the proper definition of $L^\infty_\xi \mathcal{M}_\zeta\cap L^\infty_\zeta \mathcal{M}_\xi$ is needed and one reason why we require $w$ to belong both to $L^\infty_\xi \mathcal{M}_\zeta$ and to $L^\infty_\zeta \mathcal{M}_\xi$ (which would be automatic if $w$ was symmetric). A careful analysis then allows proving that the above function is correctly defined in $L^\infty([0,\ 1])$ with 
\[
\left\|\int_0^1 \phi(\zeta)\,w(\xi,d\zeta)\right\|_{L^\infty_\xi}\leq \|w\|_{L^\infty_\xi \mathcal{M}_\zeta\cap L^\infty_\zeta \mathcal{M}_\xi}\,\|\phi\|_{L^\infty},
\]
as we prove later in Lemma~\ref{bilinearbound}. Note that the above also allow identifying any $w\in L^\infty_\xi \mathcal{M}_\zeta\cap L^\infty_\zeta \mathcal{M}_\xi$ with a bounded linear operator $L^\infty([0,\ 1])\to L^\infty([0,\ 1])$, and therefore this explains the above-mentioned relation in Figure \ref{fig:graphons-comparison} between extended graphons and graphops.

  We also note that our proofs could imply an even more general formulation of Theorem~\ref{maintheorem}. For example the assumptions $f_i^0\in L^1\cap L^\infty(\R^d)$ would be enough, and it should even be possible to further relax some of the assumptions on the graphs with
      \[
\frac{1}{N}\sum_{i=1}^N \left(\sum_{j=1}^N |w_{ij}|\right)^p=O(1),\quad \frac{1}{N}\sum_{i=1}^N \sup_{1\leq j\leq N} |w_{ij}|^{1/p} =o(1),\quad \mbox{as}\  N\rightarrow\infty,
      \]
      for some $p>1$, instead of \eqref{meanfieldscaling}-\eqref{vanishingweights}.

As in classical mean-field limits, independence of initial data is the key to the reduction of complexity provided by Theorem~\ref{maintheorem}. We assumed for simplicity that the $X_i^0$ were fully independent but it would be straightforward to relax this condition by having them approximately independent.      
      
Theorem~\ref{maintheorem} obtains the limit up to the extraction of a subsequence in $N$, since it makes no assumptions on the convergence of the initial data and it then amounts to a compactness result. Nevertheless, it is possible to derive the convergence of the whole sequence but the formulation is trickier than usual because of the role played by the connectivities~$w_{ij}$. The first necessary assumption is naturally the convergence of the full sequence of empirical measures: There exists $\bar f^0\in W^{1,\infty}(\R^d)$ such that
\begin{equation}
\mathbb{E}\,W_1\left(\bar f^0,\;\mu_N^0\right)\to 0,\quad \mbox{as}\ N\to \infty.\label{convergenceinitialempirical}
\end{equation}
Alternatively, since $X_i^0$ are independent, one may assume convergence of the $1$-particle distribution
\[
W_1\left(\bar f^0,\;\frac{1}{N}\,\sum_{i=1}^N f_i^0 \right) \to 0,\quad \mbox{as}\ N\to \infty.
\]
However \eqref{convergenceinitialempirical} is not enough and some additional convergence
of the $w_{ij}$ is also required.

Theorem~\ref{maintheorem} opens up new avenues to the analysis but still leaves some important questions unresolved. The main and obvious issue concerns the Lipschitz assumption on the kernel~$K$. This is a critical issue for many applications which involve some form of singular interactions. Among several examples, we mention the integrate and fire system for biological neurons that we briefly describe in Appendix~\ref{neurons-a}.  Integrate and fire models only very loosely fit within the multi-agent framework that we focus on in this paper  so that extending the present analysis to integrate and fire model would be a highly non trivial extension.

The Lipschitz regularity of $K$ plays a critical role in the first step of our analysis when propagating independence. It does not appear as essential in the next, more complex steps of our proofs. Replacing the classical trajectory methods to obtain propagation of independence could allow obtaining an equivalent of Theorem~\ref{maintheorem} with less stringent assumptions on~$K$.

\section{Notations, basic definitions, examples and sketch of the proof}
\subsection{Notations}
We denote by $\mathcal{M}(E)$ the space of finite Radon measures on $E$, where here the underlying space $E$ could be $\R^d$, the interval $[0,\ 1]$, or also some product of those spaces like $\R^d\times [0,\ 1]$ or $[0,\ 1]^2$. The space $\mathcal{P}(E)$ denotes the subspace of probability measures on $E$, that is measures $\mu\in \mathcal{M}(E)$ that are non-negative, with mass $1$.

For a given measure $\mu$, we denote by $\mu(x)$ the abstract notion and by $\mu(dx)$ when integrating against $\mu$. For example, the Dirac mass at $0$ is defined by $\mu(x)=\delta_0(x)$ and we write
\[
\int_E \phi(x)\,\mu(dx)=\int_E \phi(x)\,\delta_0(dx)=\phi(0),
\]
for all continuous function $\phi$.

This notation offers a convenient way of representing marginals for measures of several variables. For example, if $\mu\in \mathcal{M}(E\times E)$, then $\nu_1(x)=\int_E \mu(x,dy)$ simply denotes the first marginal of $\mu$, that is,
\[
\nu_1(O)=\mu(O\times E),
\]
for any measurable subset $O\subset E$, while $\nu_2(y)=\int_E \mu(dx,y)$ represents the second marginal of $\mu$.

We recall that $\mathcal{M}(E)$ is the topological dual of $C_0(E)$, the space of continuous functions with vanishing limit at infinity (more precisely, $C_0(\R^d)$ for $E=\R^d$ and $C([0,\ 1])$ for $E=[0,\ 1]$). Unless otherwise specified, we always use the corresponding implied weak-* topology on $\mathcal{M}(E)$. Specifically, we have that  $\mu_n\to\mu$ for the weak-* topology if, and only if,
\[
\int_E \phi(x)\,\mu_n(dx)\to \int_E \phi(x)\,\mu(dx),
\]
for all $\phi\in C_0(E)$. We also recall that, by the the Banach–Alaoglu theorem, any ball of finite radius of $\mathcal{M}(E)$ is precompact for this topology. If $E$ is compact, for example $E=[0,\ 1]$, then $\mathcal{P}(E)$ is a compact metric space for the weak-* topology.
\subsection{The space $L^\infty([0,\ 1],\ \mathcal{M}([0,\ 1])$}
We define the space $L^\infty([0,\ 1],\ \mathcal{M}([0,\ 1]))$ as the topological dual of the Bochner space $L^1([0,\ 1],\ C([0,\ 1]))$. So defined $L^\infty([0,\ 1],\ \mathcal{M}([0,\ 1]))$ is not a Bochner space because $\mathcal{M}([0,\ 1])$ fails the Radon–Nikodym property \cite{BT,DU}. However, by a variant of the Riesz representation theorem operating in the absence of the Radon-Nikodym property \cite{IT-77} we have that the dual space $L^\infty([0,\ 1],\ \mathcal{M}([0,\ 1]))$ amounts to a weak-* Bochner space. Specifically, objects $w\in L^\infty([0,\ 1],\ \mathcal{M}([0,\ 1]))$ are alternatively identified with maps $\xi\in [0,\ 1]\longmapsto w(\xi,\zeta)\in \mathcal{M}([0,\ 1])$ which are weakly-* measurables and such that there exists $C>0$ such that
$$\left\Vert \int_0^1 \phi(\zeta)\,w(\cdot,d\zeta)\right\Vert_{L^\infty([0,\ 1])}\leq C\,\Vert \phi\Vert_{C([0,\ 1])},$$
for all $\phi\in C([0,\ 1])$. We refer to Definition \ref{D-weak-star-Bochner-space} for further details. The best constant $C$ above defines the norm on the weak-* Bochner space $L^\infty([0,\ 1],\ \mathcal{M}([0,\ 1]))$ and, by the above duality theorem, it is equivalent to the norm of $L^\infty([0,\ 1],\ \mathcal{M}([0,\ 1]))$ as a dual space. The corresponding duality reads
\[
\left<w,\phi\right>=\int_{[0,\ 1]^2} \phi(\xi,\zeta)\,w(\xi,d\zeta),
\]
for all $w\in L^\infty([0,\ 1],\ \mathcal{M}([0,\ 1]))$ and all $\phi\in L^1([0,\ 1],\ C([0,\ 1]))$. Except otherwise specified, we always imbue $L^\infty([0,\ 1],\ \mathcal{M}([0,\ 1]))$ with its weak-* topology as a dual space. By the Banach-Alaouglu theorem, any ball of finite radius of $L^\infty([0,\ 1],\ \mathcal{M}([0,\ 1]))$ is again precompact in the weak-* topology.

We recall that $L^1([0,\ 1],\ C([0,\ 1]))$ is a separable, Banach space. Moreover, the following inclusions
$$C^\infty([0,\ 1]^2)\subset C([0,\ 1]^2)\subset L^1([0,\ 1],\ C([0,\ 1])),$$
are dense, which directly implies that $L^\infty([0,\ 1],\ \mathcal{M}([0,\ 1]))$ can be identified with a subspace of $\mathcal{M}([0,\ 1]^2)$ (and therefore of distributions $\mathcal{D}'([0,\ 1]^2)$). This leads to the inclusions:
$$L^\infty([0,\ 1],\mathcal{M}([0,\ 1]))\subset \mathcal{M}([0,\ 1]^2)\subset \mathcal{D}'([0,\ 1]^2).$$
Specifically, a distribution $w\in \mathcal{D}'([0,\ 1]^2)$ belongs to $L^\infty([0,\ 1],\ \mathcal{M}([0,\ 1]))$ if, and only if, there exists some $C>0$ such that
\[
\int_{E^2} \phi(\xi,\zeta)\,w(\xi,d\zeta)\,d\xi\leq C\,\|\phi\|_{L^1([0,\ 1],\ C_0([0,\ 1]))},
\]
for all $\phi\in C^\infty([0,\ 1]^2)$, the best $C$ being the norm of $L^\infty([0,\ 1],\ \mathcal{M}([0,\ 1]))$ as a dual space.

Since the weak-* topology on a dual space is metrizable when induced on balls, then we have that $w_n\to w$ for the weak-* topology of $L^\infty([0,\ 1],\ \mathcal{M}([0,\ 1]))$ if, and only if, $w_n\to w$ in the sense of distributions \emph{and} $\sup_n \|w_n\|_{L^\infty([0,\ 1],\ \mathcal{M}([0,\ 1]))}<\infty$. This also directly implies that $C^\infty([0,\ 1]^2)$ functions are dense in $L^\infty([0,\ 1],\ \mathcal{M}([0,\ 1]))$, so that this space is separable for the weak-* topology.

Similarly, the space $L^\infty([0,\ 1], L^2 [0,\ 1])$ can be defined either through duality with $L^1([0,\ 1],\ L^2([0,\ 1]))$ or as a Bochner space. Since this time $L^2([0,\ 1])$ does verify the Radon-Nikodym property, the classical version of the Riesz representation theorem \cite{DU} ensure that both approaches coincide.

\subsection{The space $L^\infty_\xi \mathcal{M}_\zeta\cap L^\infty_\zeta \mathcal{M}_\xi$ and examples}

Depending on the order of variables, we can actually define two spaces $L^\infty_\xi\mathcal{M}_\zeta:=L^\infty_\xi([0,\ 1],\,\mathcal{M}_\zeta([0,\ 1]])$ and $L^\infty_\zeta\mathcal{M}_\xi:=L^\infty_\zeta([0,\ 1],\,\mathcal{M}_\xi([0,\ 1]])$. As above, both can be regarded as subspaces of distributions $\mathcal{D}'([0,\ 1]^2)$ through the dualities
$$
\left<w_1,\phi\right>:=\int_{[0,\ 1]}\phi(\xi,\zeta)\,w_1(\xi,d\zeta),\quad
\left<w_2,\phi\right>:=\int_{[0,\ 1]}\phi(\xi,\zeta)\,w_2(d\xi,\zeta),
$$
for every $w_1\in L^\infty_\xi\mathcal{M}_\zeta$, $w_2\in L^\infty_\zeta\mathcal{M}_\xi$ and all $\phi\in C^\infty([0,\ 1]^2)$. As subspaces of distributions we can easily define the intersection $L^\infty_\xi \mathcal{M}_\zeta\cap L^\infty_\zeta \mathcal{M}_\xi$ with the corresponding union of the two topologies. Equivalently, this could be defined as the topological dual of $L^1_\xi C_\zeta+L^1_\zeta C_\xi$, with the corresponding weak-* topology. The induced norm on $L^\infty_\xi \mathcal{M}_\zeta\cap L^\infty_\zeta \mathcal{M}_\xi$ is simply
\[
\|w\|_{L^\infty_\xi \mathcal{M}_\zeta\cap L^\infty_\zeta \mathcal{M}_\xi}=\max\{\|w\|_{L^\infty_\xi \mathcal{M}_\zeta},\|w\|_{L^\infty_\zeta \mathcal{M}_\xi}\}.
\]
We again have that $w_n\to w$ for the weak-* topology of $L^\infty_\xi \mathcal{M}_\zeta\cap L^\infty_\zeta \mathcal{M}_\xi$ if, and only if, $w_n\to w$ in the sense of distributions and $\sup_n \|w_n\|_{L^\infty_\xi \mathcal{M}_\zeta\cap L^\infty_\zeta \mathcal{M}_\xi}<\infty$. As before $L^\infty_\xi \mathcal{M}_\zeta\cap L^\infty_\zeta \mathcal{M}_\xi$ is also separable and $C^\infty([0,\ 1])$ functions are dense for the weak-* topology. Again, any ball of finite radius of $L^\infty_\xi\mathcal{M}_\zeta\cap L^\infty_\zeta\mathcal{M}_\xi$ is precompact for the same weak-* topology. The space $L^\infty_\xi\mathcal{M}_\zeta\cap L^\infty_\zeta\mathcal{M}_\xi$ will also be introduced later in Definition \ref{D-extended-graphons} under the name of extended graphons, and will serve as a new graph limit theory.

An elementary but illuminating example of elements in $L^\infty_\xi\mathcal{M}_\zeta\cap L^\infty_\zeta\mathcal{M}_\xi$ consists in considering just $w=\delta(\xi-\zeta)$. Obviously this kernel can be regarded in three different equivalent ways. First, as an element in $L^\infty_\xi\mathcal{M}_\zeta$, {\it i.e.}, a parametrized family $\xi\in[0,\ 1]\mapsto \delta_\xi(\zeta)\in \mathcal{M}([0,\ 1])$ of measures in $\zeta$. Second, as an element in $L^\infty_\zeta\mathcal{M}_\xi$, {\it i.e.}, a parametrized family $\zeta\in[0,\ 1]\mapsto \delta_\zeta(\xi)\in \mathcal{M}([0,\ 1])$ of measures in $\xi$. Finally, as the measure in $\mathcal{M}([0,\ 1]^2)$ supported on the diagonal of $[0,\ 1]^2$. The definitions above provide a simple framework where we can simply switch between each interpretation as desired. 

A more interesting example consists in considering $w=\delta(\Phi(\xi)-\zeta)$ where $\Phi:[0,\ 1]\longrightarrow [0,\ 1]$ is any one-to-one, measure-preserving map. We build on this example to provide an illustration of a non-trivial limiting kernel for our main result. As above, this kernel can be seen in three different ways. First, it is clearly a parametrized family $\xi\in [0,\ 1]\mapsto\delta_{\Phi(\xi)}(\zeta)\in \mathcal{M}([0,\ 1])$ of measures in $\zeta$. More precisely,
\[
\int_{[0,\ 1]} \phi(\zeta)\,w(\xi,d\zeta)=\phi(\Phi(\xi))
\]
for all $\phi\in C([0,\ 1])$, which yields obviously a measurable function in $\xi$, and therefore $w$ defines an element in $L^\infty_\xi\mathcal{M}_\zeta$. Of course, the above embedding allows regarding $w$ as a measure in $\mathcal{M}([0,\ 1]^2)$, that is,
$$\int_{[0,\ 1]}\varphi(\xi,\zeta)\,w(d\xi,d\zeta)=\int_{[0,\ 1]}\varphi(\xi,\Phi(\xi))\,d\xi=\int_{[0,\ 1]}\varphi(\Phi^{-1}(\zeta),\zeta)\,d\zeta,$$
for all $\varphi\in C([0,\ 1]^2)$, where in the last equality we have made the change of variables $\zeta=\Phi(\xi)$ as $\Phi$ is one-to-one and measure-preserving. This allows regarding $w$ also as a  a parametrized family $\zeta\in [0,\ 1]\mapsto\delta_{\Phi^{-1}(\zeta)}(\xi)\in \mathcal{M}([0,\ 1])$ of measures in $\xi$, that is,
\[
\int_{[0,\ 1]} \phi(\xi)\,w(d\xi,\zeta)=\phi(\Phi^{-1}(\zeta)),
\]
for all $\phi\in C([0,\ 1])$. Of course, they all represent the same object as for all $\varphi\in C([0,\ 1]^2)$ we have
\[
\int_{[0,\ 1]^2} \varphi(\xi,\zeta)\,w(\xi,d\zeta)\,d\xi=\int_{[0,\ 1]^2}\varphi(\xi,\zeta)\,w(d\xi,d\zeta)=\int_{[0,\ 1]^2} \varphi(\xi,\zeta)\,w(d\xi,\zeta)\,d\zeta.
\]

\subsection{Further comments on our extended graphons}
There are compelling reasons why extended graphons $w\in L^\infty_\xi \mathcal{M}_\zeta\cap L^\infty_\zeta \mathcal{M}_\xi$ appear to be the correct scale for the type of complex structures of connectivities that we are interested in.

We first point out that our extended graphons occupy an intermediary place in the current hierarchy of graphon-like objects, see Figure \ref{fig:graphons-comparison}. Obviously they are more general than classical graphons since they are not necessarily bounded, that is, $w\not\in L^\infty_{\xi,\zeta}=L^\infty([0,\ 1]^2)$ in general. However as mentioned above, the analysis will show later that any $w\in L^\infty_\xi \mathcal{M}_\zeta\cap L^\infty_\zeta \mathcal{M}_\xi$ can also be seen as the kernel of a bounded operator from $L^\infty([0,\ 1])\to L^\infty([0,\ 1])$; instead of an operator $L^\infty([0,\ 1])\to L^1([0,\ 1])$ for graphops. As such extended graphons appear as intermediary objects between classical graphons and other extensions such as graphops, s-graphons and digraph measures, see Figure \ref{fig:graphons-comparison}.

However the topology that we consider in our analysis for this space $L^\infty_\xi \mathcal{M}_\zeta\cap L^\infty_\zeta \mathcal{M}_\xi$ is one of the major differences with respect to the existing literature. The proper convergence for our large-graph limit cannot hold in the classical cut distance, or any direct extensions such as the topology of bounded linear operators $L^\infty([0,\ 1])\to L^\infty([0,\ 1])$, which we have defined earlier, to the best of our knowledge. This is clear from the stability estimate \eqref{E-weak-stability-graphons} for graphons, which inevitably breaks for unbounded graphons. Instead the key notion of convergence that we use is based on a new countable family of observables that mixes the connection kernel and the initial laws on each agent. Specifically, the observables can be defined as follows
\begin{equation}
    \begin{split}      &\tau(T,(w_{ij})_{1\leq i,j\leq N},(f_i^0)_{1\leq i\leq N})(0,x_1,\ldots,x_{|T|})=\frac{1}{N}\,\sum_{i_1,\ldots,i_{|T|}=1}^N \prod_{(k,l)\in E(T)} w_{i_k\,i_l}\;\prod_{m\in V(T)} f_{i_m}^0(x_{m}),
\end{split}\label{observabletree}
    \end{equation}
and they are indexed by finite trees~$T$, where $f_i^0$ denotes the initial law of the variable $X^0_i$. Those observables have a corresponding representation at the limit when the discrete weights $(w_{ij})_{1\leq i,j\leq N}$ and laws $f_i^0(x)$ are replaced by kernels $w(\xi,\zeta)$ and initial data $f^0(x,\xi)$ through the definition,
\begin{equation}
 \begin{split}      &\qquad  \tau(T,w,f)(0,x_1,\ldots,x_{|T|})=\int_{[0,\ 1]^{|T|}} \prod_{(k,l)\in E(T)} w(\xi_k, \xi_l)\;\prod_{m\in V(T)} f^0(x_{m},\xi_m)\,d\xi_1\dots d\xi_{|T|}.
\end{split}\label{observabletree2}
    \end{equation}
    
Those definitions allow defining precisely the convergence in Theorem~\ref{maintheorem}: If we have the strong convergence in $L^2(\mathbb{R}^{d|T|})$ of the observables $\tau(T,(w_{ij})_{1\leq i,j\leq N},(f_i^0)_{1\leq i\leq N}) \to \tau(T,w,f^0(\xi,x))$ as $N\to\infty$, for all trees $T$, then the dynamics of the $X_i$ can effectively be represented at the limit by the solution $f$ to~\eqref{vlasov-eq-lim},
\[
\frac{1}{N}\,\sum_i \delta_{X_i(t)}(x)\to \int_0^1 f(t,x,d\xi).
\]
As it can be seen in the sketch of proof below, the observables $\tau(T,(w_{ij})_{1\leq i,j\leq N},(f_i^0)_{1\leq i\leq N})$ indeed play a critical role in our analysis as they entirely control the dynamics.

Note that since $f_i^0(x)$ and $f^0(x,\xi)$ are normalized probability measures with respect to the variable $x$, then by integrating the observables \eqref{observabletree}-\eqref{observabletree2} on the spacial variables we are led to the well-known notion of moments on a graph, that is,
$$
\int_{\mathbb{R}^{d|T|}} \tau(T,(w_{ij})_{1\leq i,j\leq N},(f_i^0)_{1\leq i\leq N})(0,x_1,\ldots,x_{|T|})\,dx_1\ldots,\,dx_{|T|}=\tau(T,(w_{ij})_{1\leq i,j\leq N}),
$$
together with its continuum version
$$
\int_{\mathbb{R}^{d|T|}} \tau(T,w,f)(0,x_1,\ldots,x_{|T|})\,dx_1\ldots,\,dx_{|T|}=\tau(T,w),
$$
where $\tau(T,(w_{ij})_{1\leq i,j\leq N})$ is the {\it homomorphism density} of the finite tree $T$ in the finite graph $(w_{ij})_{1\leq i,j\leq N}$, and $\tau(T,w)$ is its continuum version, which are also used in the classical theory of graphons. In particular, under the above-mentioned convergence of the observables, we also have $\tau(T,(w_{ij})_{1\leq i,j\leq N})\to \tau(T,w)$ as $N\to\infty$ for all finite tree $T$. In a first major difference to graphons though, when $w\in L^\infty_\xi\mathcal{M}_\zeta\cap L^\infty_\zeta\mathcal{M}_\xi$ is an extended graphons we can only define the associated observables $\tau(T,w,f)$ and the homomorphism density $\tau(T,w)$ for trees~$T$ and not any arbitrary graph.   
  
A second critical difference is that our notion of observables completely entangles the kernel with the initial conditions. In particular our main theorem does not provide independent convergence of the extended graphons and the initial data. In other words (see the example below as well), we could have that for some initial data $f_i^0(x)$ and some connections $w_{ij}$, one has that $\tau(T,(w_{ij})_{1\leq i,j\leq N},(f_i^0)_{1\leq i\leq N}) \to \tau(T,w,f^0)$ for some limiting $w(\xi,d\zeta)$ and $f^0(x,\xi)$. If one then considers different initial data $\tilde f_i^0(x)$ but for the same connections $w_{ij}$, there is no particular reason why we would be able to use the same limiting kernel: We may need to derive a different limiting kernel $\tilde w(\xi,d\zeta)$ together with a different limiting initial law $\tilde f^0(x,\xi)$ to maintain the convergence $\tau(T,(w_{ij})_{1\leq i,j\leq N},(\tilde f_i^0)_{1\leq i\leq N}) \to \tau(T,\tilde w,\tilde f^0)$. 

While our main theorem ensures that there always exists a limiting representation, it is not necessarily unique. We can for instance find different kernels $w(\xi,d\zeta)\neq\tilde w(\xi,d\zeta)$ and initial data $f^0(x,\xi)\neq \tilde f^0(x,\xi)$ such that $\tau(T,w,f^0)=\tau(T,\tilde w,\tilde f^0)$; see the example below again. 

All of this makes the notion of convergence that is developed in the present paper very different from most of mean-field limit results based on graphons or graphon-like objects. Because our result applies to any connections with only the bounds \eqref{meanfieldscaling}-\eqref{vanishingweights} without any additional convergence assumptions, it is unclear whether it would even be possible to identify a specific topology on the connections $w_{ij}$ or kernels $w$ without mixing them with the initial data. Another fundamental question is whether our class $L^\infty_\xi \mathcal{M}_\zeta \cap L^\infty_\zeta \mathcal{M}_\xi$ is optimal for detecting the sparse connectivity and the microscopic inhomogeneity of networks, and again we believe that it will depend on the properties of the network and its connectivities to be studied. Note our main Theorem \ref{maintheorem} operates up to the choice of a subsequence, that is, once a suitable subsequence of graphs is a selected so that the observables $\tau(T,(w_{ij})_{1\leq i,j\leq N},(f_i^0)_{1\leq i\leq N})$ (and therefore the homomorphism densities $\tau(T,(w_{ij})_{1\leq i,j\leq N})$) converge. Therefore, a natural question is to understand how big is the set of all possible accumulations points via this procedure. We want to clarify that any $w\in L^\infty_\xi\mathcal{M}_\zeta\cap L^\infty_\zeta\mathcal{M}_\xi$ can be recovered in this way, which suggests that the class of extended graphons is optimal for that purpose. More specifically, and as established in Lemma \ref{gooddensity}, for any $w \in L^\infty_\xi \mathcal{M}_\zeta \cap L^\infty_\zeta \mathcal{M}_\xi$, there exists $w_N \in L^\infty_{\xi,\zeta}$ uniformly bounded in $L^\infty_\xi L^1_\zeta \cap L^\infty_\zeta L^1_\xi$, converging to $w$ in $L^1_\xi H^{-1}_\zeta \cap L^1_\zeta H^{-1}_\xi$. This finding is subsequently utilized in Lemmas \ref{welldefinedM} and \ref{lemtreeF} to establish the convergence of $\tau(T, w_N)$ to $\tau(T, w)$. Certainly, if $w_N$ is further uniformly bounded in $L^\infty_{\xi,\zeta}$ (which is compatible with weights satisfying the scaling \eqref{densegraph}), then $\tau(T, w_N) \to \tau(T, w)$ for some $w\in L^\infty_{\xi,\zeta}$, thus retrieving the classical result of Lov\'asz-Szegedy \cite{Lo,LS}.

Several other open questions concern the new system of observables that we introduce in~\eqref{observabletree2}. Those are conjectured to naturally extend the notion of marginals, and hierarchy of marginals to non-exchangeable systems. The structure of the hierarchy of observables~\eqref{observabletree2} remains rather poorly understood however. A first example is the range of possible functions that one can reach through~\eqref{observabletree2}: What would be the conditions on a sequence $\alpha_T(x_1,\ldots,x_{|T|})$ indexed by trees $T$ so that we can find $w\in L^\infty_\xi \mathcal{M}_\zeta\cap L^\infty_\zeta \mathcal{M}_\xi$ and $f\in L^\infty(\R^d\times [0,\ 1])$,  such that
        \[
\tau(T,w,f)=\alpha(T), \quad\forall\,T?
        \]

\subsection{Example of convergence}\label{subsec:example-convergence}
To illustrate those remarks, we discuss in some depth a natural example of connection weights, based on a sparse graph. Choose $1\ll M\ll N$ (with $N$ multiple of $M$ for simplicity), and separate $\{1,\ldots, N\}$ into $N/M$ subsets $E_1^N,\ldots, E_{N/M}^N$, with size $|E_k^N|=M$ for each $k\leq N/M$. In other words, we separate our agents into $N/M$ distinct subpopulations. 

Introduce further a permutation $\phi\in \mathcal{S}_{N/M}$ which lets us define
\[
w_{ij}=\left\{\begin{array}{ll}
\displaystyle\frac{1}{M} &\mbox{if}\ i\in E_k^N\ \mbox{and}\ j\in E_{\phi(k)}^N\ \mbox{for some}\ k,\\
0, & \mbox{otherwise}.
\end{array}\right.
\]
This choice of connection has a very simple interpretation. Agents are divided in $N/M$ distinct classes  or layers $E_k^N$, $k=1,\dots, N/M$. Each agent $i$ in a class $E_k^N$ is connected with exactly $M$ other agents that all belong to the class $E_{\phi(k)}^N$. 

We note that the adjacency matrix $w_{ij}$ is sparse as we have exactly $N\cdot M$ non-zero entries among a total of $N^2$.  Because $M$ is much smaller than $N$ this does lead to a sparse graph of connections. Note also that those connections satisfy \eqref{meanfieldscaling}-\eqref{vanishingweights} as long as $M\gg1$. In particular if $i\in E_k^N$
\[
\sum_{j=1}^N |w_{ij}|=\sum_{j\in E^N_{\phi(k)}} \frac{1}{M}=\frac{|E^N_{\phi(k)}|}{M}=1,
\]
and also similarly $\sum_{i=1}^N |w_{ij}|=1$ since $\phi$ is a permutation. We also observe that this weights satisfy the assumptions in~\cite{CDG}. Therefore, if the $X_i^0$ are initially {\it i.i.d.}, then we can derive the mean-field limit to the classical Vlasov equation, that is, we can choose as a limiting kernel $w=1$. This is not possible in general if the $X_i^0$ are not {\it i.i.d.}, and this example will provide an excellent illustration that the choice of the limiting kernel may have to depend on the initial laws.

Let us introduce the usual graphon representation of $w_N$
\[
w_N(\xi,\zeta)=\sum_{i,j=1}^N Nw_{ij}\, \mathbb{I}_{[\frac{i-1}{N},\frac{i}{N})}(\xi)\, \mathbb{I}_{[\frac{j-1}{N},\frac{j}{N})}(\zeta).
\]
To derive a more explicit formula for $w_N$, it is useful to have a good labeling of the indices on each class $E_k^N$. Of course, for any $k=1,\ldots,N$, we can list the indices in $E_k^N$ consecutively by
\[
E_k^N=\left\{(\sigma(k)-1)\,M+1,\ldots,\sigma(k)\,M\right\},
\]
where $\sigma\in \mathcal{S}_{N/M}$ is some permutation. Introduce now the approximation of the Dirac mass at scale $M/N$,
\[
\delta_{\frac{M}{N}}(\xi)=\frac{N}{M}\,\mathbb{I}_{\xi\in [0,\;\frac{M}{N})},
\]
and let us extend the discrete permutation $\phi$ to the step function $\phi_N:[0,\ 1]\longrightarrow [0,\ 1]$ with
\[
\phi_N(\xi)=(\phi(\sigma(k))-1)\,\frac{M}{N},\quad\mbox{if}\ \xi\in\left[(\sigma(k)-1)\,\frac{M}{N},\ \sigma(k)\,\frac{M}{N}\right),
\]
for all $k=1,\ldots,\frac{N}{M}$. We then have the simple expression
\[
w_N(\xi,\zeta)= \delta_{\frac{M}{N}}(\zeta-\phi_N(\xi)).
\]

Let us note immediately that the classical Lovasz-Szegedy theory based on graphons cannot apply here. First, we of course have that $\|w_N\|_{L^\infty}=\frac{N}{M}\to \infty$. Moreover, we cannot have convergence of $w_N$ in the cut distance as it is not a Cauchy sequence for this topology. Consider $w_N$ and $w_{N'}$ (with corresponding $M$ and $M'$) with $\frac{N}{M}\gg \frac{N'}{M'}$. Choose $\psi(\xi)=\mathbb{I}_{\xi\in [(\phi(\sigma(1))-1)\,\frac{M}{N},\ \phi(\sigma(1))\,\frac{M}{N})}$ and note that
\[
\int_0^1 w_N(\xi,\zeta)\,\psi(\zeta)\,d\zeta=\frac{N}{M}\,\mathbb{I}_{\xi\in [(\sigma(1)-1)\,\frac{M}{N},\ \sigma(1)\,\frac{M}{N})}.
\]
Therefore, we have
\[
\left\|\int_0^1 w_N(\xi,\zeta) \,\psi(\zeta)\,d\zeta\right\|_{L^1}=1.
\]
However, using that $w_{N'}\in L^\infty_{\xi,\zeta}$ (and also any rearranged version) we also have
\[
\left\|\int w_{N'}(\Phi(\xi),\Phi(\zeta)) \,\psi(\zeta)\,d\zeta\right\|_{L^1}\leq \|w_{N'}\|_{L^\infty}\,\|\psi\|_{L^1}=\frac{N'}{M'}\frac{M}{N},
\]
for any measure-preserving map $\Phi:[0,\ 1]\longrightarrow [0,\ 1]$. Altogether implies the control on the cut distance
\[
\inf_{\Phi} \|w_N(\xi,\zeta)-w_{N'}(\Phi(\xi),\Phi(\zeta))\|_{L^\infty\to L^1}\geq 1-\frac{N'}{M'}\,\frac{M}{N} \, ,
\]
where the infimimum ranges over all measure-preserving maps of $[0,\ 1]$. Of course, since $\frac{N}{M}\gg \frac{N'}{M'}$, the above lower bound that does not converge to $0$ as $N'\to \infty$. Because $w_N$ is purely deterministic, there is also no renormalization that would preserve the dynamics and provide convergence in the cut distance.

\smallskip

This expression of $w_N$ of course amounts to a discretization of 
\[
\bar w_N(\xi,\zeta)=\delta_{\phi_N(\xi)}(\zeta),
\]
where we observe that $\phi_N$ is a {\it a.e.} defined measure-preserving map on $[0,\ 1]$. At this point, it is tempting to conjecture that, after re-indexing such that $\phi_N$ is compact and extracting a subsequence such that $\phi_N\to\bar\phi$, the limiting dynamics should follow the limiting kernel $\bar w(\xi,\zeta)=\delta_{\bar\phi(\xi)}(\zeta)$. But this is not correct {\it per se}. Instead, in such an example, it is not possible to determine the limiting kernel without some further information on the laws $f_i^0(x)$ of the initial positions $X_i^0$. To keep this example reasonably simple, we assume that we only have at most $N/M$ different laws $f_k$ with $k=1,\ldots,\frac{N}{M}$, and that the $X_i^0$ are actually {\it i.i.d.} for $i\in E_k^N$ according to the law $f_k(x)$. 

To determine the limiting kernel, our analysis requires to find the limit of every observable defined in the previous subsection. Let us hence write explicitly some of them. First we perform the obvious graphon extension of the initial law defined by
\[
f_N(x,\xi)=\sum_{k=1}^{N/M} f_k(x)\,\mathbb{I}_{\xi\in [(\sigma(k)-1)\,\frac{M}{N},\ \sigma(k)\,\frac{M}{N})} \, .
\]
The first observable, corresponding to the unique tree $T_1\in \Tree_1$ with just one vertex is simply
\[
\tau(T_1,(w_{ij})_{1\leq i,j\leq N},(f_i^0)_{1\leq i\leq N})=\int_0^1 f_N(x,\xi)\,d\xi=\frac{M}{N}\,\sum_{k=1}^{N/M} f_k(x).
\]
One can readily check that the observable for the unique tree $T\in \Tree_2$ with two vertices is
\[
\tau(T_2,(w_{ij})_{1\leq i,j\leq N},(f_i^0)_{1\leq i\leq N})=\int_0^1 f_N(x_1,\xi)\,f_N(x_2,\phi_N(\xi))\,d\xi.
\]
Other observables are built in a similar manner. We provide below some more examples. For instance, the tree $T_{k+1}\in \Tree_{k+1}$ with one root and $k$
leaves yields
\[
\tau(T_{k+1},(w_{ij})_{1\leq i,j\leq N},(f_i^0)_{1\leq i\leq N})=\int_0^1 f_N(x_1,\xi)\,\displaystyle \prod_{i=1}^k f_N(x_{i+1},\phi_N(\xi))\,d\xi.
\]
The tree $T_{2k+1}\in \Tree_{2k+1}$ with one root connected to $k$ vertices, each of them with one leaf, instead produces the observable
\[
\tau(T_{2k+1},(w_{ij})_{1\leq i,j\leq N},(f_i^0)_{1\leq i\leq N})=\int_0^1 f_N(x_1,\xi)\,\prod_{i=1}^k f_N(x_{2i},\phi_N(\xi))\,f_N(x_{2i+1},\phi_N(\phi_N(\xi))\,d\xi.
\]
As one can readily see, a sufficient condition to obtain $\delta_{\bar\phi(\xi)}(\zeta)$ as a limiting kernel is to have both $\phi_N$ converges strongly to $\bar\phi$ and $f_N$ converging strongly to some $\bar f$. We are allowed to reindex (or use a measure-preserving map) but it has to be the same map for both. We will in fact obtain later a lemma that guarantees that this is possible (and in fact up to a countable number of functions). However, we stress again that the limiting kernel is not unique in general. Hence, even when $\phi_N$ and $f_N$ are jointly compact, we may have other, different, and simpler acceptable limiting kernels. A trivial example was mentioned at the beginning of this subsection: if $f_N$ does not depend on $\xi$ then we may take $\bar w=1$.

We expect that the simplest limiting kernel depends on the decomposition in cycles of $\phi_N$ and how $f_N$ behaves on each cycle. Performing a full analysis is well beyond the scope of this example though. For this reason, let us further simplify things by assuming that $\phi$ has exactly one cycle of full length with simply: $\phi(k)=k+1$ for $k<\frac{N}{M}$, with $\phi(\frac{N}{M})=1$. Furthermore, let us assume that $\sigma=Id$, that dimension $d=1$ and $f_k(x)=f(x-k\,\frac{M}{N})$ for some smooth $f$. With these choices, one can check that
\[
f_N(x,\xi)=f(x-[\xi]_N),\quad \phi_N(\xi)= [\xi]_N\,
\]
where $[\xi]_N=\frac{M}{N}\,\lceil \frac{N}{M}\xi\rceil$ and $\lceil \cdot\rceil$ is the ceiling function. Since $f$ is smooth, we have that
\[
f_N(x,\xi)=f(x-\xi)+O\left(\frac{M}{N}\right), \quad \phi_N(\xi)=\xi+O\left(\frac{M}{N}\right).
\]
Hence, it is obvious that the limiting dynamics can be represented by choosing $\bar w=\delta(\xi-\zeta)$ and $\bar f(x,\xi)=f(x-\xi)$. But of course, even here, the representation is not unique: any composition by a measure-preserving map would still work.

However, we can prove in that special case that the limiting dynamics cannot be represented by some $\bar w\in L^\infty_{\xi,\zeta}$. Assume that there is one such $\bar w$ and $\bar f$ such that $\tau(T,\bar w,\bar f)=\bar \tau(T)$ for all tree $T$. To be specific,  
assume that $\bar w$
obeys the mean-field scaling, that is,
\[
\sup_{\xi\in [0,\ 1]} \int_0^1 |\bar w(\xi,\zeta)|\,d\zeta\leq 1,\quad \sup_{\zeta\in [0,\ 1]}\int_0^1 |\bar w(\xi,\zeta)|\,d\xi\leq 1.
\]
Assume moreover that the function
\[
x\in \mathbb{R}\longmapsto \int_0^1 |f(x-\xi)|^2\,d\xi,
\]
does not take only a finitely many values. We can easily calculate the limiting observables for any tree $T$
\[
\tau(T,w_N,f_N)\to \bar\tau(T):=\int_0^1 \prod_{i=1}^{|T|} f(x_i-\xi)\,d\xi.
\]
This is of course a very remarkable formula as it depends only on $|T|$. 

Consider the tree $T_3\in \Tree_3$ with two leaves attached to its root. Since we have $\int_{\mathbb{R}} \bar f(x_1,\xi_1)\,dx=1$ for all $\xi_1\in [0,\ 1]$, then we can calculate
\[
\begin{split}
\int_{\mathbb{R}^2} \tau(T_3,\bar w,\bar f)(x_1,x,x)\,dx_1\,dx&= \int_{\mathbb{R}} \int_{[0,\ 1]^3} \bar w(\xi_1,\xi_2)\,\bar w(\xi_1,\xi_3)\,\bar f(x,\xi_2)\,\bar f(x,\xi_3)\,d\xi_1\,d\xi_2\,d\xi_3\,dx\\
&=\int_{\mathbb{R}}\int_0^1 |W(\bar f)(x,\xi)|^2\,dx\,d\xi,
\end{split}
\]
where $W:L^2(\mathbb{R}\times [0,\ 1])\longrightarrow L^2(\mathbb{R}\times [0,\ 1])$ is the linear operator defined by
\[
W(g)(x,\xi):=\int_0^1 \bar w(\xi,\zeta)\,g(x,\zeta)\,d\zeta.
\]
Consider next the tree $T_4\in \Tree_4$ where we add one leaf to $T_3$ connected to one leaf (say node $2$) of $T_3$ (which is then no more a leaf of course). Calculate again,
\[
\begin{split}
\int_{\mathbb{R}^3} \tau(T_4,\bar w,\bar f)(x_1,x_2,x,x)&\,dx_1\,dx_2\,dx\\
&= \int_{\mathbb{R}} \int_{[0,\ 1]^4} \bar w(\xi_1,\xi_2)\,\bar w(\xi_2,\xi_4)\,\bar w(\xi_1,\xi_3)\,\bar f(x,\xi_3)\,\bar f(x,\xi_4)\,d\xi_1\,d\xi_2\,d\xi_3\,d\xi_4\,dx\\
&=\int_{\mathbb{R}}\int_0^1 W(\bar f)(x,\xi)\,W^2(\bar f)(x,\xi)\,dx\,d\xi.
\end{split}
\]
Now observe that by the above shape $\tau(T,\bar w,\bar f)=\bar\tau(T)$ of the limiting observables, we have
\[
\int_{\mathbb{R}^3} \tau(T_4,\bar w,\bar f)(x_1,x_2,x,x)\,dx_1\,dx_2\,dx=\int_{\mathbb{R}} \bar\tau(T_3)(x_1,x,x)\,dx_1\,dx,
\]
that is,
\[
\int_{\mathbb{R}}\int_0^1 W(\bar f)(x,\xi)\,W^2(\bar f)(x,\xi)\,dx\,d\xi=\int_{\mathbb{R}}\int_0^1 |W(\bar f)(x,\xi)|^2\,dx\,d\xi.
\]
Note that thanks to the mean-field scaling on $\bar w$, we have {
\[
\begin{split}
&\int_{\mathbb{R}}\int_0^1 |W^2(\bar f)(x,\xi)|^2\,dx\,d\xi\leq \int_{\mathbb{R}}\int_0^1 |W(\bar f)(x,\xi)|^2\,dx\,d\xi,
\end{split}
\]
since we have that by Cauchy-Schwartz
\[
\begin{split}
|W^2(\bar f)(x,\xi)|^2&=\left(\int_0^1 \bar w(\xi,\zeta)\,W(\bar f)(x,\xi)\,d\zeta\right)^2\leq \int_0^1 \bar w(\xi,\zeta)\,d\zeta\;\int_0^1 \bar w(\xi,\zeta)\,|W(\bar f)(x,\zeta)|^2\,d\zeta\\
&\leq \int_0^1 \bar w(\xi,\zeta)\,|W(\bar f)(x,\zeta)|^2\,d\zeta,
\end{split}
\]
from the mean-field scaling. Therefore as claimed,
\[
\begin{split}
&\int_{\mathbb{R}}\int_0^1 |W^2(\bar f)(x,\xi)|^2\,dx\,d\xi\leq \int_{\mathbb{R}}\int_0^1\int_0^1 \bar w(\xi,\zeta)\,|W(\bar f)(x,\zeta)|^2\,dx\,d\xi\,d\zeta\leq \int_{\mathbb{R}}\int_0^1 |W(\bar f)(x,\zeta)|^2\,dx\,d\zeta,
\end{split}
\]
still from the mean-field scaling. }

This equality, together with another Cauchy-Schwartz inequality, lead to 
\begin{align*}
\int_{\mathbb{R}}\int_0^1 W(\bar f)(x,\xi)&\,W^2(\bar f)(x,\xi)\,dx\,d\xi\\
&\leq \Vert W(\bar f)\Vert_{L^2}\Vert W^2(\bar f)\Vert_{L^2}\leq \Vert W(\bar f)\Vert_{L^2}^2=\int_\mathbb{R}\int_0^1 |W(\bar f)(x,\xi)|^2\,dx\,d\xi.
\end{align*}
The identity above corresponds to equality in the Cauchy-Schwartz inequality, which implies that $W(\bar f)$ and $W^2(\bar f)$ must be linearly dependent: there must exist $\lambda\in \mathbb{R}$ such that $W^2(\bar f)=\lambda\,W(\bar f)$. Furthermore the last inequality above also implies that $\lambda=1$. 

Once this is proved, we can keep adding leaves and with the same method, we can prove that
\[
W(\bar f)^k(x,\xi)=W(W(\bar f)^k)(x,\xi)=\int_0^1 \bar w(\xi,\zeta)\,W(\bar f)^k(x,\zeta)\,d\zeta.
\]
The density of polynomials in continuous functions shows that in fact
\[
\chi(W(\bar f)(x,\xi))=W(\chi(W(\bar f)))(x,\xi)=\int_0^1 \bar w(\xi,\zeta)\,\chi(W(\bar f)(x,\zeta))\,d\zeta,
\]
for any continuous function $\chi$. 

This unique fact implies an extremely simple structure on $W(\bar f)$ and $\bar w$ if $\bar w\in L^\infty_{\xi,\zeta}$. It is straightforward to show that {$W(\bar f)$ can only take a finite number of values}: there exists level sets $L_k$ such that
\[
W(\bar f)(x,\xi)=\sum_{k=1}^K \alpha_k\,\mathbb{I}_{(x,\xi)\in L_k}.
\]
Moreover $\bar w$ needs to be decomposed along, because we must have for all $x$ and $\xi$ that
\begin{equation}
\mathbb{I}_{(x,\xi)\in L_k}=\int_0^1 \bar w(\xi,\zeta)\,\mathbb{I}_{(x,\zeta)\in L_k}\,d\zeta.\label{identifyw1}
\end{equation}
{Indeed, assume by contradiction that $W(\bar f)$ takes infinitely many values. By taking appropriate functions $\chi$, this implies that we have an infinite number (at least countable) of sets $L_k$ such that for every $k$
\[
\mathbb{I}_{(x,\xi)\in L_k}=\int_0^1 \bar w(\xi,\zeta)\,\mathbb{I}_{(x,\zeta)\in L_k}\,d\zeta.
\]
However for any fixed $x$, we can only have a limited number of $L_k$ such that there exists $\xi$ with $(x,\xi)\in L_k$. Denote by $\bar L_k(x)$ the corresponding section:  $\bar L_k(x)=\{\xi,\ (x,\xi)\in L_k\}$. Observe that
\[
\int_0^1 \bar w(\xi,\zeta)\,\mathbb{I}_{(x,\zeta)\in L_k}\,d\zeta\leq \|\bar w\|_{L^\infty}\,|\bar L_k(x)|.
\]
Therefore if $|\bar L_k(x)|>0$, then we need to have that $|\bar L_k(x)|\geq \|\bar w\|_{L^\infty}^{-1}$ so that at most $\|\bar w\|_{L^\infty}^{-1}$ number of $\bar L_k$ sets can be of positive measure. Since we are on a compact support, a straightforward covering lemma then shows that we cannot have an infinite number of $L_k$ sets.
}

Observe that, in addition, we have that
\[
\int_{\mathbb{R}}\int_0^1 W^*(W(\bar f))(x,\xi)\,W(\bar f)(x,\xi)\,dx\,d\xi=\int_{\mathbb{R}}\int_0^1 W(\bar f)(x,\xi)\,W^2(\bar f)(x,\xi)\,dx\,d\xi,
\]
with $W^*$ defined by the formal adjoint
\[
W^*(g)(x,\xi):=\int_0^1 \bar w(\zeta,\xi)\,g(x,\zeta)\,d\zeta.
\]
We hence have from the argument above that
\[
\int_{\mathbb{R}}\int_0^1 W^*(W(\bar f))(x,\xi)\,W(\bar f)(x,\xi)\,dx\,d\xi=\int_{\mathbb{R}}\int_0^1 |W(\bar f)(x,\xi)|^2\,dx\,d\xi.
\]
Again the same mean-field scaling implies that 
\[
\int_{\mathbb{R}}\int_0^1 |W^*(W(\bar f))(x,\xi)|^2\,dx\,d\xi\leq \int_{\mathbb{R}}\int_0^1 |W(\bar f)(x,\xi)|^2\,dx\,d\xi,
\]
so we again have a case of equality in Cauchy-Schwartz and that proves that $W^*(W(\bar f))=W(\bar f)$,
or
\begin{equation}
\mathbb{I}_{(x,\xi)\in L_k}=\int_0^1 \bar w(\zeta,\xi)\,\mathbb{I}_{(x,\zeta)\in L_k}\,d\zeta.\label{identifyw2}
\end{equation}
Of course, this shows that the mean-field scaling also has to be an equality
\[
\int_0^1 \bar w(\xi,\zeta)\,d\zeta=1,\quad \int_0^1 \bar w(\xi,\zeta)\,d\xi=1.
\]
For any $x$, we recall that $\bar L_k(x)$ is the corresponding section:  $\bar L_k(x)=\{\xi,\ (x,\xi)\in L_k\}$. The identities~\eqref{identifyw1}-\eqref{identifyw2} imply that if $\xi\in \bar L_k(x)$ then $\mbox{supp}\,\bar w(\xi,\cdot)\subset \bar L_k(x)$ and reciprocally $\mbox{supp}\,\bar w(.,\zeta)\subset \bar L_k(x)$ if $\zeta\in \bar L_k(x)$. 

This has some surprising consequences on the structure of the $\bar L_k(x)$. For example, if $M=\bar L_l(y)\setminus \bar L_k(x)$, then for each $\xi\in M$, since $\xi\in \bar L_l(y)$ but $\xi\not \in \bar L_k(x)$
\[
\int_0^1 \bar w(\xi,\zeta)\,\mathbb{I}_{\zeta\in M}\,d\zeta=\int_0^1 \bar w(\xi,\zeta)\,\mathbb{I}_{\zeta\in \bar L_l(y)}\,d\zeta-\int_0^1 \bar w(\xi,\zeta)\,\mathbb{I}_{\zeta\in \bar L-l(y)\cap \bar L_k(x)}\,d\zeta=1-0=1.
\]
Because $\bar w\in L^\infty_{\xi,\zeta}$, a similar calculation as for proving that there is a finite number of level sets  implies that, if $\bar L_l(y)\cap \bar L_k(x)\neq \emptyset$ 
\[
|\bar L_l(y)\setminus \bar L_k(x)|\geq \frac{1}{\|\bar w\|_{L^\infty}},
\]
and of course, individually, we have also for any $x$ and $k$ that
\[
|\bar L_k(x)|\geq \frac{1}{\|\bar w\|_{L^\infty}}.
\]
Because the measure of $[0,\ 1]$ is finite, we can deduce that we can only have a finite number of different $\bar L_k(x)$: there exists $M_1,\dots, M_L\subset [0,\ 1]$ such that for any $x$ and $k$, there exists $l$ with $\bar L_k(x)=M_l$. This is a variant of the famous Vitali's covering lemma.

Note that for example
\[
\int_0^1 |W(\bar f)(x,\xi)|^2\,d\xi=\sum_{k=1}^K \alpha_k^2 |\bar L_k(x)|.
\]
Since $\bar L_k(x)$ has only a finite number of possibilities, this finally implies that the integral above can only take a finite number of values. But this integral is derived from $\tau(T_3,\bar w,\bar f)$ and it has to be equal to the corresponding quantity derived from $\bar \tau(T_3)$ which is
\[
\int_0^1 |f(x-\xi)|^2\,d\xi.
\]
This provides the contradiction since we assumed that the above quantity takes an infinite number of values as a function in $x$.

\medskip

This example shows that it can be very complicated to determine what are the possible limiting kernels even knowing explicitly the limiting observables (and even with simple formulas such as here). One major issue is that we have much less observables than for classical graphon theory where one can work with homomorphism densities~$\tau(G,w_N)$ for any graph~$G$ and not only trees. 

This means that  several additional transforms can leave the observables invariant. As a matter of fact, one can for example change $\bar w(\xi,\zeta)$ into $\bar w(\Phi(\xi),\Phi(\zeta))$ and $\bar f(x,\xi)$ into $\bar f(x,\Phi(\xi))\,J(\xi)$, where $J$ is the Jacobian of the transform $\Phi$. In particular, we are not necessarily limited here to measure-preserving transforms $\Phi$. This can lead to all sort of complicated phenomena, where it is sometimes possible to "split" Dirac masses for instance. It is one of the reasons why the analysis above is so intricate in spite of the simplicity of the choices.

        \subsection{Sketch of the proof}
        We briefly explain here some of the main ideas and steps that are used in the proof of the main Theorem \ref{maintheorem}.

		\medskip
        
        $\diamond$ {\em Step~1: Propagation of independence.} 
        
        This is done in Section~\ref{sec:independence} by introducing the independent system of coupled PDEs~\eqref{independ}, restated here
          \begin{equation}
          \left\{
\begin{array}{l}
\displaystyle \partial_t \bar f_i+\divop_x\left(\bar f_i(t,x)\,\sum_{j=1}^N w_{ij}\,\int_{\mathbb{R}^d} K(x-y)\,\bar f_j(t,dy)\right)=0,\\
\displaystyle \bar f_i(t=0)=f_i^0,
\end{array}\label{repeatindepend}
\right.
\end{equation}
and it proves ({\it cf.} Proposition~\ref{propindependence}) that $\bar f_i(t,x)$ are correct approximations for the law of $X_i(t)$ under the assumption \eqref{vanishingweights}. As mentioned earlier, this relies on a straightforward extension to the classical trajectorial methods for mean-field limits, see \cite{Sznit}.

For general connectivities $w_{ij}$, this is only the very beginning as it does not allow directly obtaining the limit of the corresponding $1$-particle distribution given by
\[
\frac{1}{N}\,\sum_{i=1}^N \bar f_i(t,x).
\]
The one exception concerns the special case where $\sum_{j=1}^N w_{ij}=\bar w$ is independent of $i$. If moreover all $f_i^0$ are equal to $\bar f^0$, then it can be seen that $\bar f_i=\bar f$ solves~\eqref{repeatindepend} provided that $\bar f$ is a solution to the classical mean-field limit.

\medskip

$\diamond$ {\em Step~2: Introducing extended graphons.} 

The next step in our proof is to define our extended empirical graphons for a fixed $N$, which is still done through the equivalent of formula~\eqref{graphon}, namely
  \[
  \begin{split}
    &  w_N(\xi,\zeta)=\sum_{i,j=1}^N Nw_{ij}\, \mathbb{I}_{[\frac{i-1}{N},\frac{i}{N})}(\xi)\,\mathbb{I}_{[\frac{j-1}{N},\frac{j}{N})}(\zeta),\\
    & f_N(t,x,\xi)=\sum_{i=1}^N \bar f_i(t,x)\,\mathbb{I}_{[\frac{i-1}{N}\frac{i}{N})}(\xi).
    \end{split}
  \]
 It is straightforward to check that, since the $\bar f_i$ solve~\eqref{repeatindepend}, then $w_N$ and $f_N$ solve the limiting equation~\eqref{vlasov-eq-lim}, namely we have
  \[
  \partial_t f_N(t,x,\xi)+\divop_x\,\left(f_N(t,x,\xi) \,\int_0^1 w_N(\xi,\zeta)\int_{\R^d} K(x-y)\, f_N(t,y,\zeta)\,dy\,d\zeta\right)=0.
  \]
  This is performed at the beginning of Section~\ref{sec:hierarchy} together with the study of some basic properties of our extended graphons under the only scaling $w_N\in L^\infty_\xi \mathcal{M}_\zeta\cap L^\infty_\zeta \mathcal{M}_\xi$, which include the aforementioned
  \[
\left\|\int_0^1 \phi(\zeta)\,w(\xi,d\zeta)\right\|_{L^\infty_\xi}\leq \|w\|_{L^\infty_\xi \mathcal{M}_\zeta\cap L^\infty_\zeta \mathcal{M}_\xi}\,\|\phi\|_{L^\infty}.
\]
For a fixed $w\in L^\infty_\xi \mathcal{M}_\zeta\cap L^\infty_\zeta \mathcal{M}_\xi$, this allows for example obtaining the existence and uniqueness of a solution $f\in L^\infty([0,\ T]\times[0,\ 1],\;W^{1,\infty}(\R^d))$ to~\eqref{vlasov-eq-lim}, see Proposition~\ref{existenceweak}. However, we are unfortunately incapable of passing to the limit directly in the above equation for $f_N$ to derive~\eqref{vlasov-eq-lim}. The main obstruction is the lack of compactness in $\xi$ on $f_N$ combined with the very weak topology in $L^\infty_\xi \mathcal{M}_\zeta\cap L^\infty_\zeta \mathcal{M}_\xi$.

\medskip

$\diamond$ {\em Step~3: Introducing and studying the new observables.} 

The latter difficulty is what leads to the introduction of the observables, that can be defined at any time $t$ through the formula~\eqref{observabletree2} that we repeat here,
  \[
\tau(T,w_N,f_N)(t,x_1,\ldots,x_{|T|})=\int_{[0,\ 1]^{|T|}} \prod_{(k,l)\in E(T)} w_N(\xi_k, \xi_l)\;\prod_{m\in V(T)}f_N(t,x_{m},\xi_m)\,d\xi_1\dots d\xi_{|T|}.
  \]
  Those observables include the $1$-particle distribution which simply corresponds to choosing the tree $T=T_1$ with only one vertex:
  \[
\tau(T_1,w_N,f_N)(t,x)=\frac{1}{N}\,\sum_{i=1}^N \bar f_i(t,x). 
\]
The fact that the $\tau(T,w_N,f_N)$ are well defined follows from the basic analysis of our extended graphons in the previous point. 

A critical point is that those observables solve an independent hierarchy of equations, which reads
\[
\partial_t \tau(T,w_N,f_N)+\sum_{i=1}^{|T|}\mbox{div}_{x_i}\left(\int_{\R^d} K(x_i-z)\,\tau(T+i,w_N,f_N)(t,x_1,\ldots,x_{|T|},z)\,dz\right)=0,
\]
where $T+i$ denotes the tree obtained from $T$ by adding a leaf on the $i$-th vertex.

It turns out that we can pass to the limit in the above hierarchy and even obtain the uniqueness of solutions upon appropriate assumptions on the $\tau(T,w_N,f_N)$ as per Theorem~\ref{stabilitynonviscous}. The analysis is performed in the second part of Section~\ref{sec:hierarchy} and is one of the major contribution of the paper.

Independently to the proof in this paper, we mention the very recent~\cite{Lacker} that obtains uniqueness on the classical exchangeable BBGKY or Vlasov hierarchies. The basic idea in that proof is reminiscent of the strategy that we develop here for Theorem~\ref{stabilitynonviscous}, even if~\cite{Lacker} uses relative entropies while we base our estimates on $L^2$ bounds.

\medskip

$\diamond$ {\em Step~4: Identifying the limit.}

  Theorem~\ref{stabilitynonviscous} allows passing to the limit in all $\tau(T,w_N,f_N)$, after extracting a sub-sequence, to some hierarchy of $h_T$. It remains to identify the $h_T$ by finding $w$ and $f^0$ such that if $f$ solves the limiting equation~\eqref{vlasov-eq-lim} with $w$ and $f^0$ as initial data then we have the representation
  \[
\tau(T,w,f)=h_T.
  \]
  This is done in Section~\ref{sec:graphons} in Theorem~\ref{lg} and relies on a modified regularity lemma for graphons stated in Lemma.~\ref{kl}. This also requires the definition of the $\tau(T,w,f)$  in general with only~$f\in L^\infty([0,\ T]\times[0,\ 1],\;W^{1,\infty}(\R^d))$ and~$w\in L^\infty_\xi \mathcal{M}_\zeta\cap L^\infty_\zeta \mathcal{M}_\xi$, which is based on an abstract algebra that encompasses the sequence of operations on the tree that are performed to produce $\tau(T,w,f)$. Putting all estimates together finally allows to conclude.

  As one can see, this does not prove the convergence of $f_N$ to $f$ in any direct sense. Instead it only shows the convergence of all $\tau(T,w_N,f_N)$ to the $\tau(T,w,f)$. It implies the claimed result by taking $T=T_1$ the trivial tree with only one node and no edge, which in particular yields
  \[
\int_0^1 f_N(t,x,\xi)\,d\xi=\frac{1}{N}\,\sum_{i=1}^N \bar f_i(t,x)\to \int_0^1 f(t,x,\xi)\,d\xi.
  \]

\section{A generalized McKean SDE and propagation of independence}\label{sec:independence}
The aim of this section is to show that independence of events is propagated in our system \eqref{eq1} in an appropriate sense. It will be the equivalent of the aforementioned classical concept of propagation of chaos (for uniform weights $w_{ij}=\frac{1}{N}$), but it is not and it requires a careful extension of the usual arguments. To that end, we shall propose an extension of the classical \cite{Sznit} McKean SDE associated with the particles system   to our current case with non-uniform weights.
\begin{lem}[McKean SDE]\label{lemindependence}
Consider the nonlinear system of SDE for $(\bar X_1,\ldots,\bar X_N)$ given by
\begin{equation}\label{eq1-independ}
\left\{\begin{array}{l}
\displaystyle \frac{d\bar X_i}{dt}=\sum_{j=1}^N w_{ij}\int_{\mathbb{R}^d}K(\bar X_i-y)\bar f_j(t,dy),\\
\displaystyle \bar X_i(t=0)=\bar X_i^0,
\end{array}
\right.
\end{equation}
where $K\in W^{1,\infty}$ and we denote $\bar f_i(t,\cdot)=\Law(\bar X_i(t))$. Then, for any random initial data $(\bar X_1^0,\ldots,\bar X_N^0)$ such that $\mathbb{E}[\vert \bar X_i^0\vert]<\infty$ there is existence and uniqueness, trajectorial and in law of solutions of \eqref{eq1-independ}. In addition, if $\bar X_i^0$ are independent then $\bar X_{i}(t)$ are also independent for each $t\in \mathbb{R}_+$.
\end{lem}

The proof is an elementary extension of the classical case with uniform weights, see \cite[Theorem 1.1]{Sznit}.  We notice that the independence of initial data is crucial in the previous result. However, $\bar X_i^0$ are not necessarily identically distributed, although we could have also assumed so. Unfortunately, notice that the presence of non-uniform weights makes particles non-exchangeable and there is a rupture of symmetry. Thus, the necessity for the use of the different laws $\bar f_j$ in the SDE is justified. This was not the case for uniform weights, where symmetry is propagated and for each $\bar X_i$ the McKean SDE can be closed in terms of its law $\bar f_i$ itself.

\begin{pro}[Propagation of independence]\label{propindependence}
Let $(X_1,\ldots,X_N)$ be solution to \eqref{eq1} with $K\in W^{1,\infty}$ and consider the associated laws $f_i(t,\cdot)=\Law (X_i(t))$. Assume that $X_i^0$ are independent random variables such that $\mathbb{E}[\vert X_i^0\vert^2]<\infty$, and the following uniform estimates hold
$$
\sup_{1\leq i\leq N} \sqrt{\E |X_i^0|^2}\leq M,\qquad \sup_{1\leq i\leq N} \sum_{j=1}^N |w_{ij}|\leq C,
$$
for every $N\in \mathbb{N}$ and appropriate $M,C\in \mathbb{R}_+$. Consider the solution $\bar f_i$ to the coupled PDE system
\begin{equation} \label{independ}
\left\{
\begin{array}{l}
\displaystyle \partial_t \bar f_i+\divop_x\left(\bar f_i(t,x)\,\sum_{j=1}^N w_{ij}\,\int_{\mathbb{R}^d} K(x-y)\,\bar f_j(t,dy)\right)=0,\\
\displaystyle \bar f_i(0,x)=f_i(0,x).
\end{array}
\right.
\end{equation}
Then, the following estimate holds
\begin{equation} \label{independ-error}
\sup_{1\leq i\leq N}W_1(f_i(t,\cdot),\bar f_i(t,\cdot))\leq C_1(t)\,\sup_{1\leq i,j\leq N} |w_{ij}|^{1/2},
\end{equation}
for every $t\in \mathbb{R}_+$, where $W_1$ is the $1$-Wasserstein distance. We also have a direct control on the empirical measure of the system, namely
\begin{equation}\label{independ-error-2}
\E\,W_1\left(\frac{1}{N}\,\sum_{i=1}^N \delta_{X_i(t)},\; \frac{1}{N}\,\sum_{i=1}^N \bar f_i(t,\cdot)\right)\leq \frac{\tilde C\,C_2(t)}{N^\theta}+C_1(t)\,\sup_{1\leq i,j\leq N} |w_{ij}|^{1/2},
\end{equation}
for every $t\in \mathbb{R}_+$ and  appropriate constants $\tilde C,\theta>0$ depending only on the dimension $d$. In addition,  $C_1(t)$ and $C_2(t)$ depend on $M,C,\Vert K\Vert_{W^{1,\infty}}$ and $t$, and they can be made explicit by
\begin{align*}
C_1(t)&:=\sqrt{\frac{2}{C}}\,(e^{2C\,t\,\|K\|_{W^{1,\infty}}}-1),\\
C_2(t)&:=(2M^2+2C^2\Vert K\Vert_{L^\infty}t^2)^{1/2}.
\end{align*}
\end{pro}

We notice that if weights relax uniformly to zero, {\em i.e.}, $\lim_{N\rightarrow \infty}\sup_{1\leq i,j,\leq N}\vert w_{ij}\vert^{1/2}=0$, then the above estimate \eqref{independ-error} guarantees that the creation of correlation between particles of system \eqref{eq1} is weak and relax to zero as $N\rightarrow \infty$. Indeed,  the relaxation rate extends the usual $N^{-1/2}$ of the classical propagation of chaos for uniform weights $w_{ij}=\frac{1}{N}$, and the laws $f_i$ of each particle is approximated by the independent laws $\bar f_i$ in the McKean SDE \eqref{eq1-independ} in the mean field limit $N\rightarrow \infty$, see \cite[Theorem 1.4]{Sznit}.

\begin{proof}
  Let $(\bar X_1,\ldots,\bar X_N)$ be the unique solution to \eqref{eq1-independ} with initial data $\bar X_i^0=X_i^0$, according to Lemma \ref{lemindependence}. We note first that we can simply propagate the second moments,
  \[
\E |X_i|^2\leq \E (|X_i^0|+C\,\|K\|_{L^\infty}\,t)^2\leq C_2(t)^2.
  \]
  In addition, let us denote the laws $\bar f_i(t,\cdot)=\Law (\bar X_i(t))$ and consider the sub-$\sigma$-algebras generated by $\bar X_i$, {\em i.e.}, $\mathcal{F}_i(t)=\sigma(\bar X_i(t))$, for each $t\in \mathbb{R}_+$ and $i=1,\ldots,N$. Since the variables $\bar X_i$ and $\bar X_j$ are independent by Lemma \ref{lemindependence}, for any $j=1,\ldots,N$ with $j\neq i$, then we have
$$
\int_{\mathbb{R}^d}K(\bar X_i-y)\,\bar f_j(t,dy)=\mathbb{E}_i[K(\bar X_i-\bar X_j)],
$$
where $\mathbb{E}_i=\mathbb{E}_i[\,\cdot\,\vert \mathcal{F}_i(t)]$ denotes the conditional expectation given $\mathcal{F}_i(t)$. Therefore, taking the difference in equations \eqref{eq1} and \eqref{eq1-independ} we obtain the SDE
\begin{align*}
\frac{d}{dt}(X_i-\bar X_i)&=\sum_{j=1}^N w_{ij}\,(K(X_i-X_j)-\mathbb{E}_i[K(\bar X_i-\bar X_j)])\\
&=\sum_{j=1}^N w_{ij}\,(K(X_i-X_j)-K(\bar X_i-\bar X_j)) +\sum_{j=1}^N w_{ij}\,(K(\bar X_i-\bar X_j)-\mathbb{E}_i[K(\bar X_i-\bar X_j)]),
\end{align*}
for every $t\in \mathbb{R}_+$. Hence, taking total expectation and using the Lipschitz continuity of $K$ we have that
\begin{equation}\label{independ-error-pre}
\frac{d}{dt}\mathbb{E}\,\vert X_i-\bar X_i\vert\leq 2C[K]_{\Lip}\sup_{1\leq j\leq N}\mathbb{E}\,\vert X_j-\bar X_j\vert+\mathbb{E}\left[\left\vert\sum_{j=1}^N w_{ij}\,(K(\bar X_i-\bar X_j)-\mathbb{E}_i[K(\bar X_i-\bar X_j)])\right\vert\right].
\end{equation}
We remark that all the expectations are finite by the assumption on $X_i^0$. Our goal now is to derive a bound from \eqref{independ-error-pre}. Notice that the key point is to control the relaxation of the second term of the right hand side as $N\rightarrow \infty$ in terms of weights $w_{ij}$. Such a step is reminiscent of the classical propagation of chaos with uniform weights $w_{ij}=\frac{1}{N}$, see  \cite[Theorem 1.4]{Sznit}. However, the presence of non-uniform weights makes the argument more subtle. By Jensen's inequality we shall control the second order moment instead. Indeed, expanding the square we obtain that
\begin{multline*}
\mathbb{E}\left[\left\vert \sum_{j=1}^N w_{ij}(K(\bar X_i-\bar X_j)-\mathbb{E}_i[K(\bar X_i-\bar X_j)])\right\vert ^2\right]\\
=\sum_{j,k=1}^N w_{ij}\,w_{ik}\,\mathbb{E}\bigg[(K(\bar X_i-\bar X_j)-\mathbb{E}_i[K(\bar X_i-\bar X_j)])\cdot(K(\bar X_i-\bar X_k)-\mathbb{E}_i[K(\bar X_i-\bar X_k)])\bigg].
\end{multline*}
Since $K(0)=0$, then only the terms with $j\neq i$ and $k\neq i$ persist in the above sum.  We shall prove that all the terms with $j\neq k$ also disappear. This will be a consequence of the statistical independence of $\bar X_i$. Indeed, take $j\neq k$ in the sum and notice that $(\bar X_i,\bar X_j,\bar X_k)$ are independent. Therefore,
\begin{multline*}
\mathbb{E}\bigg[(K(\bar X_i-\bar X_j)-\mathbb{E}_i[K(\bar X_i-\bar X_j)])\cdot (K(\bar X_i-\bar X_k)-\mathbb{E}_i[K(\bar X_i-\bar X_k)])\bigg]\\
=\mathbb{E}\bigg\{\mathbb{E}_i\bigg[(K(\bar X_i-\bar X_j)-\mathbb{E}_i[K(\bar X_i-\bar X_j)])\bigg]\,\mathbb{E}_i\bigg[(K(\bar X_i-\bar X_k)-\mathbb{E}_i[K(\bar X_i-\bar X_k)])\bigg]\bigg\}=0.
\end{multline*}
Here, we have used again the law of total expectation $\mathbb{E}=\mathbb{E}\,\mathbb{E}_i$ and the fact that the random variables $K(\bar X_i-\bar X_j)-\mathbb{E}_i[K(\bar X_i-\bar X_j)]$ and $K(\bar X_i-\bar X_k)-\mathbb{E}_i[K(\bar X_i-\bar X_k)]$ are conditionally independent given $\mathcal{F}_i$. This means that only the terms $j=k$ persist in the sum and
\begin{align*}
\mathbb{E}\left[\left|\sum_{j=1}^N w_{ij}\,(K(\bar X_i-\bar X_j) -\mathbb{E}_i[K(\bar X_i-\bar X_j)])\right|^2\right] 
&=2\,\sum_{j=1}^N w_{ij}^2\,|K(\bar X_i-\bar X_j)-\mathbb{E}_i[K(\bar X_i-\bar X_j)]|^2 \\
&\leq 8\,C\,\|K\|_{L^\infty}^2\,\sup_{1\leq j\leq N} |w_{ij}|,
\end{align*}
for each $i=1,\ldots,N$ and $t\in \mathbb{R}_+$. Finally, putting everything together into \eqref{independ-error-pre}, using Gr\"{o}nwall's inequality and noting that $\mathbb{E}[\vert X_i(0)-\bar X_i(0)\vert]=0$ (by definition), we obtain the following estimate
\[
 \sup_{1\leq i\leq N}\mathbb{E}\,|X_i-\bar X_i|\leq C_1(t)\,\sup_{1\leq i,j\leq N} |w_{ij}|^{1/2},
\]
 for every $t\in \mathbb{R}_+$. 
 
 To conclude the proof, it only remains to check that  $\bar f_i=\Law (\bar X_i)$ are precisely the solution to the PDE system \eqref{independ}, and that the above error estimate for the processes implies the corresponding $W_1$ error bounds \eqref{independ-error} and \eqref{independ-error-2} for the laws. Let us assume the former and let us prove the later first. Since $\bar f_i(t,\cdot)=\Law(\bar X_i(t))$, this directly implies \eqref{independ-error}, since
\begin{align*}
W_1(f_i(t,\cdot),\bar f_i(t,\cdot))&=\sup_{\|\nabla\phi\|_{L^\infty}\leq 1} \int_{\mathbb{R}^d} \phi(x)\,(f_i-\bar f_i)(t,dx)\\
&= \sup_{\|\nabla\phi\|_{L^\infty}\leq 1} \E[\phi(X_i(t))-\phi(\bar X_i(t))]\leq \E\,|X_i(t)-\bar X_i(t)|.
\end{align*}
Similarly, denoting the two empirical measures
\[
\mu_N(t,dx):=\frac{1}{N}\,\sum_{i=1}^N \delta_{X_i(t)}(dx),\quad \bar\mu_N(t,dx):=\frac{1}{N}\,\sum_{i=1}^N \delta_{\bar X_i(t)}(dx),
\]
we have that
\begin{align*}
\E W_1(\mu_N(t,\cdot),\bar\mu_N(t,\cdot))&=\E \left[\sup_{\|\nabla\phi\|_{L^\infty}\leq 1} \frac{1}{N}\,\sum_{i=1}^N (\phi(X_i(t))-\phi(\bar X_i(t)))\right]\\
&\leq \E\left[\frac{1}{N}\,\sum_{i=1}^N |X_i(t)-\bar X_i(t)|\right]\leq \sup_{1\leq i\leq N} \E |X_i(t)-\bar X_i(t)|. 
\end{align*}
Furthermore, by the independence of the $\bar X_i$, we have, through a straightforward extension of the proofs in \cite{DY95} or \cite{Boi11} (see also Theorem~1 in~\cite{FouGui}) which consider the exchangeable case that for some $\tilde C,\theta>0$, 
\[
\E W_1\left(\bar\mu_N(t,\cdot),\;\frac{1}{N}\,\sum_{i=1}^N \bar f_i(t,\cdot) \right)\leq \frac{\tilde C}{N^\theta}\,\sup_{1\leq i\leq N}\E [\vert X_i(t)\vert^2]^{1/2}\leq \frac{\tilde C\,C_2(t)}{N^\theta},
\]
provided for example that the second moments $\E |\bar X_i|^2$ are bounded uniformly in $N$, which we proved earlier. Combining the above estimates and using the triangle inequality allows deducing \eqref{independ-error-2}.

 Finally, for completeness we show that the laws $\bar f_i(t,\cdot)=\Law(\bar X_i(t))$ verify \eqref{independ}. Set any $\varphi\in C^1_c(\mathbb{R}^{2d})$. Then, we have
\[
\frac{d}{dt}\int_{\mathbb{R}^d}\varphi(x_i)\bar f_i(t,x)\,dx=\frac{d}{dt}\mathbb{E}[\varphi(\bar X_i)]=\mathbb{E}\left[\nabla\varphi(\bar X_i)\cdot \frac{d\bar X_i}{dt}\right]=\sum_{j=1}^N w_{ij}\,\mathbb{E}\left[\nabla\varphi(\bar X_i)\cdot \mathbb{E}_i[K(\bar X_i-\bar X_j)]\right].
\]
We notice that $\mathbb{E}_i[K(\bar X_i-\bar X_j)]= \int_{\mathbb{R}^d}K(\bar X_i-y)\,\bar f_j(t,dy)$, by independence of $\bar X_i$ and $\bar X_j$. Then, we have
\begin{align*}
\mathbb{E}[\nabla \varphi(\bar X_i)\cdot \mathbb{E}_i[K(\bar X_i-\bar X_j)]]
&=\mathbb{E}\left[\nabla \varphi(\bar X_i)\cdot \int_{\mathbb{R}^d}K(\bar X_i-y)\,\bar f_j(t,dy)\right]\\
&=\int_{\mathbb{R}^{2d}}\nabla\varphi(x)\cdot K(x-y)\, \bar f_i(t,dx)\,\bar f_j(t,dy).
\end{align*}
Putting everything together and noticing that $\varphi$ is arbitrary, we deduce that $\bar f_i$ verifies \eqref{independ} in weak form, for each $i=1,\ldots,N$.
\end{proof}

The above result is only the beginning to prove our main result because deriving the system \eqref{independ} is not enough to conclude. Indeed, so far we have only replaced a coupled system of ODEs by a couple system of PDEs. We do not even have any compactness with respect to $N$. There is one big exception: 

\begin{rem}[Exchangeable case]\label{R-exchangeable-case}
Assume that $\sum_{j=1}^N w_{ij}=\bar w$ and $\bar f_i^0=\bar f^0$ for every $i=1,\ldots,N$ and some $\bar w\in \mathbb{R}$ and $\bar f^0\in \mathcal{P}(\mathbb{R}^d)$. Then, there is usual propagation of chaos. Specifically, we have that $\bar f_i(t,x)=\bar f(t,x)$ for every $i=1,\ldots,N$, where $\bar f$ solves the usual Vlasov equation:
\[
\partial_t \bar f+\bar w\,\divop_x\left(\bar f(t,x)\,\int_{\R^d} K(x-y)\,\bar f(t,dy)\right)=0.
\]
\end{rem}

\section{A new hierarchy}\label{sec:hierarchy}
To make some of the calculations in this section rigorous, let us introduce a variant of the original system \eqref{independ} for the $\bar f_i(t,\cdot) :=\Law (\bar X_i(t))$ which includes an artificial diffusion term:
\begin{eqnarray} \label{independnu}
\partial_t \bar f_i+\divop_x\left(\bar f_i(t,x)\,\sum_{j=1}^N w_{ij}\,\int_{\R^d} K(x-y)\,\bar f_j(t,dy)\right)=\nu\,\Delta_x \bar f_i.
\end{eqnarray}
 With identical arguments as in Proposition \ref{propindependence} introducing a Wiener process in the McKean process \eqref{eq1-independ} with variance $2\,\nu\, t$ we can deduce  \eqref{independnu}. More specifically, this can be done  starting from
 system \eqref{eq1} with a noise $dW_i$, where the $W_i$ are $N$ independent Wiener processes
\[
\displaystyle dX_i=
  \sum_{j=1}^N w_{ij}\,K(X_i-X_j)\,dt+  \sqrt{2\nu}  \,dW_i.
\]
 Of course \eqref{independ} is just a special case of \eqref{independnu} with $\nu=0$. As in the case without diffusion, in \eqref{independnu} there is an abrupt rupture of symmetry caused by the presence of potentially heterogeneous weights $w_{ij}$. Indeed, even for identical initial data $\bar f_i^0$, the intrinsic dynamics instantaneously yields non-exchangeable distributions $\bar f_i(t)$ for any $t>0$. This prevent us from reducing the system \eqref{independnu} of $N$ coupled PDEs to a single Vlasov-McKean type PDEs in the usual way as in the classical setting with uniform weights. In addition, there is no realistic hope to identify the limit of each individual $\bar f_i$ as $N\rightarrow \infty$. However, we may still be able to study the limit of appropriate averaged statistics. In this section, we shall introduce a natural family of observables that will allow determining a closed averaged description of \eqref{independnu} in terms of an infinite hierarchy of PDEs. For such a novel hierarchy we study uniqueness and stability under the presence of artificial diffusion ($\nu>0$). Although an analogous result is still open for the hierarchy in the absence of artificial diffusion ($\nu=0$), we shall prove that the stability can still be recovered for the full initial system \eqref{independ} under appropriate regularity of the initial data.

\subsection{Hierarchy of observables indexed by trees}
For simplicity of the presentation, in this part we shall restrict to the starting case \eqref{independ} without artificial diffusion ($\nu=0$). Associated with any solution of \eqref{independ} we shall define an infinite family of observables indexed by the set of trees containing an averaged information of $\bar f_i$ and $w_{ij}$. The role of trees is suggested by the construction and will become apparent in a moment. For clarity, we recall some necessary notation of graph theory that will be used throughout the paper.

\begin{defi}[Graphs and trees]\label{D-graphs-trees}
~
\begin{enumerate}[label=(\roman*)]
\item A {\bf (simple) graph} is a pair $G=(V,E)$ where $V$ is a finite set (vertices) and $E\subseteq \{(i,j)\in V\times V:\,i\neq j\}$ is a subset (edges). A graph $G=(V,E)$ is called a {\bf weighted graph} if it is endowed with a (weight) function $W:V\times V\rightarrow \mathbb{R}$ such that $E=\{(i,j)\in V\times V:\, W(i,j)\neq 0\}$. The amount of vertices (or order) will be denoted by $\vert G\vert:=\# V$. 
\item A {\bf tree} is a graph $T=(V,E)$ so that any couple of different vertices $i,j\in V$ are connected by exactly one path (concatenation of edges) and each vertex $i\in V$ cannot be connected with itself.
\item We define the family $\Tree_n$ of {\bf labeled trees} of order $n$ by the following recursive formula
\begin{align*}
\Tree_1&:=\{T_1\},\\
\Tree_{n+1}&:=\{T+i:\,T\in \Tree_n,\,i\in \{1,\ldots,n\}\},\quad n\in \mathbb{Z}_0^+,
\end{align*}
where $T_1$ is the only tree with one vertex $\{1\}$. Here, for any tree $T$ with vertices $\{1,\ldots,n\}$ and any such vertex $i$, we will denote by $T+i$ to the new tree with vertices $\{1,\ldots,n+1\}$ after adding the leaf $n+1$ at node $i$. The family of all labeled trees of arbitrary order is then defined by $\Tree:=\cup_{n=1}^\infty\Tree_n$. Note that, in particular, $\Tree$ contains all possible trees modulo graph isomorphism.
\end{enumerate}
\end{defi}

\begin{rem}[Directed graphs and directed trees.]
Our graphs $G=(V,E)$ are {\em a priori} assumed directed so that $(i,j)\in E$ and $(j,i)\in E$ do not in general represent the same edge. Since we work with labeled trees, those are imbued with a natural orientation starting from the root. \end{rem}

With this notation, we introduce our family of observables $\{\mathcal{O}_N^T:\,T\in \Tree\}$ as follows.

\begin{defi}[Observables]\label{D-observables}
Consider any solution $(\bar f_i)_{i=1,\ldots,N}$ of \eqref{independ} with weights $(w_{ij})_{i,j=1,\ldots,N}$. Then, we define the family of time-dependent Radon measures:
\begin{equation}\label{E-observables}
\mathcal{O}_N^T(t,x_1,\ldots,x_{|T|}):=\frac{1}{N}\,\sum_{i_1,\ldots,i_{|T|}=1}^N \prod_{(k,l)\in E(T)} w_{i_k i_l}\;  \prod_{m\in V(T)} \bar f_{i_m}(t,x_{m}),
\end{equation}
for any $t\geq 0$, every $(x_1,\ldots,x_{|T|})\in \mathbb{R}^{d |T|}$ and each tree $T\in \Tree$.
\end{defi}

Let us note that $\mathcal{O}_N^T(t,\cdot)\in \mathcal{M}(\mathbb{R}^{d |T|})$ are finite Radon measures for each $t\geq 0$. Unfortunately, the heterogeneity of weights still implies that $\mathcal{O}_N^T$ are not symmetric. In the next section we shall show that they indeed satisfy an appropriate hierarchy of PDEs containing an averaged information of \eqref{independ}. Just to illustrate the underlying idea, let us explore some examples. First, consider the only existing tree $T_1\in \Tree_1$ of order $1$ and note that \eqref{E-observables} simply reduces to $\mathcal{O}_N^{T_1}(t,x)=\frac{1}{N}\sum_{i=1}^N \bar f_i(t,x)$. Summing in \eqref{independ} over $i=1,\ldots,N$ yields
\begin{equation}\label{E-observable-1-pre}
\partial_t \mathcal{O}_N^{T_1}(t,x)+\divop_x\left(\int_{\R^d} K(x-y)\,\frac{1}{N}\sum_{i,j=1}^N w_{ij}\,\bar f_i(t,x)\,\bar f_j(t,dy)\right)=0,
\end{equation}
which cannot be closed in terms of $\mathcal{O}_N^{T_1}$ only. In fact, setting the only existing tree $T_2\in \Tree_2$ of order 2 allows identifying the new observable $\mathcal{O}_N^{T_2}(t,x,y)=\frac{1}{N}\sum_{i,j=1}^Nw_{ij}\bar f_i(t,x)\bar f_j(t,y)$ inside the divergence of \eqref{E-observable-1-pre}. In other words, \eqref{E-observable-1-pre} can be readily written as
\begin{equation}\label{E-observable-1}
\partial_t \mathcal{O}_N^{T_1}(t,x)+\divop_x\left(\int_{\R^d} K(x-y)\,\mathcal{O}_N^{T_2}(t,x,dy)\right)=0.
\end{equation}
Since \eqref{E-observable-1} requires a priori knowledge of $\mathcal{O}_N^{T_2}$ we may seek for an equation of such a new observable. Indeed, from \eqref{independ} we obtain again that
\begin{align}\label{E-observable-2-pre}
\begin{aligned}
\partial_t \mathcal{O}_N^{T_2}(t,x,y)&+\divop_x\left(\int_{\mathbb{R}^d}K(x-z)\frac{1}{N}\sum_{i,j,k=1}^Nw_{ij}w_{ik}\bar f_i(t,x)\bar f_j(t,y)\bar f_k(t,dz)\right)\\
&+\divop_y\left(\int_{\mathbb{R}^d}K(y-z)\frac{1}{N}\sum_{i,j,k=1}^Nw_{ij}w_{jk}\bar f_i(t,x)\bar f_j(t,y)\bar f_k(t,dz)\right)=0.
\end{aligned}
\end{align}
Note that in this case two different divergence terms appear. Accordingly, note that $\Tree_3$ also contains two possible trees $T_3^1=T_2+1$ and $T_3^2=T_2+2$ which respectively arise by adding the new leaf $3$ at either of the vertices $1$ and $2$ of $T_2\in \Tree_2$ (see Figure \ref{fig:trees-T3})
\begin{figure}[t]
\centering
\begin{subfigure}[b]{0.2\textwidth}
\centering
\begin{forest}
for tree={grow=north,edge={-}}
[1 [{\color{red}3}] [2]]
\end{forest}
\caption{Tree $T_3^1\in \Tree_3$}
\end{subfigure}
\begin{subfigure}[b]{0.2\textwidth}
\centering
\begin{forest}
for tree={grow=north,edge={-}}
[1 [2 [{\color{red}3}]]]
\end{forest}
\caption{Tree $T_3^2\in\Tree_3$}
\end{subfigure}
\caption{Trees in $\Tree_3$}
\label{fig:trees-T3}
\end{figure}
By construction
\begin{align*}
\mathcal{O}_N^{T_3^1}(t,x,y,x)&=\frac{1}{N}\sum_{i,j,k=1}^N w_{ij}w_{ik}\bar f_i(t,x)\bar f_j(t,y)\bar f_k(t,z),\\
\mathcal{O}_N^{T_3^2}(t,x,y,x)&=\frac{1}{N}\sum_{i,j,k=1}^N w_{ij}w_{jk}\bar f_i(t,x)\bar f_j(t,y)\bar f_k(t,z).
\end{align*}
Again, we identify the above observables in the divergences of \eqref{E-observable-2-pre}, so that the latter can be restated as
\begin{equation}\label{E-observable-2}
\partial_t \mathcal{O}_N^{T_2}(t,x,y)+\divop_x\left(\int_{\mathbb{R}^d}K(x-z)\mathcal{O}_N^{T_3^1}(t,x,y,dz)\right)+\divop_y\left(\int_{\mathbb{R}^d}K(y-z)\mathcal{O}_N^{T_3^2}(t,x,y,dz)\right)=0.
\end{equation}
It is now apparent that an analogous recursive argument would lead to similar equations like \eqref{E-observable-1} and \eqref{E-observable-2} involving observables indexed by all other higher order trees. In the following section we will make this argument rigorous and will recover the full hierarchy of equations for the observables $\mathcal{O}_N^T$ with $T\in \Tree$. To do so, we shall rely on a more general formulation of the system \eqref{independ} and the observables $\mathcal{O}_N^T$ in \eqref{E-observables} inspired in the treatment of graphons which is formally independent of the decomposition into agents. 

 \subsection{A graphon-like representation of the system and its hierarchy \label{subsec:graphonrepresentation}}
Let us remark that in the limit $N\rightarrow\infty$ we must control two different processes simultaneously: a many-particles limit and a large-graph limit. Unfortunately, the explicit dependence on the discrete labels $i=1,\ldots,N$ of agents still persists both in the expression of the observables in Definition \ref{D-observables} and in the systems \eqref{independ} and \eqref{independnu} with and without artificial diffusion. This is reminiscent of the theory of graphons ({\it cf.} \cite{Lo,LS}). Specifically, graphons have proved a suitable tool to identify the limit of dense graph limits by replacing discrete indices by continuous ones. Accordingly, we expect to recover integrals as $N\rightarrow\infty$ that replace the various finite sums over particles' labels. However, for our graphs, we need to develop an appropriate extension of the theory of graphons that allows for the joint treatment of the large-graph limit for non necessarily dense graphs and the many-particles limit.
\begin{defi}[Empirical graphons]\label{D-graphon-representation}
Consider weights $(w_{ij})_{i,j=1,\ldots,N}\subseteq \mathbb{R}$ and probability distributions $(\bar f_i)_{i=1,\ldots,N}\subseteq \mathcal{P}(\mathbb{R}^d)$. Then, we define $w_N$ and $f_N$ as follows
\begin{align}\label{wf}
\begin{aligned}
w_N(\xi,\zeta)&:=\sum_{i,j=1}^N N\,w_{ij}\,\mathbb{I}_{[\frac{i-1}{N},\frac{i}{N})}(\xi)\;\mathbb{I}_{[\frac{j-1}{N},\frac{j}{N})}(\zeta), && \xi,\,\zeta\in [0,\ 1],\\
f_N(x,\xi)&:=\sum_{i=1}^N \bar f_i(x)\,\mathbb{I}_{[\frac{i-1}{N},\frac{i}{N})}(\xi), && x\in \R^d,\,\xi\in [0,\ 1].
\end{aligned}
\end{align}
\end{defi}
Note that Definition \ref{D-graphon-representation} allows representing discrete weights $(w_{ij})_{i,j=1,\ldots,N}$ as scalar functions of two continuous variables $\xi,\eta\in [0,\ 1]$. Similarly, distributions $(\bar f_i)_{i=1,\ldots,N}$ can be represented as a family of probability measures in $\mathcal{P}(\R^d)$ parametrized by the continuous variable $\xi\in [0,\ 1]$. Using such a notation, we can rewrite the coupled system of PDEs \eqref{independ} and \eqref{independnu} with and without artificial diffusion in a more compact form that forgets labels of the specific individuals. Specifically, we equivalently obtain
\begin{align}
&\partial_t  f_N(t,x,\xi)+\mbox{div}_x\,\left( f_N(t,x,\xi)\,\int_0^1 \int_{\R^d} w_N(\xi,\zeta)\,K(x-y)\, f_N(t,y,\zeta)\,dy\,d\zeta\right)=0,\label{independ2-fN}\\
&\partial_t  f_N(t,x,\xi)+\mbox{div}_x\,\left( f_N(t,x,\xi)\,\int_0^1 \int_{\R^d} w_N(\xi,\zeta)\,K(x-y)\, f_N(t,y,\zeta)\,dy\,d\zeta\right)=\nu \Delta_x f_N(t,x,\xi),\label{independ2nu-fN}
\end{align}
in the sense of distributions. Our ultimate goal will be to identify the limit of the above systems as $N\rightarrow\infty$. For later use, we specify minimal assumptions on the space which will contain the limiting structures.

\begin{defi}\label{D-weak-star-Bochner-space}
We define the weak-* Bochner space $L^\infty_\xi\mathcal{M}_\zeta$ as the topological dual of $L^1_\xi C_\zeta$; namely those elements consist of the weak-* measurable essentially bounded maps $\xi\in [0,\ 1]\mapsto w(\xi,d\zeta)\in \mathcal{M}([0,\ 1])$ which make the following seminorm finite
$$\Vert w\Vert_{L^\infty_\xi\mathcal{M}_\zeta}:=\sup_{\Vert \phi\Vert_{C([0,\ 1])}\leq 1}\esssup_{\xi\in [0,\ 1]}\left\vert\int_{0}^1\phi(\zeta)\,w(\xi,d\zeta)\right\vert.$$
Similarly, we define $L^\infty_\zeta\mathcal{M}_\xi$, or more generally $L^\infty_\xi B_\zeta^*$ for any dual Banach space $B^*$. 
\end{defi}

As usual, we will identify $L^\infty_\xi \mathcal{M}_\zeta$ with its Kolmogorov quotient under the above semi-norm $\Vert \cdot\Vert_{L^\infty_\xi \mathcal{M}_\zeta}$, which is a Banach space. Specifically, we shall identify elements $w_1\approx w_2$ when for any $\phi\in C([0,\ 1])$ the scalar functions $\xi\in [0,\ 1]\mapsto \int_{\R^d}\phi(\zeta)\,w_i(\xi,d\zeta)$ are equal almost everywhere for $i=1,2$. As mentioned above, the main feature of such a Banach space is its duality representation $L^\infty_\xi\mathcal{M}_\zeta\equiv (L^1_\xi C_\zeta)^*$, where $L^1_\xi C_\zeta$ stands for the (strong) Bochner space, see \cite{IT-77}. Although this property is desirable, the weak-* Bochner spaces raise several analytical difficulties. For instance, strong-measurability of the representatives is not granted because $\mathcal{M}([0,\ 1])$ fails the Radon-Nikodym property. Fortunately, in our special case $C([0,\ 1])$ is separable so that, at least, $\xi\in [0,\ 1]\mapsto\Vert w(\xi,\cdot)\Vert_{\mathcal{M}([0,\ 1])}$ is measurable and essentially bounded. For the same reason, the above equivalence relation $\approx$ reduces to the usual {\it a.e.} identity of functions. Therefore, a simpler expression of the norm in the separable case is available:
$$\Vert w\Vert_{L^\infty_\xi \mathcal{M}_\zeta}=\esssup_{\xi\in [0,\ 1]}\Vert w(\xi,\cdot)\Vert_{\mathcal{M}([0,\ 1])}.$$

\begin{defi}\label{D-extended-graphons}
We shall say that $w\in L^\infty_\xi\mathcal{M}_\zeta \cap L^\infty_\zeta\mathcal{M}_\xi$ is an (extended) graphon.
\end{defi}
For our purpose we will need to integrate extended graphons against functions $\phi$ that are only in $L^\infty$. This leads to the straightforward issue as to whether $\xi\in [0,\ 1]\mapsto \int_0^1 \phi(\zeta)\,w(\xi,d\zeta)$ can be well defined if $\phi\in L^\infty([0,\ 1])$. In fact if we only had $w\in L^\infty_\xi\mathcal{M}_\zeta$, then the above integral would be highly sensitive to modifying $\phi(\zeta)$ over the atoms of $w(\xi,d\zeta)$, preventing any straightforward definition.

The key point for our investigations is that our extended graphons belongs to both  $L^\infty_\xi\mathcal{M}_\zeta$ and $L^\infty_\zeta\mathcal{M}_\xi$. This yields two weakly-star measurable and bounded families of measures $\xi\in [0,\ 1]\mapsto w(\xi,d\zeta)\in \mathcal{M}([0,\ 1])$ and $\zeta\in [0,\ 1]\mapsto w(d\xi,\zeta)\in \mathcal{M}([0,\ 1])$ that induce the same bi-variate measure $w(\xi,d\zeta)\,d\xi=w(d\xi,\zeta)\,d\zeta\in \mathcal{M}([0,\ 1]^2)$. 
This implies that $w$ is a transition kernel in the full sense and the following continuity of the associated linear operator on $L^\infty$ holds true.
\begin{lem}[Operator representation of kernels]\label{bilinearbound}
Consider the following bounded bilinear operator
\[
\begin{array}{ccl}
(L^\infty_\xi\mathcal{M}_\zeta\cap L^\infty_\zeta\mathcal{M}_\xi)\times C_\zeta & \longrightarrow & L^\infty_\xi,\\
(w,\phi) & \longmapsto & \int_0^1\phi(\zeta)\,w(\cdot,d\zeta).
\end{array}
\]
This operator extends into a bounded operator from $(L^\infty_\xi\mathcal{M}_\zeta\cap L^\infty_\zeta\mathcal{M}_\xi)\times L^\infty_\zeta$ to $L^\infty_\xi$. 
Specifically, we have the following estimates
\begin{align}
\begin{aligned}
&\left\|\int_0^1 \phi(\zeta)\,w(\xi,d\zeta)\right\|_{L^1_\xi}\leq \|w\|_{L^\infty_\zeta \mathcal{M}_\xi}\,\|\phi\|_{L^1},\\ 
&\left\|\int_0^1 \phi(\zeta)\,w(\xi,d\zeta)\right\|_{L^\infty_\xi}\leq \|w\|_{L^\infty_\xi \mathcal{M}_\zeta}\,\|\phi\|_{L^\infty},
  \label{wfL1Linfty}
  \end{aligned}
  \end{align}
  for any $w\in L^\infty_\xi\mathcal{M}_\zeta\cap L^\infty_\zeta\mathcal{M}_\xi$ and $\phi\in L^\infty$. Moreover, the following continuity property holds: if $w_n\overset{*}{\rightharpoonup} w$ weakly-star in $L^\infty_\xi\mathcal{M}_\zeta\cap L^\infty_\zeta\mathcal{M}_\xi$, $\phi_n\rightarrow \phi$ strongly in $L^1$ and the $\phi_n$ are uniformly bounded in $L^\infty$, then
  \begin{equation}\label{wfweakcont}
\int_0^1\phi_n(\zeta)\,w_n(\cdot,d\zeta)\rightarrow \int_0^1 \phi(\zeta)\,w(\cdot,d\zeta),
  \end{equation}
  weakly-star in $L^\infty$.
\end{lem} 
The proof of Lemma~\ref{bilinearbound} relies on a careful but straightforward density argument, which we provide in Appendix~\ref{technicalappendix}.
Lemma \ref{bilinearbound} immediately implies the following version with test functions $\phi$ depending on~$x$.
\begin{lem}\label{bilinearboundBx}
Consider the following bounded bilinear operator
\[
\begin{array}{ccl}
(L^\infty_\xi\mathcal{M}_\zeta\cap L^\infty_\zeta\mathcal{M}_\xi)\times C_\zeta B_x^* & \longrightarrow & L^\infty_\xi B_x^*,\\
(w,\phi) & \longmapsto & \int_0^1\phi(\cdot,\zeta)w(\cdot,d\zeta),
\end{array}
\]
where $B_x^*$ stands for any Banach space. Then it is also bounded from $(L^\infty_\xi\mathcal{M}_\zeta\cap L^\infty_\zeta\mathcal{M}_\xi)\times L^\infty_\zeta B_x^*$ to $L^\infty_\xi B_x^*$, and from  $(L^\infty_\xi\mathcal{M}_\zeta\cap L^\infty_\zeta\mathcal{M}_\xi)\times L^\infty_\zeta B_x^*$ (with $L^\infty_\zeta B_x^*$ endowed with the $L^1_\zeta B_x^*$ norm) to $L^1_\xi B_x^*$. Specifically, we have the following estimates
\begin{align}\label{wfL1LinftyBx}
\begin{aligned}
\left\|\int_0^1 \phi(x,\zeta)\,w(\xi,d\zeta)\right\|_{L^1_\xi B_x^*}&\leq \|w\|_{L^\infty_\zeta \mathcal{M}_\xi}\,\|\phi\|_{L^1_\zeta B_x^*},\\
\left\|\int_0^1 \phi(x,\zeta)\,w(\xi,d\zeta)\right\|_{L^\infty_\xi B_x^*}&\leq \|w\|_{L^\infty_\xi \mathcal{M}_\zeta}\,\|\phi\|_{L^\infty_\zeta B_x^*},
\end{aligned}
\end{align}
for any $w\in L^\infty_\xi\mathcal{M}_\zeta\cap L^\infty_\zeta\mathcal{M}_\xi$ and $\phi\in L^\infty_\zeta B_x^*$, 
\end{lem}  
Note that $w_N$ and $f_N$ in Definition \ref{D-graphon-representation} must belong to the previous spaces. Then, the study of the limit of those structures as $N\rightarrow\infty$ and the identification of the limiting equations of \eqref{independ2-fN} and \eqref{independ2nu-fN} is in order. Namely, we expect to obtain $w\in L^\infty_\xi\mathcal{M}_\zeta \cap L^\infty_\zeta\mathcal{M}_\xi$ and $f\in L^\infty([0,\ t_*],\,L^\infty_\xi L^1_x)$ with $t_*>0$ verifying either of the following equations (with or without artificial diffusion) in a weak sense:
\begin{align}
\partial_t  f(t,x,\xi)+\mbox{div}_x\,\left( f(t,x,\xi)\,\int_0^1 w(\xi,d\zeta) \int_{\R^d} K(x-y)\, f(t,y,\zeta)\,dy\right)&=0,\label{independ2}\\
\partial_t  f(t,x,\xi)+\mbox{div}_x\,\left( f(t,x,\xi)\,\int_0^1 w(\xi,d\zeta) \int_{\R^d} K(x-y)\, f(t,y,\zeta)\,dy\right)&=\nu \Delta_x f(t,x,\xi).\label{independ2nu}
\end{align}

In \eqref{independ2} and \eqref{independ2nu} we recognize a nonlinear drift, which is fibered as there is no transport on the variable $\xi$. For later use, we define formally the associated velocity field.

\begin{defi}[Velocity field]\label{D-velocity-field}
Consider any $w\in L^\infty_\xi\mathcal{M}_\zeta\cap L^\infty_\zeta\mathcal{M}_\xi$, $f\in L^\infty_{t,\xi} L^1_x$ and $K\in L^\infty$. Let us set $\phi_f(t,x,\zeta):=\int_{\mathbb{R}^d}K(x-y)f(t,y,\zeta)\,dy$ which is well defined in $L^\infty_{\xi,t,x}$ by the usual convolution estimates.  We define the associated velocity field $\mathcal{V}[f]$ as follows,
\[
\mathcal{V}_f(t,x,\xi):=\int_0^1 w(\xi,d\zeta)\,\phi_f(t,x,\zeta)=\int_0^1 w(\xi,d\zeta)\int_{\mathbb{R}^d}K(x-y)\,f(t, \zeta,y)\,dy,\quad (\xi, x)\in  [0,\ 1]\times \mathbb{R}^d, \ t \geq 0,
\]
which is well defined in $L^\infty_{t,\xi,x}$ by Lemma~\ref{bilinearboundBx}.
\end{defi}
We may then define the weak solutions to \eqref{independ2}-\eqref{independ2nu} in the usual way.
\begin{defi}[Weak solutions]\label{D-weak-solution-independ2}
Consider any pair consisting of a $w\in L^\infty_\xi\mathcal{M}_\zeta \cap L^\infty_\zeta\mathcal{M}_\xi$ and any $f\in L^\infty([0,\ t_*],\,L^\infty_\xi L^1_x)$. We say that $(w,f)$ is a weak solution of \eqref{independ2nu} for some $K\in L^\infty$ and for any $\nu\geq 0$ if the following weak formulation is verified
\begin{align}\label{E-weak-solution-independ2}
\begin{aligned}
\int_0^{t_*} \int_0^1\int_{\mathbb{R}^d}&\partial_t \varphi(t,x,\xi)\,f(t,x,\xi)\,dx\,d\xi\,dt\\
&+\int_0^{t_*} \int_0^1\int_{\mathbb{R}^d}\nabla_x\varphi(t,x,\xi)\cdot \mathcal{V}_f(t,x,\xi)\,f(t,x,\xi)\,dx\,d\xi\,dt\\
&+\nu\int_0^{t_*}\int_0^1\int_{\mathbb{R}^d}\Delta_x\varphi(t,x,\xi)\,f(t,x,\xi)\,dx\,d\xi\,dt=0,
\end{aligned}
\end{align}
for any $\varphi\in C^2_c((0,\ t_*)\times [0,\ 1]\times \mathbb{R}^d)$.
\end{defi}
We have the following result about the existence of weak solutions. 
\begin{pro}
Consider any $K\in L^1\cap L^\infty$ with $\divop K\in L^1$, any $w\in L^\infty_\xi \mathcal{M}_\zeta\cap L^\infty_\zeta \mathcal{M}_\xi$ and any $f^0\in L^\infty([0,\ 1],\ W^{1,1}\cap W^{1,\infty}(\R^d))$. Then for any $\nu\geq 0$, there exists $f\in L^\infty([0,\ t_*]\times [0,\ 1],\ W^{1,1}\cap W^{1,\infty}(\R^d))$ such that $(w,f)$ is a weak solution to \eqref{independ2nu} (or \eqref{independ2} if $\nu=0$) for every $t_*>0$  as per Definition~\ref{D-weak-solution-independ2}.\label{existenceweak}
  \end{pro}
The proof is given in Appendix~\ref{technicalappendix} and follows a straightforward fixed point argument from classical a priori estimates for advection-diffusion equations. We also note that it would be possible to obtain an even more existence result with $f$ only a measure; for example through the pathwise approaches developed in~\cite[Chapter~10]{RachevR}. 
  
We finish this short analysis of the space $L^\infty_\xi \mathcal{M}_\zeta\cap L^\infty_\zeta \mathcal{M}_\xi$ by noting the following density result
\begin{lem}
For any $w\in L^\infty_\xi \mathcal{M}_\zeta\cap L^\infty_\zeta \mathcal{M}_\xi$, there exists $w_n\in L^\infty_{\xi,\zeta}$, uniformly bounded in $L^\infty_\xi L^1_\zeta\cap L^\infty_\zeta L^1_\xi$,   converging to $w$ in $L^1_\xi H^{-1}_\zeta\cap L^1_\zeta H^{-1}_\xi$.
\label{gooddensity}  \end{lem}
We also provide the proof of Lemma~\ref{gooddensity} in Appendix~\ref{technicalappendix}.

Unfortunately we do not know how to pass to the limit directly in Eqs. \eqref{independ2-fN} or \eqref{independ2nu-fN}, as we lack some compactness in the $\xi$ variable either for $w_N$ or for $f_N$. Indeed, from our uniform bounds, we could extract converging subsequences $w_N\overset{*}{\rightharpoonup} w$ and $f_N\overset{*}{\rightharpoonup} f$ respectively in $L^\infty_\xi \mathcal{M}_\zeta\cap L^\infty_\zeta \mathcal{M}_\xi$ and $L^\infty_{t,\xi} (L^1_x\cap W^{1,\infty}_x)$. But this gives only weak convergence on $w_N$ and weak convergence in $\xi$ on $f_N$ whereas Lemma~\ref{bilinearbound} for example requires some strong convergence on $f_N$. To further emphasize the point, the existence result that we just stated does rely on strong convergence of $f$ through the contraction argument.

To bypass this issue, we will look at an analogous family of observables providing an averaged information of the system, since it is much easier to obtain compactness on those observables as we will see in the next section. Inspired in Definition \ref{D-observables} we set the following operator acting on abstract elements $w\in L^\infty_\xi\mathcal{M}_\zeta\cap L^\infty_\zeta\mathcal{M}_\xi$ and $f\in L^\infty_\xi L^1_x$.

\begin{defi}[$\tau$ operator]\label{D-tau}
Consider any  $w\in L^\infty_\xi\mathcal{M}_\zeta \cap L^\infty_\zeta\mathcal{M}_\xi$ and any $f\in L^\infty_\xi L^1_x$. Then, we (formally) define the operator
\begin{equation}\label{tau}
\tau(T,w,f)(x_1,\ldots,x_{|T|}):=\int_{[0,\ 1]^{|T|}} \prod_{(k,l)\in E(T)} w(\xi_k,\xi_l)\,\prod_{m\in V(T)} f(x_m,\xi_m)\,d\xi_1\dots d\xi_{|T|},
\end{equation}
for every tree $T\in \Tree$. Time dependence can be included when $f$ also depends on $t$.
\end{defi}
This notation allows removing again the explicit discrete indices for the agents' labels in Definition \ref{D-observables} of the observables $\mathcal{O}_N^T$. Specifically, note that the latter enjoy the following more universal representation
\[
\mathcal{O}_N^T=\tau(T,w_N,f_N),\quad T\in \Tree.
\]
We also note that, at the present, it is not clear why $\tau(T,w,f)$ given by \eqref{tau} is well defined with only $w\in L^\infty_\xi\mathcal{M}_\zeta \cap L^\infty_\zeta\mathcal{M}_\xi$ and $f\in L^\infty_\xi L^1_x$. This requires more work and will be performed in the next section ({\it cf.} Definition \ref{D-tau-extended}). Instead, for the moment it is enough to observe that $\tau(T,w_N,f_N)$ is indeed well defined, and more generally, $\tau(T,w,f)$ is also well defined as soon as $w$ has a density, {\it i.e.} $w\in L^\infty_\xi L^1_\zeta\cap L^\infty_\zeta L^1_\xi$. Moreover, we can then obtain uniform bounds in terms of only the $L^\infty_\xi\mathcal{M}_\zeta \cap L^\infty_\zeta\mathcal{M}_\xi$ norms as follows.
\begin{lem}[Homomorphism densities]\label{L-tau-density-bound}
Consider any $w\in L^\infty_\xi L^1_\zeta \cap L^\infty_\zeta L^1_\xi$ 
and define
\begin{equation}\label{tau-density}
\tau(T,w):=\int_{[0,\ 1]^{|T|}}\prod_{(k,l)\in E(T)} w(\xi_k,\xi_l)\,d\xi_1\ldots d\xi_{|T|},\quad T\in \Tree.
\end{equation}
Then, $\tau(T,w)$ is finite and the following relation holds true
\begin{equation}\label{tau-density-bound}
\vert \tau(T,w)\vert\leq \Vert w\Vert_{L^\infty_\xi \mathcal{M}_\zeta}^{|T|-1}.
\end{equation}
\end{lem}
\begin{proof}
We argue by induction. First, note the only tree $T_1\in \Tree_1$ actually consists of only one vertex and no edges. Then, we trivially obtain that $\prod_{(k,l)\in E(T_1)}| w(\xi_k,\xi_l)|=1$ so that $\tau(T_1,\vert w\vert)=1$. By the induction hypothesis, let us assume now that we have
$$\tau(|T_n|,w)\leq \Vert w\Vert_{L^\infty_\xi\mathcal{M}_\zeta}^{|T_n|-1},$$
for any $T_n\in \Tree_n$ . Consider any $T_{n+1}\in \Tree_{n+1}$ and find a tree $T_n\in \Tree_n$ and a vertex $i\in \{1,\ldots,n\}$ so that we can write $T_{n+1}=T_n+i$. Hence, we obtain
\begin{align*}
\tau(T_{n+1},\vert w\vert)&=\int_{[0,\ 1]^{n+1}}\prod_{(k,l)\in E(T_{n+1})}\vert w(\xi_k,\xi_l)\vert\,d\xi_1\,\ldots\,d\xi_{n+1}\\
&=\int_{[0,\ 1]^n}\left(\int_0^1 \vert w(\xi_i,\xi_{n+1})\vert\,d\xi_{n+1}\right)\prod_{(k,l)\in E(T_n)}|w(\xi_k,\xi_l)|\,d\xi_1\ldots\,d\xi_n,
\end{align*}
where we have used that $\xi_{n+1}$ does not appears in the above product over edges of $T_n$. Hence, by our hypothesis on $w$ and by definition \eqref{tau-density} of $\tau(T_n,\vert w\vert)$ we obtain
$$\tau(T_{n+1},\vert w\vert)\leq \Vert w\Vert_{L^\infty_\xi\mathcal{M}_\zeta}\,\tau(T_n,\vert w\vert).$$
We then conclude by applying the induction hypothesis to $\tau(T_n,\vert w\vert)$.
\end{proof}
 We can obtain as a straightforward consequence of Lemma \ref{L-tau-density-bound}, 
\begin{lem}\label{L-tau-bound}
Consider any $w\in L^\infty_\xi L^1_\zeta \cap L^\infty_\zeta L^1_\xi$, any $f\in L^\infty_\xi L^1_x$ and $T\in \Tree$. Then $\tau(T,w,f)\in L^1(\R^{d |T|})$ and the following estimate is fulfilled
$$\Vert \tau(T,w,f)\Vert_{L^1(\R^{d|T|})}\leq \Vert w\Vert_{L^\infty_\xi \mathcal{M}_\zeta \cap L^\infty_\zeta \mathcal{M}_\xi}^{|T|-1}\Vert f\Vert_{L^\infty_\xi L^1_x}^{|T|}.$$
A similar result holds  when the space  $ L^\infty_\xi L^1_x$ for $f$ is replaced by any other function Banach space stable under tensorization and admitting a Minkowski integral inequality, but maintaining $w\in L^\infty_\xi L^1_\zeta \cap L^\infty_\zeta L^1_\xi$. In particular, if $f\in L^\infty_\xi W^{k,p}_x$ for some $k\geq 0$ and $p\in [1,\infty]$, then $\tau(T,w,f)\in W^{k,p}(\R^{d|T|})$ and
$$\Vert \tau(T,w,f)\Vert_{W^{k,p}(\R^{d|T|})}\leq \Vert w\Vert_{L^\infty_\xi \mathcal{M}_\zeta \cap L^\infty_\zeta \mathcal{M}_\xi}^{|T|-1}\,\Vert f\Vert_{L^\infty_\xi W^{k,p}_x}^{\vert T\vert}.$$
\end{lem}
Formula \eqref{tau-density} for $\tau(T,w)$ in Lemma \ref{L-tau-bound} is reminiscent of the usual extension of the homomorphism density for bounded graphons $w\in L^\infty([0,\ 1]^2)$  \cite{Lo,LS}. However, our theory must account for eventually unbounded graphons so that a more delicate treatment is required. Namely, note that in \eqref{tau-density} we have restricted to trees $T\in \Tree$ and  graphons $w\in L^\infty_\xi\mathcal{M}_\zeta \cap L^\infty_\zeta\mathcal{M}_\xi$. Indeed, in such a class of graphons, if $T$ is replaced by a generic graph $G$, then the above proof breaks down due to the eventual presence of cycles in the graph. This leads to eventual infinite values of $\tau(G,w)$. 

We note that we prove in section~\ref{sec:graphons} that the operator $\tau(T,w,f)$ can be rigorously defined for any $w\in L^\infty_\xi \mathcal{M}_\zeta \cap L^\infty_\zeta \mathcal{M}_\xi$. We refer more precisely to Definition~\ref{D-tau-extended} and in the present context to Lemma~\ref{welldefinedM} which implies that if $w_n\in L^\infty_\xi L^1_\zeta \cap L^\infty_\zeta L^1_\xi$ converges to $w$ in $L^1_\xi H^{-1}_\zeta \cap L^1_\zeta H^{-1}_\xi$ then $\tau(T,w,f)$ can be obtained as
\[
\tau(T,w,f)=\lim_{n\rightarrow\infty} \tau(T,w_n,f).
\]
Of course Lemma~\ref{gooddensity} exactly allows constructing an appropriate sequence of $w_n$. But deriving the limit of $\tau(T,w_n,f_n)$ as above requires the precise construction of an algebra describing the sequence of operations involved in obtaining $\tau(T,w,f)$. This algebra plays a key role in Section~\ref{sec:graphons} and, for this reason, the corresponding results are stated there. 

Assuming for the time being that type of property holds, it implies an immediate extension of Lemma~\ref{L-tau-bound} to more generic kernels.
\begin{cor}\label{cor-L-tau-bound}
Consider any $w\in L^\infty_\xi \mathcal{M}_\zeta \cap L^\infty_\zeta \mathcal{M}_\xi$, any $f\in L^\infty_\xi W^{k,p}_x$ for any $k\geq 0$, $p\in [1,\;\infty]$, and $T\in \Tree$. Then $\tau(T,w,f)\in W^{k,p}(\R^{d |T|})$ 
\[
\Vert \tau(T,w,f)\Vert_{W^{k,p}(\R^{d|T|})}\leq   \Vert w\Vert_{L^\infty_\xi \mathcal{M}_\zeta \cap L^\infty_\zeta \mathcal{M}_\xi}^{|T|-1}\,\Vert f\Vert_{L^\infty_\xi W^{k,p}_x}^{\vert T\vert}.
\]
\end{cor}
In the following result, we derive the full hierarchy of observables presented in the previous section under our graphon representation. Specifically, we now provide the system solved by $\tau(T,w,f)$  whenever $(w,f)$ is a weak solution to either \eqref{independ2} or \eqref{independ2nu} in the sense of Definition \ref{D-weak-solution-independ2}.

\begin{pro}\label{hierarchyeq}
Consider any  $w\in L^\infty_\xi \mathcal{M}_\zeta \cap L^\infty_\zeta \mathcal{M}_\xi$ and $f\in L^\infty([0,\ t_*],\,L^\infty_\xi L^1_x)$ so that $(w,f)$ is a weak solution to \eqref{independ2nu} in the sense of Definition \ref{D-weak-solution-independ2} for some $K\in L^\infty$ and for any $\nu\geq 0$. Then, $\tau(T,w,f)$ solves the generalized, linear, non-exchangeable Vlasov hierarchy
\begin{equation}\label{E-hierarchyeq}
\partial_t \tau(T,w,f)+\sum_{i=1}^{|T|} \divop_{x_i}\left(\int_{\R^d} K(x_i-z)\,\tau(T+i,w,f)(t,x_1,\ldots,x_{|T|},z)\,dz\right)=\nu\sum_{i=1}^{|T|}\Delta_{x_i}\tau(T,w,f),
\end{equation}
for any tree $T\in \Tree$ in the sense of distributions.
\end{pro}
\begin{proof}
First of all, we notice that from the properties of $w$ and $f$ we can guarantee that $\tau(T,w,f)\in L^\infty([0,\ t_*],\,L^1(\R^{d |T|}))$ thanks to Corollary~\ref{cor-L-tau-bound}. Since $K\in L^\infty$, then there is no issue to define the various terms in the equation and justify our various calculations in the distributional sense. For simplicity of the presentation, we avoid using weak formulations here but a straightforward adaptation yields the rigorous argument. Also, we shall shorten notation by denoting $\tau(T)\equiv \tau(T,w,f)$, for any $T\in \Tree$, when there is no confusion. First, we differentiate $\tau(T,w,f)$ in time
\[
\partial_t\tau(T)=\sum_{n=1}^{|T|}\int_{[0,\ 1]^{|T|}} \prod_{(k,l)\in E(T)} w(\xi_k,\xi_l)\,\partial_t f (t,x_n,\xi_n)\,\prod_{m\neq n} f(t,x_m,\xi_m)\,d\xi_1\dots d\xi_{|T|}.
 \]
Using \eqref{independ2nu} on $ f(t,x_n,\xi_n)$ yields
\begin{align*}
\partial_t\tau(T)
&=-\sum_{n=1}^{|T|}\int_{[0,\ 1]^{|T|}} \prod_{(k,l)\in E(T)} w(\xi_k,\xi_l)\,\divop_{x_n} \Bigg(\int_0^1\int_{\R^d} K(x_n-y)\,w(\xi_n,\zeta)\, f(t,x_n,\xi_n)\,f(t,y,\zeta)\,dy\,d\zeta\Bigg)\\
&\qquad \times \prod_{m\neq n} f(t, x_m,\xi_m)\,d\xi_1\dots d\xi_{|T|}       \\
&\quad+\nu\sum_{n=1}^{|T|}\,\int_{[0,\ 1]^{|T|}} \prod_{(k,l)\in E(T)} w(\xi_k,\xi_l)\,\Delta_{x_n} f (t,x_n,\xi_n)\,\prod_{m\neq n} f(t, x_m,\xi_m)\,d\xi_1\dots d\xi_{|T|}\\
=&-\sum_{n=1}^{|T|}\divop_{x_{n}}\Bigg(\int_{[0,\ 1]^{|T|}}\prod_{(k,l)\in E(T)}w(\xi_k,\xi_l)\int_0^1\int_{\R^d} K(x_n-y)\,w(\xi_n,\zeta)\, f (t,x_n,\xi_n)\,f(t,y,\zeta)\,dy\,d\zeta\\
&\qquad\times \prod_{m\neq n}  f(t,x_m,\xi_m)\,d\xi_1\dots d\xi_{|T|}\Bigg)+\nu\sum_{n=1}^{|T|}\Delta_{x_n} \tau(T),
\end{align*}
where in the last line we have used that no other $f(t,x_m,\xi_m)$ depends on $x_n$ as we exclude $m=n$ from the product. Hence, we find
\begin{align*}
\partial_t\tau(T)=&-\sum_{n=1}^{|T|}\divop_{x_{n}}\int_{\R^d} K(x_{n}-z)\int_{[0,\ 1]^{|T|}}\int_0^1\Bigg(w(\xi_n,\zeta) \prod_{(k,l)\in E(T)} w(\xi_k,\xi_l)\\
&\qquad\times f(t,z,\zeta)\prod_{m\neq n}  f(t,x_m,\xi_m)\Bigg)\,d\zeta\,d\xi_1\ldots\,d\xi_{|T|}\,dz+\nu\sum_{n=1}^{|T|}\Delta_{x_n} \tau(T).
\end{align*}
We recall that for any $n=1,\ldots,|T|$ the tree $T+n$ contains exactly the same edges as $T$ plus a new edge from the vertex $n$ to the new vertex $\vert T\vert+1$. Thus, recalling that we only consider edges $(k,l)\in E(T+n)$ that run from the root to the leaves, we have that
\[
\tau(T+n)=\int_{[0,\ 1]^{|T|+1}} \, w(\xi_n, \xi_{|T|+1})\prod_{(k,l)\in E(T)} w(\xi_k,\xi_l)\;f(t,x_{|T|+1},\xi_{|T|+1})\prod_{m\in V(T)} f(t,x_m,\xi_m)\,d\xi_1\ldots\,d\xi_{|T|+1}.
\]
Changing variables $\xi_{|T|+1}$ with $\zeta$ and $x_{|T|+1}$ with $z$, we arrive at
\[
\partial_t\tau(T)=-\sum_{n=1}^{|T|}\divop_{x_n}\int_{\R^d} K(x_n-z)\,\tau(T+n)(x_1,\ldots,x_{|T|},z)\,dz+\nu\sum_{n=1}^{|T|}\Delta_{x_n} \tau(T),
\]
thus, concluding the proof.
\end{proof}

Proposition ~\ref{hierarchyeq} shows that we can study directly the propagation of the observables $\tau(T,w,f)$.
Our next step is naturally to analyze the stability on the hierarchy with diffusion~\eqref{E-hierarchyeq}.

  \subsection{Stability on the hierarchy with artificial diffusion}
Along this section, we will study a stability property of generic solutions to the previous hierarchy of non-exchangeable Vlasov-type equations \eqref{E-hierarchyeq} which are not necessarily parametrized as $\tau(T,w,f)$ for a weak solution $(w,f)$ of \eqref{independ2nu}. Specifically, we shall consider any sequence $h=(h_T)_{T\in \Tree}$ with $h_T\in L^\infty([0,\ t_*],\,L^1(\R^{d\vert T\vert}))$ for any $T\in \Tree$ that solves the analogous hierarchy of non-exchangeable Vlasov equations
\begin{equation}\label{hierarchydiff}
\partial_t h_T+\sum_{i=1}^{|T|} \mbox{div}_{x_i}\left(\int_{\R^d} K(x_i-z)\,h_{T+i}(t,x_1,\ldots,x_{|T|},z)\,dz\right)=\nu\,\sum_{i=1}^{|T|} \Delta_{x_i} h_T,
\end{equation}
in the sense of distributions, for some $K\in L^\infty$ and any $\nu>0$.

We do need to emphasize that this section does not provide a well-posedness theory for~\eqref{hierarchydiff}. While we do obtain uniqueness on solutions with enough a priori estimates, this does not directly yield the existence of any solutions in that class. This does not impact the present paper as we only need the uniqueness result for the special type of solutions parametrized as $h_T=\tau(T,w,f)$ with $(w,f)$ a weak solution of \eqref{independ2nu} ({\it cf.} Proposition \ref{hierarchyeq}).

We also emphasize that the presence of artificial diffusion is needed in the method of proof in this subsection and a similar result with $\nu=0$ is still unknown. To analyze stability we shall define the following norms.
%
\begin{defi}[Norm of the hierarchy]\label{D-norm-hierarchy}
Consider any family $h=(h_T)_{T\in \Tree}$ such that $h_T\in L^2(\R^{d |T|})$ for every $T\in \Tree$. Then, we define the following norms:
\begin{equation}\label{E-norm-hierarchy}
 \|h\|_{\lambda}=\sup_{T\in \Tree} \lambda^{|T|/2}\,\|h_T\|_{L^2(\R^{d\,|T|})},
\end{equation}
for any $\lambda>0$.
\end{defi}
The norm \eqref{E-norm-hierarchy} is natural since, for any $f\in L^2(\R^d)$ and the special sequence $h=(h_T)_{T\in \Tree}$ given by $h_T=\tau(T,w,f)$ for any $T\in \Tree$, we obtain the following bound
\[
\Vert h\Vert_\lambda\leq \sup_{T\in \Tree} \lambda^{|T|/2}\Vert w\Vert_{L^\infty_\xi\mathcal{M}_\zeta}^{|T|-1}\Vert f\Vert_{L^2}^{|T|},
\]
for any $\lambda>0$ thanks to Corollary~\ref{cor-L-tau-bound}. Hence, by choosing $\lambda > 0$ small enough relative to the inverses of the preceding norms  of $w$ and $f$, {\it i.e.} less than $\|w\|^{-2\frac{|T|-1}{|T|}}_{L^\infty_\xi\mathcal{M}_\zeta} \|f\|^{-2}_{L^2}$, we obtain a finite quantity. Indeed, it provides the following stability result for solutions of \eqref{hierarchydiff}.
%
\begin{theo} \label{stabilitydiff}
Consider any solution $h=(h_T)_{T\in \Tree}$ in the sense of distribution to \eqref{hierarchydiff} for some $\nu>0$ and $K\in L^2$ such that $h_T\in L^\infty([0,\ t_*],\,(L^1\cap L^2)(\R^{d |T|}))$ for any $T\in \Tree$. Assume that there exists some $\lambda>0$ such that $C_\lambda:=\sup_{t\in [0,\ t_*]}\|h(t,\cdot)\|_{\lambda}<\infty$. Then, for any $p>1$ and any $\theta\in (0,\ 2^{-p'})$ there exists a constant $C_{p,\theta}\in \mathbb{R}_+$ such that 
\begin{equation}\label{E-stabilitydiff}
\Vert h(t,\cdot)\Vert_{\theta\lambda}\leq C_\lambda C_{p,\theta}\exp\left(p^{-\frac{\Vert K\Vert_{L^2}^2}{2\theta\lambda\nu}\,t}\,\log\frac{\Vert h^0\Vert_{\theta \lambda}}{C_\lambda}\right),
\end{equation}
for any $t\in [0,\ t_*]$, where $p'$ is the conjugate of $p$.
\end{theo}

\begin{rem} 
Taking $p>1$ sufficiently large, note that any $\theta\in (0,\ \frac{1}{2})$ provides an estimate \eqref{E-stabilitydiff}. However, such a constraint on the size of $\theta$ is clearly an artifact of the method of proof since a similar estimate can be obtained for any $\theta\in (0,\ 1)$ by interpolation. Namely, set any $p_0>1$ and any $\theta_0\in (0,\ 2^{-p_0'})$ so that \eqref{E-stabilitydiff} is fulfilled. Consider any arbitrary $\theta\in [\frac{1}{2},\ 1)$. Then, note that $\theta_0\lambda<\theta\lambda<\lambda$ and
$$\Vert h(t,\cdot)\Vert_{\theta\lambda}\leq \Vert h(t,\cdot)\Vert_{\theta_0\lambda}^\alpha\,\Vert h(t,\cdot)\Vert_{\lambda}^{1-\alpha},$$
for any $t\in [0,\ t_*]$, where we have set $\alpha:=\log\theta/\log\theta_0$. Then, we obtain
\begin{equation}\label{E-stabilitydiff-full}
\Vert h(t,\cdot)\Vert_{\theta\lambda}\leq C_\lambda C_{p_0,\theta_0}^\alpha\exp\left(\alpha p_0^{-\frac{\Vert K\Vert_{L^2}^2}{2\theta_0\lambda\nu}\,t}\,\log\frac{\Vert h^0\Vert_{\theta \lambda}}{C_\lambda}\right),
\end{equation}
for any $t\in [0,\ t_*]$.
\end{rem}

\begin{rem}\label{R-no-semigroup}
Thanks to the linearity of the hierarchy, Theorem \ref{stabilitydiff} provides uniqueness among weak solutions $h=(h_T)_{T\in\Tree}$ to \eqref{hierarchydiff} for some $\nu>0$, such that $h_T\in L^\infty([0,\ t_*],\,(L^1\cap L^2)(\R^{d |T|})$ for any $T\in\Tree$ and $\sup_{t\in [0,\ t_*]}\|h(t,\cdot)\|_\lambda<\infty$ for some $\lambda>0$. Indeed, assume that $h^0=0$. Then, $\|h^0\|_{\theta\lambda}=0$ for any $\theta\in (0,\ 1)$ so that \eqref{E-stabilitydiff} trivially implies that  $\|h(t,\cdot)\|_{\theta\lambda}=0$ ({\it i.e.}, $h(t,\cdot)=0$) at any later time $t\in (0,\ t_*]$. Unfortunately, it does not however provide a bounded or Markov semi-group for those norms and in fact semi-group like estimates blow-up in finite time as it will be apparent in the proof.
\end{rem}

\begin{proof}[Proof of Theorem \ref{stabilitydiff}]
We restrict to the special case $C_\lambda=1$, all the other cases following by linearity of the hierarchy \eqref{hierarchydiff} and homogeneity of the norms \eqref{E-norm-hierarchy}. First, for a given $T\in \Tree$, notice that \eqref{hierarchydiff} implies
\begin{align*}
\frac{d}{dt}\int_{\R^{d\,|T|}} |h_T|^2\,dx_1\dots dx_{|T|}&= 2\,\sum_{i=1}^{|T|}\int_{\R^{d\,(|T|+1)}} \nabla_{x_i} h_{T}\,K(x_i-z)\,h_{T+i}\, dx_1\dots dx_{|T|} \,dz\\
&\quad-2\,\nu\,\sum_{i=1}^{|T|}\int_{\R^{d\,|T|}} |\nabla_{x_i}h_T|^2\,dx_1\dots dx_{|T|}.
\end{align*}
To justify the above formal calculation, we note that
\[
\partial_t h_T=\nu\Delta h_T-\divop j_T,
\]
where $j_T=\int_{\R^d}K(x_i-z)h_{T+i}\,dz$. From our assumption on $K$ and $h_{T+i}$, we immediately have that $j_T\in L^\infty([0,\ t_*],\,L^2)$. Namely, we obtain
  \[\Vert j_T(t,\cdot)\Vert_{L^2(\R^{d|T|})}\leq \Vert K\Vert_{L^2}\Vert h_{T+i}(t,\cdot)\Vert_{L^2(\R^{d(|T|+1)})}.
  \]
Standard properties of the heat kernel then prove that $h_T \in L^2([0,\ t_*],\, H^1(\mathbb{R}^{d|T|}))$.

By the Cauchy--Schwartz inequality and Young's inequality we further obtain
 \[
\frac{d}{dt}\|h_T(t,\cdot)\|_{L^2(\R^{d\,|T|})}^2\leq \frac{\|K\|_{L^2}^2}{2\nu}\, \sum_{i=1}^{|T|}\|h_{T+i}(t,\cdot)\|_{L^2(\R^{d\,(|T|+1)})}^2.
\]

For this reason, we introduce the family of intermediary norms
\[
\|h(t,\cdot)\|_{n}:=\sup_{|T|=n} \|h_T(t,\cdot)\|_{L^2(\R^{d\,|T|})},
\]
for any tree order $n\in \mathbb{N}$, which readily satisfies
\[
\frac{d}{dt} \|h(t,\cdot)\|_{n}^2\leq n \frac{\Vert K\Vert_{L^2}^2}{2\nu}\,\|h(t,\cdot)\|_{n+1}^2.
\]
By induction this estimate yields 
\[\begin{split}
\|h(t,\cdot)\|_{n}^2\leq &\left(\frac{\Vert K\Vert_{L^2}^2}{2\nu}\right)^{m-n}\,\int_{s}^t\,\frac{(m-1)!\,(t-r)^{m-n-1} }{(n-1)!\,(m-n-1)!}\, \|h(r,\cdot)\|_{m}^2\,dr\\
&+\sum_{k=n}^{m-1}  \left(\frac{\Vert K\Vert_{L^2}^2}{2\nu}\right)^{k-n}\, \frac{(k-1)!\,(t-s)^{k-n}}{(n-1)!\,(k-n)!}\, \|h(s,\cdot)\|_{k}^2.
\end{split}
\]
for any $m>n$. Using the assumption on $h$ one has 
\[
\|h(r,\cdot)\|^2_m\leq \lambda^{-m},
\]
for any $r\in [s,\ t]$,  recall that we have assumed that $ C_\lambda = 1 $. We then infer
\begin{equation}
  \begin{split}
\|h(t,\cdot)\|_{n}^2\leq &\left(\frac{\Vert K\Vert_{L^2}^2}{2\nu}\right)^{m-n}\,\binom{m-1}{n-1}\,(t-s)^{m-n}\, \lambda^{-m}\\
&+\sum_{k=n}^{m-1}  \left(\frac{\Vert K\Vert_{L^2}^2}{2\nu}\right)^{k-n}\,\binom{k-1}{n-1} \,(t-s)^{k-n}\, \|h(s,\cdot)\|_{k}^2.
\end{split}\label{intermediarycrazy}
\end{equation}
Obviously, \eqref{intermediarycrazy} can only provide smallness for short time intervals $(t-s)\sim \nu$,  if used directly. As we advanced in Remark \ref{R-no-semigroup}, we cannot bound a semi-group corresponding to the hierarchy for the norm $\|h(t,\cdot)\|_\lambda$ on all times. Instead we fix now the time step
\[
\delta:=\frac{2\theta\lambda\nu}{\Vert K\Vert_{L^2}^2},
\]
and consider the associated sequence of discrete times $t_i=\delta\,i$ with $i=1,\ldots,N$. Applying \eqref{intermediarycrazy} to $t=t_{i+1}$ and $s=t_i$, one obtains
\[
\sup_{t\in [t_i,\ t_{i+1}]}(\theta\lambda)^n\|h(t,\cdot)\|_{n}^2\leq \theta^m\,\binom{m-1}{n-1}+\sum_{k=n}^{m-1}  (\theta\lambda)^k\, \binom{k-1}{n-1}\, \|h(t_i,\cdot)\|_{k}^2.
\]
We now interpolate the last factors
\[
(\theta^{1/p}\lambda)^k\,\|h(t_i,\cdot)\|_{k}^2= \left((\theta\lambda)^k\,\|h(t_i,\cdot)\|_{k}^2\right)^{1/p}\, \left(\lambda^k\,\|h(t_i,\cdot)\|_{k}^2\right)^{1/p'}\leq \Vert h(t_i,\cdot)\Vert_{\theta\lambda}^{2/p}\,\Vert h(t_i,\cdot)\Vert_{\lambda}^{2/p'}\leq  \Vert h(t_i,\cdot)\Vert_{\theta\lambda}^{2/p},
\]
where $1/p+1/p'=1$ and we have used the assumption $\|h(t_i,\cdot)\|_{\lambda}\leq 1$. This yields
\[
\sup_{t\in [t_i,\ t_{i+1}]}(\theta\lambda)^n\,\|h(t,\cdot)\|_{n}^2\leq \theta^m\,\binom{m-1}{n-1}+\sum_{k=n}^{m-1}  \theta^{k/p'}\, \binom{k-1}{n-1}\, \|h(t_i,\cdot)\|_{\theta\lambda}^{2/p}.
\]
Since $\sum_{n=1}^m \binom{m-1}{n-1}=2^{m-1}$, we have that $\binom{m-1}{n-1}\leq 2^{m-1}$. Similarly, summing over $k$, we obtain
\[
\sup_{t\in [t_i,\ t_{i+1}]}(\theta\lambda)^n\,\|h(t,\cdot)\|_{n}^2\leq \frac{1}{2}(2\theta)^m+\frac{1}{2}\sum_{k=n}^{m-1} (2\theta^{1/p'})^k\, \|h(t_i,\cdot)\|_{\theta\lambda}^{2/p}\leq \frac{1}{2}(2\theta)^m+\frac{1}{2(1-2\theta^{1/p'})}\,\|h(t_i,\cdot)\|_{\theta\lambda}^{2/p},
\]
where we have used that $2\theta^{1/p'}<1$ by hypothesis to guarantee the summability of the last factor. Indeed, since $2\theta<1$ then taking limits $m\to \infty$ shows that
\[
\sup_{t\in [t_i,\ t_{i+1}]}\|h(t,\cdot)\|_{\theta\lambda}\leq \left(\frac{1}{2(1-2\theta^{1/p'})}\right)^{1/2}\, \|h(t_i,\cdot)\|_{\theta\lambda}^{1/p}.
\]
Then, for $t\in[t_i, t_{i+1}]$ and $C_p = \frac{1}{2(1 - 2\theta^{1/p'})}$, we have
\[
\begin{split}
\|h(t,\cdot)\|_{\theta \lambda} &\leq C_p^{\frac 1 2} \|h(t_i,\cdot)\|_{\theta \lambda}^{1/p} \leq C_p^{\frac 1 2} C_p^{\frac{1}{2p}} \|h(t_{i-1},\cdot)\|_{\theta \lambda}^{1/p^2} \leq C_p^{\frac 1 2} C_p^{\frac{1}{2p}} C_p^{\frac{1}{2p^2}} \|h(t_{i-2},\cdot)\|_{\theta \lambda}^{1/p^3}\\
& \leq \cdots \leq C_p^{\frac 1 2 \sum_{j=0}^{i-1} \frac{1}{p^j}} \|h(0,\cdot)\|_{\theta \lambda}^{1/p^i}.
\end{split}
\]
From this, by employing the fact that $\sum_{j=0}^{i-1} \frac{1}{p^j} = \frac{1}{1 - 1/p} = p'$
 we find
\[
\sup_{t\in [t_i,\ t_{i+1}]}\|h(t,\cdot)\|_{\theta\lambda}\leq \left(\frac{1}{2(1-2\theta^{1/p'})}\right)^{p'/2}\,\|h^0\|_{\theta\lambda}^{p^{-i}}.
\]
For any $t\in [0,\ t_*)$ set $i$ with $t\in [t_i,\ t_{i+1})$. Since $\Vert h^0\Vert_{\theta\lambda}\leq \Vert h^0\Vert_{\lambda}\leq 1$, the relation
\[
i\leq \frac{t}{\delta}=\frac{\Vert K\Vert_{L^2}^2}{2\theta\lambda\nu}\,t
\]
allows concluding the proof. \end{proof}

We emphasize that Theorem~\ref{stabilitydiff} only requires $K\in L^2$ instead of $K\in W^{1,\infty}$, but it only provides stability on the {\em viscous} hierarchy \eqref{hierarchydiff}. Indeed, we do not know if any similar result could hold on the non-viscous hierarchy ({\it i.e.} $\nu=0$). However Theorem~\ref{stabilitydiff} does imply a comparable stability result on the solutions $f$ of the starting graphon equation \eqref{independ2} without diffusion, provided that some added regularity is available on the $f$ or on $K$. The strategy is straightforward and consists in adding some artificial viscosity, which is performed in the next subsection. 
\subsection{Stability on the system without artificial diffusion}
Our stability result relies on stronger regularity for weak solutions $(w,f)$ to \eqref{independ2}. Before stating it, it is useful to observe that smoothness indeed propagates in time  by a similar argument as in Proposition \ref{existenceweak}.
%
\begin{lem}\label{regularityfi}
Consider any $K\in W^{1,1}$ with $\divop K\in  L^\infty$,  any $w\in L^\infty_\xi\mathcal{M}_\zeta\cap L^\infty_\zeta\mathcal{M}_\xi$ and any that $f^0\in L^\infty_\xi (L^1_x\cap L^\infty_x\cap H^1_x)$. Assume that $(w,f)$ is a weak solution to \eqref{independ2}. Then, $f\in L^\infty([0,\ t_*],\,L^\infty_\xi (L^1_x\cap L^\infty_x\cap H^1_x))$, for any $t_*<\infty$, and it satisfies
  \[
\|f(t,\cdot,\cdot)\|_{L^\infty_\xi (L^1_x\cap L^\infty_x\cap H^1_x)}\leq {C\,e^{e^{C\,t}-1}},
  \]
for some $C\in \mathbb{R}_+$ depending only on $\|w\|_{L^\infty_\xi \mathcal{M}_\zeta}$, $\|K\|_{W^{1,1}}$, $\|\divop K\|_{L^\infty}$ and $\|f^0\|_{L^\infty_\xi (L^1_x\cap L^\infty_x\cap H^1_x)}$.
\end{lem}

\begin{proof}
The following estimates can be made rigorous through a regularization process, for example adding viscosity as in the equation \eqref{independ2nu}, together with a classical iterative delay in time leading to a linear system. We first observe that by integration by parts we obtain the {\it a priori} bounds
\begin{align*}
\frac{d}{dt} \|f(t,\cdot,\xi)\|_{L^p_x}^p&=-(p-1)\int_0^1w(\xi,d\zeta)\int_{\R^{2d}}\divop K(x-y)\,f(t,x,\xi)^p f(t,y,\zeta)\,dx\,dy\\
&\leq (p-1)\,\Vert w\Vert_{L^\infty_\xi\mathcal{M}_\zeta}\,\|\divop K\|_{L^\infty}\,\| f(t,\cdot,\cdot)\|_{L^\infty_\xi L^1_x}\,\|f(t,\cdot,\xi)\|_{L^p_x}^p,
\end{align*}
for $t\in [0,\ t_*]$ and $\xi\in [0,\ 1]$, where above we have used again Lemma \ref{bilinearboundBx}. Since $f(t,\cdot,\xi)$ preserves mass for each $\xi\in [0,\ 1]$, one has that $\| f(t,\cdot,\cdot)\|_{L^\infty_\xi L^1_x}=\Vert f^0\Vert_{L^\infty_\xi L^1_x}$ for each $t\in [0,\ t_*]$ and this allows easily propagating any $L^\infty_\xi L^p_x$ bound of $f$ by Gronwall's inequality, leading to the following estimate
$$\Vert f(t,\cdot,\cdot)\Vert_{L^\infty_\xi L^p_x}\leq C_p\,e^{C_p\,t}.$$
Similarly, differentiating \eqref{independ} with respect to $x$ we now get
 \begin{align*}
    \partial_t \nabla_x f(t,x,\xi)&+\divop_{x}\left(\nabla_x f(t,x,\xi)\otimes \int_0^1 w(\xi,d\zeta)\,\int_{\R^d} K(x-y)\,f(t,y,\zeta)\,dy\right)\\
    & +\divop_{x}\left(f(t,x,\xi)\,\int_0^1 w(\xi,d\zeta)\,\int_{\R^d} \nabla K(x-y)^\top\, f(t,y,\zeta)\,dy\right)=0.
\end{align*}
Again, by integrating by parts we easily get the following decomposition
$$
\frac{d}{dt} \frac{1}{2}\|\nabla_x f(t,\cdot,\xi)\|_{L^2_x}^2=I_1+I_2+I_3,
$$
where each term takes the form
\begin{align*}
I_1&:=\frac{1}{2}\int_0^1w(\xi,d\zeta)\int_{\R^{2d}}\divop K(x-y)\,|\nabla_x f(t,x,\xi)|^2\,f(t,y,\zeta)\,dx\,dy,\\
I_2&:=-\int_0^1 w(\xi,d\zeta)\int_{\R^{2d}}\divop K(x-y)\,\nabla_x f(t,x,\xi) \cdot \nabla_y f(t,y,\zeta)\,f(t,x,\xi)\,dx\,dy,\\
I_3&:=-\int_0^1w(\xi,d\zeta)\int_{\R^{2d}} \nabla_x f(t,x,\xi)^\top\cdot \nabla K(x-y)\cdot\nabla_x f(t,x,\xi)\, f(t,y,\zeta)\,dx\,dy.
\end{align*}
By the hypothesis on $K$, we have
\begin{align*}
\vert I_1\vert&\leq \frac{1}{2}\Vert w\Vert_{L^\infty_\xi\mathcal{M}_\zeta}\,\Vert \divop K\Vert_{L^\infty}\,\Vert f(t,\cdot,\cdot)\Vert_{L^\infty_\xi L^1_x}\,\Vert \nabla_x f(t,\cdot,\xi)\Vert_{L^2_x}^2,\\
\vert I_2\vert &\leq \Vert w\Vert_{L^\infty_\xi\mathcal{M}_\zeta}\,\Vert \divop K\Vert_{L^1}\,\Vert f(t,\cdot,\cdot)\Vert_{L^\infty_\xi L^\infty_x}\,\Vert \nabla_x f(t,\cdot,\xi)\Vert_{L^2_x}^2,\\
\vert I_3\vert &\leq \Vert w\Vert_{L^\infty_\xi\mathcal{M}_\zeta}\,\Vert \nabla K\Vert_{L^1}\,\Vert f(t,\cdot,\cdot)\Vert_{L^\infty_\xi L^\infty_x}\,\Vert \nabla_x f(t,\cdot,\xi)\Vert_{L^2_x}^2,
\end{align*}
where we have used again Lemma \ref{bilinearboundBx}. Putting everything together, using the previous $L^p$ estimates and the fact that $\|f(t,\cdot,\cdot)\|_{L^\infty_{\xi} L^1_x} = \Vert f^0\Vert_{L^\infty_\xi L^1_x}$ for all $t\in [0,\ t_*]$, one obtains
\[
\frac{d}{dt}\Vert \nabla_x f(t,\cdot,\xi)\Vert_{L^2_x}^2\leq C\,e^{C\,t}\Vert \nabla_x f(t,\cdot,\xi)\Vert_{L^2_x}^2.
\]
Therefore, we conclude by means of Gronwall's lemma.
\end{proof}
 
With this additional regularity, in the context of the preceding subsections, and from Theorem \ref{stabilitydiff} and Proposition \ref{hierarchyeq}, we may derive the following uniqueness result.

We do note that in the following theorem, we use the $\tau(T,w,f)$ for generic
$w,\,\tilde w\in L^\infty_\xi\mathcal{M}_\zeta\cap L^\infty_\zeta\mathcal{M}_\xi$. While the operator $\tau$ is trivially well defined if $w,\,\tilde w\in L^\infty_\xi L^1_\zeta\cap L^\infty_\zeta L^1_\xi$, we rigorously justify the definition of $\tau(T,w,f)$ for kernels in $L^\infty_\xi\mathcal{M}_\zeta\cap L^\infty_\zeta\mathcal{M}_\xi$ in the next Section~\ref{sec:graphons} around Definition~\ref{D-tau-extended}. We refer to the more extensive discussion just after Lemma~\ref{L-tau-bound}.

\begin{theo}\label{stabilitynonviscous}
Consider any couple of weak solution to \eqref{independ2} $f,\;\tilde f\in L^\infty([0,\ t_*],\,L^\infty_\xi (L^1_x\cap L^\infty_x\cap H^1_x))$ with associated $w,\;\tilde w\in L^\infty_\xi\mathcal M_\zeta \cap L^\infty_\zeta \mathcal M_\xi$, for some $K\in L^\infty\cap W^{1,1}$ with $\divop K\in L^\infty$. Then, we have
  \[
\left\|  \int_0^1 ( f -\tilde f) (t,\cdot, \xi)\,d\xi\,\right\|_{L^2_x}\leq {\frac{C}{(\log |\log \|\tau(\cdot,w, f^0)-\tau(\cdot,\tilde w,\tilde f^0)\|_{\lambda}|)_+^{1/2}},}
  \]
for any $t\in [0,\ t_*]$, and some constants $C,\,\lambda>0$, which only depend on $t_*$, the various norms of $K$, the norm of $w,\,\tilde w$ in $L^\infty_\xi\mathcal M_\zeta \cap L^\infty_\zeta \mathcal M_\xi$ and the norm of the initial data $f^0,\,\tilde f^0$ in $L^\infty_\xi(L^1_x\cap L^\infty_x\cap H^1_x)$.
 \end{theo}
%
\begin{rem}
A direct $L^p$ estimate would only provide the following trivial estimate
  \[
\| f-\tilde f\|_{L^\infty_t L^\infty_\xi L^2_x} \leq  e^{Ct}\,\left(\| f^0-\tilde f^0\|_{L^\infty_\xi L^2_x}+C\, \|w-\tilde w\|_{L^\infty_\xi {\mathcal M}_\zeta} \right).
\]
Note that such estimate would require a very precise strong control on the difference $w-\tilde w$, which is not available to us. However, as we will see in the next section, we can obtain good estimates on the weaker objects $\tau(T, w,  f^0)-\tau(T, \tilde w, \tilde f^0)$.
  \end{rem}

\begin{rem}
  If we assume more regularity on $K$, for example $K\in W^{1,\infty}$, then it is possible to obtain stability with instead less regularity on $f$ and $\tilde f$ (typically $f^0\in L^\infty_\xi L^\infty_x$, for example). 
  \end{rem}

Before turning to the proof of Theorem~\ref{stabilitynonviscous}, let us emphasize some of its main consequences. The key point is that it offers a direct way to obtain some form of compactness on our solutions. More specifically,  we can extract strong convergence in $L^2$ on each of the $\tau(T,w_N,f_N^0)$ through a classical diagonal extraction process since each of them is bounded in $L^1\cap W^{1,\infty}$ by Corollary~\ref{cor-L-tau-bound}. From Theorem~\ref{stabilitynonviscous}, this implies that $\int_0^1 f_N(t,x,\xi)\,d\xi$ is strongly compact in $L^2_x$.  However, we are still not able to identify the corresponding limit in terms of a solution to \eqref{independ2}. This will require an appropriate extension of graphons that is fully developed in the next section.

\begin{proof}[Proof of Theorem \ref{stabilitynonviscous}]
For some $\nu>0$, let us consider the solution $f^\nu$ to the system \eqref{independ2nu} with artificial diffusion with the same initial data $f^0$ and same weights $w$ as $f$. We define similarly $\tilde f^\nu$ solution to \eqref{independ2nu} with the same initial data $\tilde f^0$ and weights $\tilde w$ as $\tilde f$. The existence of both $f^\nu$ and $\tilde f^\nu$ is provided by  Proposition \ref{existenceweak}.

The strategy of the proof is classical: First use the regularity of $f$ (resp. $\tilde f$) to compare them with $f^\nu$ (resp. $\tilde f^\nu$). Then use Theorem~\ref{stabilitydiff} to compare the observables from $f^\nu$ and $\tilde f^\nu$. To compare $f$ with $f^\nu$, we first derive $H^1$ estimates for $f^\nu$ in the usual way
\begin{align*}
\frac{d}{dt} \frac{1}{2}\,\| f^\nu(t, \cdot, \xi)\|_{L^2_x}^2 &=  -\nu\,\int_{\R^d} |\nabla_x  f^\nu (t, x, \xi)|^2\,dx\\ 
&\quad + \int_0^1 \ w(\xi, d\zeta)\int_{\R^{2d}}  f^\nu (t, x, \xi)\,\nabla  f^\nu(t, x, \xi)\cdot K(x-y) \, f^\nu(t, y, \zeta)\,dy\,dx.
\end{align*}
Integrating by parts in the second term yields
\begin{align*}
   \frac{d}{dt} \frac{1}{2}\,\| f^\nu(t, \cdot, \xi)\|_{L^2_x}^2 &+{\nu}\,\int_{\R^d} |\nabla_x  f^\nu (t, x, \xi)|^2\,dx \\ 
   & =-\frac{1}{2}\,\int_0^1w(\xi, d\zeta)\int_{\R^{2d}} \divop K(x-y)\,| f^\nu (t, x, \xi)|^2\, f^\nu(t,y, \zeta)\,dy\,dx\\
   &\leq \frac{1}{2}\, \| w\|_{L^\infty_\xi\mathcal{M}_\zeta}\,\|\divop K\|_{L^\infty}  \| f^\nu (t, \cdot, \cdot)\|_{L^\infty_\xi L^1_x} \,\| f^\nu(t,\cdot, \xi)\|_{L^2_x}^2,
\end{align*}
where we have used Lemma \ref{bilinearboundBx} in the last step. Again, notice that  $\| f^\nu(t,\cdot,\cdot)\|_{L^\infty_{\xi} L^1_x}=\Vert f^0\Vert_{L^\infty_\xi L^1_x}$ for all $t\in [0,\ t_*]$. Using Gronwall's lemma we obtain
   \begin{equation}
  \|f^\nu(t,\cdot,\cdot)\|_{L^\infty_\xi L^2_x}^2+\nu\,\int_0^t \int_{\R^d} |\nabla_x  f^\nu (s, x, \xi)|^2\,dx\,ds\leq  \| f^0\|_{L^\infty_\xi L^2_x }^2\,e^{\frac{t}{2}\, \| w\|_{L^\infty_\xi \mathcal{M}_\zeta} \|\divop K\|_{L^\infty}\,\Vert f^0\Vert_{L^\infty_\xi L^1_x}}.\label{H1finu}
     \end{equation}
Now, observe that the difference $f^\nu-f$ satisfies the following equation
  \[
  \begin{split}
    \partial_t ( f^\nu- f)(t, x, \xi)&+\divop_{x}\left( ( f^\nu- f)(t, x, \xi)\,\int_0^1 w(\xi, d\zeta) \int_{\R^d} K(x-y) \, f^\nu(t,y, \zeta)\,dy\right)\\
    & +\divop_{x}\left(  f (t, x, \xi)\int_0^1 w(\xi, d\zeta) \int_{\R^d} K(x-y)\, \left( f^\nu- f\right)(t, y, \zeta)\,dy \right)\\ &-\nu\,\Delta_x  ( f^\nu- f)(t, x, \xi)=\nu\,\Delta_x  f (t, x, \xi).
\end{split}
  \]
Therefore, similar arguments as above lead to analogous $H^1$ estimates 
 \begin{align*}
    &\hspace{-1cm}\frac{d}{dt} \frac{1}{2}\left\| (f^\nu - f)(t, \cdot, \xi)\right\|_{L^2_x}^2 +\nu\,\int_{\R^d} |\nabla_x  (f^\nu- f) (t, x, \xi)|^2\,dx\\
    &\leq \frac{1}{2}\Vert w\Vert_{L^\infty_\xi\mathcal{M}_\zeta}\,\|\divop K\|_{L^\infty}\,\Vert f^\nu(t,\cdot,\cdot)\Vert_{L^\infty_\xi L^1_x}\,\|(f^\nu- f)(t, \cdot, \xi)\|_{L^2_x}\\
    & \quad + \| w\|_{L^\infty_\xi \mathcal{M}_\zeta}   \|K\|_{L^2}\,\|f (t, \cdot,\cdot)\|_{L^\infty_\xi H^1_x } \,\|(f^\nu- f)(t, \cdot, \xi)\|_{L^2_x}\,\|(f^\nu- f)(t, \cdot,\cdot)\|_{L^\infty_\xi L^2_x}\\
    & \quad + \| w \|_{L^\infty_\xi \mathcal{M}_\zeta}   \|\divop K\|_{L^1} \,\| f (t, \cdot,\cdot)\|_{L^\infty_\xi L^\infty_x}\,\|(f^\nu- f)(t, \cdot, \xi)\|_{L^2_x}\,\|(f^\nu- f)(t, \cdot,\cdot)\|_{L^\infty_\xi L^2_x }\\
    &\quad +\nu\,\|\nabla_x(  f^\nu- f) (t, \cdot,\xi)\|_{L^2_x}\,\|  f (t, \cdot,\cdot)\|_{L^\infty_\xi H^1_x }.
\end{align*}
Using the additional regularity $f,\,\tilde f\in L^\infty([0,\ t_*],\,L^\infty_\xi(L^1_x\cap L^\infty_x\cap H^1_x))$ from Lemma \ref{regularityfi}, applying Young's inequality in the last term, and recalling that $ f(0,\cdot,\cdot)=f^\nu(0,\cdot,\cdot)$ we obtain by Gronwall's lemma that
  \begin{equation}\label{L2fifinu}
\| (f^\nu - f)(t,\cdot,\cdot)\|_{L^\infty_\xi L^2_x} \leq C(t)\,\sqrt{\nu},
  \end{equation}
for some continuous and non-decreasing function $C=C(t)\in \R_+$ that only depends on the various norms of $K$, the norm of $w$ in $L^\infty_\xi\mathcal{M}_\zeta$ and the norm of the initial datum $f^0$ in $L^\infty_\xi(L^1_x\cap L^\infty_x\cap H^1_x)$. Note that a similar estimate can be obtained for $\tilde f^\nu -\tilde f$. The function $C$ is actually a triple exponential. 

By Proposition \ref{hierarchyeq}, the observables $\tau(T, w, f^\nu)$ and $\tau(T, \tilde w,\tilde f^\nu)$ with $T\in \Tree$ both solve the same hierarchy \eqref{hierarchydiff} (since $w,\,\tilde{w}$ do not appear explicitly in that formulation). Therefore, $h=(h_T)_{T\in \Tree}$, with $h_T=\tau(T, w, f^\nu) - \tau(T, \tilde w,\tilde f^\nu)$, again solves \eqref{hierarchydiff} by linearity. In addition, Corollary~\ref{cor-L-tau-bound} implies that
\begin{align*}
\Vert h_T(t,\cdot,\cdot)\Vert_{L^2(\R^{d |T|})}&\leq \Vert \tau(T,w,f^\nu(t,\cdot,\cdot))\Vert_{L^2(\R^{d |T|})}+\Vert \tau(T,w,f^\nu(t,\cdot,\cdot))\Vert_{L^2(\R^{d |T|})}\\
&\leq \Vert w\Vert_{L^\infty_\xi\mathcal{M}_\zeta}^{|T|-1}\,\Vert f^\nu(t,\cdot,\cdot)\Vert_{L^\infty_\xi L^2_x}^{|T|}+\Vert \tilde w\Vert_{L^\infty_\xi\mathcal{M}_\zeta}^{|T|-1}\,\Vert \tilde f^\nu(t,\cdot,\cdot)\Vert_{L^\infty_\xi L^2_x}^{|T|}\\
&\leq \frac{2}{w_{\text{min}}}\left(w_{\text{max}}\frac{\Vert f^0 \Vert_{L^\infty_\xi L^2_x}+\Vert \tilde f^0\Vert_{L^\infty_\xi L^2_x}}{2}e^{\frac{t_*}{4}w_{\text{max}}\Vert \divop K\Vert_{L^\infty}\,\Vert f^0\Vert_{L^\infty_\xi L^1_x}}\right)^{|T|},
\end{align*}
where $w_{\text{min}}=\min\{\Vert w\Vert_{L^\infty_\xi\mathcal{M}_\zeta},\Vert \tilde w\Vert_{L^\infty_\xi\mathcal{M}_\zeta}\}$ and $w_{\text{max}}=\max\{\Vert w\Vert_{L^\infty_\xi\mathcal{M}_\zeta},\Vert \tilde w\Vert_{L^\infty_\xi\mathcal{M}_\zeta}\}$. Consequently, we obtain that $\sup_{t\in [0,\ t_*]}\|h(t,\cdot)\|_\lambda<1$, for any $\lambda>0$ such that
\[
\sqrt{\lambda}<\min\left\{\frac{2}{w_{\text{max}}},\frac{w_{\text{min}}}{w_{\text{max}}}\right\}\frac{e^{\frac{-t_*}{4}w_{\text{max}}\Vert \divop K\Vert_{L^\infty}\,\Vert f^0\Vert_{L^\infty_\xi L^1_x}}}{(\Vert f^0 \Vert_{L^\infty_\xi L^2_x}+\Vert \tilde f^0\Vert_{L^\infty_\xi L^2_x})}.
\]
  We can now apply Theorem~\ref{stabilitydiff} for this $\lambda$, any $p>1$ and $\theta\in (0,2^{-p'})$, and obtain
  \[
\sup_{t\in [0,\ t_*]}\|h(t,\cdot,\cdot)\|_{\theta\lambda}\leq C_{p,\theta}\,\exp\left(p^{-\frac{\Vert K\Vert_{L^2}^2}{2\theta\lambda\nu}\,t}\log \Vert h^0\Vert_{\theta\lambda}\right),
  \]
for any $T\in \Tree$ and some constant $C_{p,\theta}>0$ depending only on $p$ and $\theta$. In the special case of the trivial tree $T_1\in \Tree_1$ with only one vertex, we obtain 
  \[
  \|h\|_{\theta\lambda}\geq \lambda \|h_{T_1}\|_{L^2}=\sqrt{\lambda} \left\|\tau(T_1,w, f^\nu)-\tau(T_1,\tilde w,\tilde f^\nu)\right\|_{L^2_x}=\sqrt{\lambda}\,\left\| \int_0^1 ( f^\nu -\tilde f^\nu) (t,\cdot, \xi)\,d\xi\,\right\|_{L^2_x}.
  \]
  Therefore,  we deduce that
  \begin{equation}
\left\| \int_0^1 ( f^\nu -\tilde f^\nu) (t,\cdot, \xi)\,d\xi\,\right\|_{L^2_x} \leq \frac{C_{p,\theta}}{\sqrt{\lambda}}\,\exp\left(p^{-\frac{\Vert K\Vert_{L^2}^2}{2\theta\lambda\nu}\,t}\, \log \|h^0\|_{\theta\lambda}\right),\label{difftaunu}
  \end{equation}
  for any $t\in [0,\ t_*]$. By \eqref{L2fifinu} and Minkowski's integral inequality, we also have that
  \[
\left\| \int_0^1 ( f^\nu - f) (t,\cdot, \xi)\,d\xi\,\right\|_{L^2_x}  \leq  \| f^\nu- f \|_{L^\infty_\xi L^2 _x}\leq C(t)\,\sqrt{\nu},  
  \]
and similarly for $\tilde f^\nu-\tilde f$, where the function $C=C(t)$ is given above. This combined with \eqref{difftaunu} immediately yields that
\[
\left\| \int_0^1 ( f -\tilde f) (t,\cdot, \xi)\,d\xi\,\right\|_{L^2_x}  \leq \frac{C_{p,\theta}}{\sqrt{\lambda}}\,\exp\left(p^{-\frac{\Vert K\Vert_{L^2}^2}{2\theta\lambda\nu}\,t}\, \log \|h^0\|_{\theta\lambda}\right)+C(t)\,\sqrt{\nu}.
\]
Since the above is valid for any $\nu>0$, it only remains to optimize in $\nu$ appropriately. For example, considering $p=e$ for simplicity and setting
\[
\nu=\frac{\Vert K\Vert_{L^2}^2}{\theta\lambda}\,t\,(\log |\log \Vert h^0\Vert_{\theta\lambda}|)_+^{-1},
\]
ends the proof.
  \end{proof}

 \section{Extending graphons}\label{sec:graphons}
 \subsection{Our goals}
    This section is centered on the process of taking limits on the $\tau(T,w_N,f_N)$ defined by the operator \eqref{tau} for the pairs $(w_N,f_N)$ given by Definition \ref{D-graphon-representation}, at some fixed time $t\in \mathbb{R}_+$. Since the time variable plays no role in this part, we omit it from the calculations. The first aim of this section is to analyze the following question already raised in the previous section: If $\lim_{N\rightarrow \infty} \tau(T,w_N,f_N)$ exists for any $T\in \Tree$, can we find appropriate $w$ and $\,f$ such that we recover the representation
    \[
\tau(T,w,f)=\lim_{N\rightarrow \infty} \tau(T,w_N,f_N), \quad \forall\,T\in \Tree ?
\]
As we will see below, this can be rephrased as the question of how to take appropriate limits of $w_N$ and $f_N$ so that they allow passing to the limit in $\tau(T,w_N,f_N)$. The following result answers this question.

\begin{theo}\label{lg}
The definition of $\tau(T,w,f)$ can be uniquely extended for any $w\in L^\infty_\zeta\mathcal{M}_\xi \cap L^\infty_\xi\mathcal{M}_\zeta$ and $f\in L^\infty_{\xi}L^\infty_x$. Furthermore, consider any sequence $\{w_N\}_{N\in \mathbb{N}}$ and $\{f_N\}_{N\in \mathbb{N}}$ such that the following hypothesis hold true
    \begin{enumerate}[label=(\roman*)]
    \item $\quad \displaystyle \sup_{N\in \mathbb{N}} \sup_{\xi\in [0,\ 1]} \int_0^1 |w_N(\xi,\zeta)|\,d\zeta\, <\infty, \quad \sup_{N\in \mathbb{N}} \sup_{\zeta\in [0,\ 1]} \int_0^1 |w_N(\xi,\zeta)|\,d\xi\, <\infty,$\\
    \item $\quad \displaystyle\sup_{N\in \mathbb{N}} \|f_N\|_{L^\infty_\xi (W^{1,1}_x\cap W^{1,\infty}_x)}<\infty$.\\
   \end{enumerate}
 Then, there exists an extracted subsequence of $N$ (that we still denote as $N$ for simplicity) and $w\in L^\infty_\zeta\mathcal{M}_\xi \cap L^\infty_\xi\mathcal{M}_\zeta$  together with $f\in L^\infty_\xi (W^{1,1}_x\cap W^{1,\infty}_x)$ such that
    \[
\tau(T,w_N,f_N)\rightarrow \tau(T,w,f)\quad \mbox{in}\quad L^p_{loc}(\mathbb{R}^{d\,|T|}),
    \]
as $N\rightarrow \infty$, for each $T\in \Tree$ and any $1\leq p<\infty$.
  \end{theo}
  
We note that Theorem \ref{lg} only imposes that the sequences $\{w_N\}_{N\in \mathbb{N}}$ and $\{f_N\}_{N\in \mathbb{N}}$  fulfill the above uniform estimates $(i)$ and $(ii)$. Those sequences can, hence, be more general than the empirical graphons in Definition \ref{D-graphon-representation} associated to discrete objects $(w_{ij})_{i,j=1,\ldots,N}$ and $(\bar f_i)_{i=1,\ldots,N}\subseteq \mathcal{P}(\mathbb{R}^d)$. However, when restricting to empirical $w_N$ and $f_N$, note that we have equivalently the following uniform bounds
\begin{enumerate}[label=$(\roman*')$]
\item $\quad \displaystyle \sup_{N\in \mathbb{N}}\max_{1\leq i\leq N}\sum_{j=1}^N \vert w_{ij}\vert<\,\infty,\quad \sup_{N\in \mathbb{N}}\max_{1\leq j\leq N}\sum_{i=1}^N \vert w_{ij}\vert<\,\infty,$\\
\item $\quad \displaystyle\sup_{N\in \mathbb{N}}\max_{1\leq i\leq N}\Vert \bar f_i\Vert_{W^{1,\infty}_x}<\infty.$
\end{enumerate}

Let us remark some critical differences  of our generalized graphons with standard graphons:
    \begin{itemize}
  \item As discussed in Section \ref{sec:introduction}, the notion of graphon arises naturally under the scaling condition \eqref{densegraph}, {\it i.e.}, $w_{ij}\lesssim 1/N$, since it leads to $w_N$ uniformly bounded in $L^\infty_{\xi,\zeta}$.
  Compactness is then obtained by considering $w_N$ as the kernel of an operator from $L^1$ to $L^\infty$.  This corresponds to the natural topology on graphons characterized by the {\it cut metric}, which can be reframed in terms of the homomorphism densities $\tau(G,w)$ in Lemma \ref{L-tau-density-bound} for general simple graphs $G$, see the pioneer works of Lov\'asz and  Lov\'asz-Szegedy \cite{Lo,LS}.
  \item Instead, in this work, we focus on the generalized mean-field scaling given by \eqref{meanfieldscaling}-\eqref{vanishingweights}. This leads to $w_N$ which are not uniformly bounded in $L^\infty_{\xi,\zeta}$, but only in $L^\infty_{\xi} \mathcal{M}_{\zeta} \cap L^\infty_{\zeta} \mathcal{M}_{\xi}$.  As we mentioned in Section \ref{sec:introduction}, graphons could still be able to capture graph limits in the absence of the scaling $w_{ij}\lesssim 1/N$ by suitably renormalizing the weights, but these results do not appear to be able to handle such a general scaling assumption as we have here. Our scaling still allows to consider $w_N$ as the kernel of an operator, for example from $L^\infty$ to $L^\infty$ and from $L^1$ to $L^1$ as it has been depicted in Lemma \ref{bilinearbound}. But the connection with the cut metric and the natural compactness that derives from it are lost. Instead Theorem \ref{lg} suggests an alternative natural topology on pairs $(w,f)$ with $w\in L^\infty_\xi\mathcal{M}_\zeta\cap L^\infty_\zeta\mathcal{M}_\xi$ and $f\in L^\infty_\xi L^\infty_x$.
\item One straightforward result that highlights these differences is the following: As we have seen in the previous section ({\it cf.} Lemma \ref{L-tau-bound}), we have
\[
\|\tau(T,w,f)\|_{L^\infty}\leq \|w\|_{ L^\infty_\xi\mathcal{M}_\zeta}^{|T|-1}\,\|f\|_{L^\infty_{\xi}L^\infty_x}^{|T|},
\]
for any tree $T$ and any $w\in L^\infty_\xi\mathcal{M}_\zeta$, while this estimate is obviously false if we consider $\tau(G,w)$ for general simple graphs $G$ as it is done in the standard theory of graphons.
    \end{itemize}

Of direct impact to our work here is the fact that we cannot easily define $\tau(T,w,f)$ through formula \eqref{tau} in Definition \ref{D-tau} when we only have $w\in L^\infty_{\xi} \mathcal{M}_{\zeta} \cap L^\infty_{\zeta} \mathcal{M}_{\xi}$ and $f\in L^\infty_{\xi}L^\infty_x$. We explain how this can be done in the next subsection through the construction of an appropriate algebra. The rest of the section is devoted to the proof of Theorem~\ref{lg} which is based on obtaining simultaneous compactness on all elements of the algebra.

\subsection{Defining $\tau(T,w,f)$ for unbounded $w$}
The key to the definition of $\tau(T,w,f)$ (and later of the compactness argument) for only $f\in L^\infty_{\xi}(L^1_x\cap L^\infty_x)$ and $w\in L^\infty_\xi  \mathcal{M}_\zeta\cap L^\infty_\zeta  \mathcal{M}_\xi$, is an iterative way to construct it from the leaves of the tree. This leads us to the following definition.

\begin{defi}[A countable algebra] \label{defalgebra}
We will denote by $\mathcal{T}$ the countable algebra of transforms over spaces of arbitrarily large dimensions which is built as follows: For each transform $F\in {\mathcal T}$ there exists $n\in \mathbb{N}$ (the rank of $F$) so that $F$ maps each couple $(w,f)$ into a scalar function $F(w,f)$ on $[0,\ 1]\times \R^{dn}$. The full algebra $\mathcal{T}$ is obtained in a recursive way according to the following rules:
\begin{enumerate}[label=(\roman*)]
\item The elementary $1$-rank transform $F_0:\,(w,f)\mapsto f$ belongs to the algebra $\mathcal{T}$.
\item Let $F_1\in \mathcal{T}$ and $F_2\in \mathcal{T}$ be $n_1$-rank and $n_2$-rank transforms respectively. Then, the following $(n_1+n_2)$-rank transform also belongs to $\mathcal{T}$:
\[
F_1\otimes F_2:\,(w,f)\mapsto F_1(w,f)(\xi,x_1,\ldots,x_{n_1})\,F_2(w,f)(\xi,x_{n_1+1},\ldots,x_{n_1+n_2}).
\]
\item Let $F\in \mathcal{T}$ be a $n$-rank transform. Then, the following $n$-rank transform also belongs to $\mathcal{T}$:
\[
F^\star:\,(w,f)\mapsto \int_0^1 F(w,f)(\zeta,x_1,\ldots,x_n)\,w(\xi,d\zeta).
\]
\end{enumerate}
\end{defi}
For a given choice of $(w,f)$, we also denote by $M(w,f)$ the countable algebra consisting of functions over spaces of arbitrary large dimensions given by $F(w,f)$ for any transform $F\in\mathcal{T}$. Of course $M(w,f)$ can also be obtained directly by transposing the rules $(i)$, $(ii)$ and $(iii)$ above:
\begin{enumerate}[label=$(\roman*)$]
\item We have $f(x_1,\xi)\in M(w,f)$.
\item If $\phi(\xi,x_1,\ldots,x_{n_1}),\,\psi(\xi,x_1,\ldots, x_{n_2})\in M(w,f)$, then we have:
\[
\phi(\xi,x_1,\ldots,x_{n_1})\,\psi(\xi,x_{n_1+1},\ldots,x_{n_1+n_2})\in M(w,f).
\]
\item If $\phi(\xi,x_1,\ldots,x_n)\in M(w,f)$, then we have:
\[
\int_0^1 \phi(\zeta,x_1,\ldots,x_n)\,w(\xi,d\zeta)\in M(w,f).
\]
\end{enumerate}
The first point is to observe that $M(w,f)$ (and thus the algebra $\mathcal{T}$) is well defined with the only hypotheses $f\in L^\infty_{\xi}(L^1_x\cap L^\infty_x)$ and $w\in L^\infty_\xi \mathcal{M}_\zeta\cap L^\infty_\zeta  \mathcal{M}_\xi$.

\begin{lem}
  Consider any $f\in L^\infty_{\xi}(L^1_x\cap L^\infty_x)$ and any $w\in L^\infty_\xi  \mathcal{M}_\zeta\cap L^\infty_\zeta  \mathcal{M}_\xi$. Then all functions of $M(w,f)$ are well defined and belong to $L^\infty_\xi(L^1_{x_1,\ldots,x_n}\cap L^\infty_{x_1,\ldots,x_n})$, for some $n\in \mathbb{N}$. Moreover, 
\[
F(w_N,f_N)\rightarrow F(w,f)\quad \mbox{in}\quad L^1([0,\ 1]\times \R^{dn}),
\]
an $N\rightarrow\infty$ for any fixed $F\in \mathcal{T}$, any sequence $\{f_N\}_{N\in \mathbb{N}}$ uniformly bounded in $L^\infty_\xi (L^1_x\cap L^\infty_x)$ and converging to some $f$ in $L^\infty_{\xi}(L^1_x \cap L^\infty_x)$, and any sequence of $\{w_N\}_{N\in \mathbb{N}}$ uniformly bounded in $L^\infty_\xi  \mathcal{M}_\zeta\cap L^\infty_\zeta  \mathcal{M}_\xi$ and converging to $w$ in $L^1_\xi H^{-1}_\zeta\cap L^1_\zeta H^{-1}_\xi$. \label{welldefinedM}
\end{lem}
\begin{proof}
We of course use an induction argument based on the recursive rules above defining the countable algebra $M(w,f)$.

\medskip

$\diamond$ {\it Step 1: Good definition of elements of $M(w,f)$}.

\medskip

Obviously, for the elementary transform $F_0$ in rule $(i)$ we have that $F_0(w,f)=f$ is well defined and belongs to $L^\infty([0,\ 1]\times \mathbb{R}^d)$ by hypothesis. The second rule $(ii)$ also poses no issue. Indeed, let us set $F=F_1\otimes F_2$  and assume $F_1,F_2\in \mathcal{T}$ are transform with rank $n_1,n_2\in \mathbb{N}$ for which we already now that $F_1(w,f)$ and $F_2(w,f)$ are well defined and belong to $L^\infty([0,\ 1]\times \mathbb{R}^{dn_1})$ and $L^\infty([0,\ 1]\times \mathbb{R}^{dn_2})$ respectively. Then, $F(w,f)$ is well defined and belongs to $L^\infty([0,\ 1]\times \mathbb{R}^{d(n_1+n_2)})$ as the product of two bounded functions. Finally, it only remains to check the third rule $(iii)$. Consider any $F\in \mathcal{T}$ of rank $n\in \mathbb{N}$ for which we already now that $F(w,f)$ is well defined and belongs $L^\infty([0,\ 1]\times \mathbb{R}^{dn})$. Then, $F^\star (w,f)$ is also well defined and belongs to $L^\infty([0,\ 1]\times \mathbb{R}^{dn})$ by the estimate $\eqref{wfL1LinftyBx}_2$ in Lemma \ref{bilinearboundBx}.

Moreover, since $L^1$ norms are also stable under tensor products, by the same argument we also have that if $f\in L^\infty_\xi(L^1_x\cap L^\infty_x)$ and $w\in L^\infty_\xi\mathcal{M}_\zeta\cap L^\infty_\zeta\mathcal{M}_\xi$, then all functions in $M(w,f)$ are well defined and belong to $L^\infty_\xi (L^1_{x_1,\ldots, x_n}\cap L^\infty_{x_1,\ldots,x_n})$ for some $n\in \mathbb{N}$.

\medskip

$\diamond$ {\it Step 2: Convergence of elements of $M(w_N,f_N)$.}

\medskip

First note that taking the elementary transform $F_0\in \mathcal{T}$ in rule $(i)$ we have that 
\[
F_0(w_N,f_N)=f_N\rightarrow f=F_0(w,f)\quad \mbox{in}\quad L^1([0,\ 1]\times \mathbb{R}^d),
\] 
by the convergence hypothesis on $\{f_N\}_{N\in \mathbb{N}}$. Let us set $F=F_1\otimes F_2$  as in the rule $(ii)$ for a couple of transforms $F_1,F_2\in \mathcal{T}$ and assume that we already now that $F_1(w_N,f_N)$ and $F_2(w_N,f_N)$ are both convergent in $L^1$ and uniformly bounded in $L^\infty$. Therefore, it is clear that
\[
F(w_N,f_N)\rightarrow F(w,f)\quad \mbox{in}\quad  L^1([0,\ 1]\times \mathbb{R}^{d(n_1+n_2)}),
\]
and it is uniformly bounded in $L^\infty$ as the product of uniformly bounded sequences in $L^\infty$, each of them convergent in $L^1$. Finally, consider a $n$-rank transform $F\in \mathcal{T}$ and assume that we already now that $F(w_N,f_N)$ converges to $F(w,f)$ in $L^1$ and is uniformly bounded in $L^\infty$. Our last goal is to show that $F^\star(w_N,f_N)$ given by the rule $(iii)$ is also convergent in $L^1$ and uniformly bounded in $L^\infty$. Note that Lemma \ref{bilinearbound} yields a partial answer in the weak topology of $L^\infty$ though. We shall improve it here at the expense of the stronger convergence assumed on $\{w_N\}_{N\in \mathbb{N}}$. For simplicity of notation, we set $\phi_N:=F(w_N,f_N)$, which by the induction hypothesis converges to $\phi:=F(w,f)$ in $L^1$ and uniformly bounded in $L^\infty$. Then, we find that
\[
F^\star(w_N,f_N)(\xi,x_1,\ldots,x_n)=\int_0^1 \phi_N(\zeta,x_1,\ldots,x_N)\,w_N(\xi,d\zeta),
\]
is uniformly bounded in $L^\infty$ by estimate $\eqref{wfL1LinftyBx}_2$ in Lemma \ref{bilinearboundBx} thanks to the uniform bound of $\{w_N\}_{N\in \mathbb{N}}$ in $L^\infty_\xi\mathcal{M}_\zeta$ and of $\phi_N$ in $L^\infty$. Regarding convergence, we have
\[
F^\star(w_N,f_N)-F^\star(w,f)=I_N^1+I_N^2,
\]
where each factor reads
\begin{align*}
I_N^1(\xi,x_1,\ldots,x_n)&:=\int_0^1 (\phi_N-\phi)(\zeta,x_1,\ldots,x_n)\,w_N(\xi,d\zeta),\\
I_N^2(\xi,x_1,\ldots,x_n)&:=\int_0^1 \phi(\zeta,x_1,\ldots,x_n)\,(w_N(\xi,d\zeta)-w(\xi,d\zeta)).
\end{align*}
for $N\in \mathbb{N}$. On the one hand, it is clear by estimate $\eqref{wfL1LinftyBx}_1$ in Lemma \ref{bilinearboundBx} that we have
\[
\Vert I_N^1\Vert_{L^1}\leq \|w_N\|_{L^\infty_\zeta  \mathcal{M}_\xi}\,\|\phi_N-\phi\|_{L^1},
\]
and, therefore, $I_N^1\rightarrow 0$ in $L^1$ by the convergence of $\phi_N$ to $\phi$ in $L^1$ and the uniform estimate of $w_N$ in $L^\infty_\zeta\mathcal{M}_\xi$. On the other hand, consider any sequence $\{\phi_\varepsilon\}_{\varepsilon>0}\subset C^\infty_c([0,\ 1]\times \mathbb{R}^{dn})$ such that $\phi_\varepsilon\rightarrow \phi$ in $L^1$ as $\varepsilon\rightarrow 0$ and define
\[
I_{N,\varepsilon}^2(\xi,x_1,\ldots,x_n):=\int_0^1 \phi_\varepsilon(\zeta,x_1,\ldots,x_n)\,(w_N(\xi,d\zeta)-w(\xi,d\zeta)),
\]
for $N\in \mathbb{N}$ and $\varepsilon>0$. Then, by $\eqref{wfL1LinftyBx}_1$ in Lemma \ref{bilinearboundBx} we have
\begin{align*}
\Vert I_N^2\Vert_{L^1}&\leq \Vert w_N\Vert_{L^\infty_\zeta\mathcal{M}_\xi}\Vert \phi-\phi_\varepsilon\Vert_{L^1}+\Vert  I_{N,\varepsilon}^2\Vert_{L^1}\\
&\leq \Vert w_N\Vert_{L^\infty_\zeta\mathcal{M}_\xi}\Vert \phi-\phi_\varepsilon\Vert_{L^1}+\Vert \phi_\varepsilon\Vert_{L^1_{x_1,\ldots,x_n}H^1_\zeta}\Vert w_N-w\Vert_{L^1_\xi H^{-1}_\zeta},
\end{align*}
for every $N\in \mathbb{N}$ and $\varepsilon>0$. Taking $\limsup$ as $N\rightarrow \infty$ we have that
\[
\limsup_{N\rightarrow 0}\Vert  I_N^2\Vert_{L^1}\leq \sup_{N\in \mathbb{N}}\Vert w_N\Vert_{L^\infty_\zeta \mathcal{M}_\xi}\Vert \phi_\varepsilon-\phi\Vert_{L^1}.
\]
Since $\varepsilon>0$ is arbitrary, taking $\varepsilon\rightarrow 0$ and recalling that $w_N$ is uniformly bounded in $L^\infty_\zeta\mathcal{M}_\xi$ and $\phi_\varepsilon\rightarrow \phi$ in $L^1$ allow concluding.
\end{proof}
%
The critical reason for the introduction of the algebras $M(w,f)$ and $\mathcal{T}$ is that they allow easily recovering all $\tau(T,w,f)$, at least for bounded $w$. Namely, for any tree $T\in \Tree$, there exists a transform $F\in \mathcal{T}$ so that $\tau(T,w,f)$ can be described in terms of the transform $F(w,f)$ of the pair $(w,f)$.
%
\begin{lem}
For any  tree $T\in \Tree$ there exists a transform $F\in {\mathcal T}$ such that
\[
\tau(T,w,f)(x_1,\ldots,x_{|T|})=\int_0^1 F(w,f)(\zeta,x_1,\ldots,x_{|T|})\,d\zeta,
\]
for any $w\in L^\infty_{\xi,\zeta}$, and any $f\in L^\infty_{\xi}(L^1_x\cap L^\infty_x)$.
\label{lemtreeF}
\end{lem}
%
\begin{proof}
For any tree $T_{n}\in \Tree_{n}$, index the vertices so that the root has index $1$. We define the modified $\tau$ operator as follows
\[
\hat{\tau}(T_{n},w,f)(\xi,x_1,\ldots,x_{n}):=\int_{[0,\ 1]^n} \prod_{(k,l)\in E(T_n)}w(\xi_k,\xi_l)\,\prod_{m\in (T_n)} f(\xi_m,x_m)\,d\xi_2\ldots\,d\xi_n\Big|_{\xi_1=\xi}.
\]
Note that since $w\in L^\infty_{\xi,\zeta}$ and $f\in L^\infty_\xi L^\infty_x$ then all the computations above make sense and this will also apply in the manipulations below. We also note that we readily obtain the following representation 
\[
\tau(T,w,f)=\int_0^1\hat{\tau}(T,w,f)(\xi,\cdot)\,d\xi,
\]
for every tree $T\in \Tree$. Then, we prove the more precise result that there must exists $F\in \mathcal{T}$ such that $\hat{\tau}(T,w,f)=F(w,f)$. We proceed by induction on the size of the tree $T$. 

Let $T_2\in \Tree_2$ be the only tree with two vertices. Then, we have
\[
\hat{\tau}(T_2,w,f)(\xi,x_1,x_2)=f(\xi,x_1)\int_0^1 w(\xi,\zeta)\,f(\zeta,x_2)\,d\zeta.
\]
In other words, $\hat{\tau}(T_2,w,f)=F_2(w,f)$ with $F_2:=F_0^\star\otimes F_0$. Of course, $F_0\in \mathcal{T}$ by item $(i)$ in Definition \ref{defalgebra}. Hence $F_0^\star\in \mathcal{T}$ by item $(iii)$, and finally $F_2$ so does by item $(ii)$.

Assume by induction that the above holds for any tree in $\Tree_n$. We need to show that it is also verified for any tree in $\Tree_{n+1}$. Assume that the root is index by $1$ and has degree $k$. Index the corresponding vertices connected to the root by $2,\ldots,k+1$. Denote by $T^2,\ldots,T^{k+1}$ the subtrees starting from each vertex $2$ to $k+1$. Then, we have that
\[
\begin{split}
\hat \tau(T_{n+1},w,f)&=f(\xi,x_1)\,\int_{[0,\ 1]^k} \Pi_{i=2}^{k+1} w(\xi,\xi_i)\,\hat\tau(T^i,w,f)(\xi_i)\,d\xi_i\\
&=f(\xi,x_1)\,\Pi_{i=2}^{k+1}\int_{[0,\ 1]} w(\xi,\xi_i)\,\hat\tau(T^i,w,f)(\xi_i)\,d\xi_i.
\end{split}
\]
By our induction hypothesis, we have $F_i$ such that 
$\hat{\tau}(T^i,w,f)=F_i(w,f)$. By rule $(iii)$, for each $i$, 
\[
\int_{[0,\ 1]} w(\xi,\xi_i)\,\hat\tau(T^i,w,f)(\xi_i)\,d\xi_i=\int_{[0,\ 1]} w(\xi,\xi_i)\,F_i(w,f)(\xi_i)\,d\xi_i=F_i^*
\]
belongs to $\mathcal{T}$. Of course, $f(\xi,x_1)=F_0$ by rule $(i)$. Hence finally $\hat \tau(T_{n+1},w,f)=F_0\otimes F_2^*\dots\otimes F_{k+1}^*$ by rule $(ii)$.
\end{proof}
%
The combination of Lemmas \ref{welldefinedM} and \ref{lemtreeF} naturally lead to the following extended definition of $\tau(T,w,f)$
%
\begin{defi}[Extension of the $\tau$ operator]\label{D-tau-extended}
 Consider a transform $F\in \mathcal{T}$, for any $T\in \Tree$,  as provided by Lemma~\ref{lemtreeF}. We then define 
  \[
\tau(T,w,f)(x_1,\ldots,x_{|T|})=\int_0^1 F(w,f)(\zeta, x_1, \ldots, x_{|T|})\,d\zeta,
  \]
  for any $f\in L^\infty_{\xi}(L^1_x\cap L^\infty_x)$ and any $w\in L^\infty_\xi  \mathcal{M}_\zeta\cap L^\infty_\zeta  \mathcal{M}_\xi$.
\end{defi}
In Lemma \ref{lemtreeF} we built a special transform $F\in \mathcal{T}$ associated to each tree $T\in\Tree$. It relies on a recursive construction involving the operations $(i)$, $(ii)$ and $(iii)$ in Definition \ref{lg}, starting from the leafs of $T$ and descending towards the root. As there could a priori be several transforms $F$ that fit a given tree, the last point to check is that Definition \ref{D-tau-extended} is independent of the particular choice of $F$.
%
\begin{lem}
  Consider any $F_1,\;F_2\in \mathcal{T}$ such that 
  \begin{equation}
\int_0^1 F_1(w,f)(\zeta,.)\,d\zeta=\int_0^1 F_2(w,f)(\zeta,.)\,d\zeta ,\label{equniqueF}
  \end{equation}
  for all $f\in L^\infty_{\xi}(L^1_x\cap L^\infty_x)$ and all $w\in L^\infty_{\xi,\zeta}$.
Then, the equality \eqref{equniqueF} also holds for all $f\in L^\infty_{\xi}(L^1_x\cap L^\infty_x)$ and all $w\in L^\infty_\xi  \mathcal{M}_\zeta\cap L^\infty_\zeta  \mathcal{M}_\xi$.  \label{uniqueF}
\end{lem}
\begin{proof}
Define $w_N=K_N\star_{\xi,\zeta} w$  for some smooth convolution kernel
$K_N$ with $K_N\to \delta_0$, and for any  $w\in L^\infty_\xi  \mathcal{M}_\zeta\cap L^\infty_\zeta  \mathcal{M}_\xi$.

First of all $w_N\in L^\infty_{\xi,\zeta}$ since $w\in L^\infty_\xi  \mathcal{M}_\zeta$.
We also have that $w_N$ converges to $w$ in $L^1_\xi H^{-1}_\zeta\cap L^1_\zeta H^{-1}_\xi$.

Therefore, we first have that for $w_N$
\[
\int_0^1 F_1(w_N,f)(\zeta,.)\,d\zeta=\int_0^1 F_2(w_N,f)(\zeta,.)\,d\zeta.
\]
But, furthermore, by applying Lemma~\ref{welldefinedM}, we have that $F_1(w_N,f)$ and $F_2(w_N,F)$ converge strongly in $L^1$ to  $F_1(w,f)$ and $F_2(w,f)$, respectively. Hence, we obtain the desired equality.
\end{proof}

Being able to correctly define $\tau(T,w,f)$ is of course only the first and simplest step in the proof of Theorem~\ref{lg}. Passing to the limit in $\tau(T,w_N,f_N)$ is considerably more intricate and in particular we cannot apply Lemma~\ref{welldefinedM} as of course we cannot have the required compactness on $f_N$ and $w_N$ from the assumptions of Theorem~\ref{lg}.

Instead we have to derive compactness through the clever use of measure-preserving transforms which are the object of the next subsections.

\subsection{An extended multilinear framework}
We briefly explain in this subsection how the analysis and estimates presented here can also be applied to a larger framework of multilinear operators that extends on the transforms introduced earlier. We define a family of multilinear operators
\[
\begin{split}
M:\,&(L^\infty_\xi \mathcal{M}_\zeta  \cap L^\infty_\zeta \mathcal{M}_\xi)^k \times (L^\infty_\xi (L^1_x \cap L^\infty_x ))^l \to L^\infty_\xi ((L^1 \cap L^\infty(\RR^{dn}))\\
& (w_1,\ldots,w_k)\times (f_1,\ldots, f_l)\mapsto M(w_1,\ldots,w_k,f_1,\ldots, f_l)(\xi,x_1,\ldots,x_n).
\end{split}
\]
This family $\mathcal{M}$ is constructed exactly as the algebra of transform before through the following rules
\begin{enumerate}[label=(\roman*)]
\item The map $M_0:\,f\mapsto f$ belongs to  $\mathcal{M}$.
\item Let $M_1\in \mathcal{M}$ and $M_2\in \mathcal{M}$ be posed on $(L^\infty_\xi \mathcal{M}_\zeta  \cap L^\infty_\zeta \mathcal{M}_\xi)^{k_1} \times (L^\infty_\xi (L^1_x \cap L^\infty_x ))^{l_1}$ and $(L^\infty_\xi \mathcal{M}_\zeta  \cap L^\infty_\zeta \mathcal{M}_\xi)^{k_2} \times (L^\infty_\xi (L^1_x \cap L^\infty_x ))^{l_2}$.  Then, the following multilinear operator also belongs to $\mathcal{M}$:
\[
\begin{split}
M_1\otimes M_2:\,&(L^\infty_\xi \mathcal{M}_\zeta  \cap L^\infty_\zeta \mathcal{M}_\xi)^{k_1+k_2} \times (L^\infty_\xi (L^1_x \cap L^\infty_x ))^{l_1+l_2} \to L^\infty_\xi ((L^1 \cap L^\infty(\RR^{dn}))\\
& (w_1,\ldots,w_{k_1},w_{k_1+1},\ldots, w_{k_1+k_2})\times (f_1,\ldots, f_{l_1},f_{l_1+1},\ldots, f_{l_1+l_2})\\
&\qquad\qquad\mapsto M_1(w_1,\ldots,w_{k_1},f_1,\ldots, f_{l_1})\,M_2(w_{k_1+1},\ldots, w_{k_1+k_2},f_{l_1+1},\ldots,f_{l_1+l_2}).
\end{split}
\]
\item Let $M\in \mathcal{M}$ be posed on $(L^\infty_\xi \mathcal{M}_\zeta  \cap L^\infty_\zeta \mathcal{M}_\xi)^{k} \times (L^\infty_\xi (L^1_x \cap L^\infty_x ))^{l}$. Then, the following operator posed on $(L^\infty_\xi \mathcal{M}_\zeta  \cap L^\infty_\zeta \mathcal{M}_\xi)^{k+1} \times (L^\infty_\xi (L^1_x \cap L^\infty_x ))^{l}$ also belongs to $\mathcal{M}$:
\[
(w_1,\ldots,w_k)\times (f_1,\ldots,f_l)\mapsto \int_0^1 M(w_1,\ldots,w_k,f_1,\ldots,f_l)(\zeta,x_1,\ldots,x_n)\,w_{k+1}(\xi,d\zeta).
\]
\end{enumerate}
As one can readily see, these exactly mimic our previous construction: the only difference is that we can actually take different kernels $w_i$ and functions $f_i$ at every step when building the family~$\mathcal{M}$. This extended freedom does not seem to play a significant role in this paper but could be useful in the future. 

As one would expect, there is a one-to-one correspondance between the $F\in \mathcal{T}$ and the $M\in \mathcal{M}$ with the straightforward relation
\[
F(w,f)=M(w,\ldots,w,f,\ldots,f).
\]
As we mentioned above, all estimates and properties of $\mathcal{T}$ could be extended with the exact same arguments to $\mathcal{M}$.

 \subsection{The key compactness lemma}
A key tool to prove Theorem \ref{lg} is the following result that encompasses the inherent regularity on graphons.
%
 \begin{lem}\label{kl}
    Consider any sequence $g_n$ in $L^\infty([0,\ 1])$ with $0\leq g_n(x)\leq 2^{-n+1}$. Then, there exists $\Phi: [0,\ 1]\to [0,\ 1]$, {\it a.e.} injective, measure-preserving, such that  the following uniform regularity estimate is verified  
      \[
\sup_{n\in \mathbb{N}}\,\int_0^{1} |(g_n\circ\Phi)(\xi)-(g_n\circ\Phi)(\xi+h)|\,d\xi\leq 2^{-C\,\sqrt{\log\frac{1}{\vert h\vert}}},
      \]
for any $0<\vert h\vert <1$ and some universal constant $C$.
  \end{lem}
 \begin{rem}
By {\it a.e.} injective, we mean that there exists a full measure subset $F'\subset [0,\ 1]$, $|[0,\ 1]\setminus F'|=0$, such that $\Phi:F'\longrightarrow [0,\ 1]$ injective.
   \end{rem}
%
We note that independently on the measure-preserving map $\Phi$, by hypothesis we get that $g_n\circ \Phi\rightarrow 0$ in $L^1$ when $n\rightarrow \infty$. In particular, $g_n\circ \Phi$ is compact in $L^1$. Therefore, by the Fr\'echet-Kolmogorov theorem we infer that the sequence $g_n\circ \Phi$ must be uniformly equicontinuous so that the left hand side in the above inequality always converges to zero when $h\rightarrow 0$. The novelty of the previous result is that we can choose a special measure-preserving map $\Phi$ so that such an equicontinuity condition is made explicit in terms of a modulus of continuity. As usual, we extend by zero for evaluations outside $[0,\ 1]$, namely $(g_n\circ\Phi)(\xi+h)=0$, if $\xi+h\notin [0,\ 1]$. 
  
\begin{proof}[Proof of Lemma \ref{kl}]
  Our proof is performed in several steps and relies on a hierarchical construction associated to a suitable hierarchical decomposition of the interval $[0,\ 1]$. Throughout the proof we shall assume that $g_{n}>0$ and that the $g_n$ do not charge any point, {\em i.e.}, $|\{x\in [0,\ 1]:\ g_{n}(x)=t\}|=0$ for every $t\in \mathbb{R}$ which will make decomposing $[0,\ 1]$ according to the level sets of the $g_n$ easier. For this, we observe that if $g_n$ charges some points (at most in a countable way), then there exists $g_{n,\varepsilon}$, for every $\varepsilon>0$,  that does not charge any point and with $\|g_n-g_{n,\varepsilon}\|_{L^\infty}< \varepsilon$,  for all $n\in \mathbb{N}$. We may hence obtain the claimed regularity on each $g_{n,\epsilon}$ independently of $\varepsilon$ which implies the desired result on $g_n$.

\medskip    
$\diamond$ {\em Step~1: The hierarchical decomposition.}

\medskip

We built a sequence of covering of $[0,\ 1]$ determined through a hierarchical decomposition. For simplicity of the construction, we use indices ranging over the following special sets of words with $k$ letters
$$\Words_k:=\{i_1i_2\cdots i_k:\,i_m\in \{1,\ldots,2^m\},\ m=1,\ldots,k\},$$
for any $k\in \mathbb{Z}_+$ (see Figure \ref{fig:words}). Note that $\Words_0=\{\emptyset\}$ and for $k\geq 1$ we have that $\Words_k$ contains all possible words with $k$ letters, where the letter at position $m\in \{1,\ldots, k\}$ is allowed to take values in $\{1,\ldots,2^m\}$. For any $i\in \Words_k$ we shall define $j_k(i):=\{i1,\ldots,i2^{k+1}\}$ the subset of $\Words_{k+1}$ obtained by juxtaposition of a letter at the end of the word $i$. If we denote $n_k:=\# W_k$, then the following induction follows
\begin{equation}
\Words_{k+1}=\bigcup_{i\in \Words_k}j_k(i),\quad n_{k+1}=2^{k+1}\,n_k.\label{inductionnk}
\end{equation}
Hence, we actually obtain that $n_k=\prod_{m=0}^k 2^{m}=2^{\sum_{m=0}^k m}=2^{\frac{k(k+1)}{2}}$.

For every level $k\in \mathbb{Z}_+$ we shall consider a covering of $[0,\ 1]$ into $n_k$ sub-intervals of identical size as follows. At level $0$ we set $I_0^\emptyset=[0,\ 1]$ and given a covering $[0,\ 1]=\bigcup_{i\in \Words_k} I_k^i$ at level $k$, we define the covering at level $k+1$ by decomposing each $I_k^i$ with $i\in \Words_k$ into $2^{k+1}$ identical sub-intervals and labeling them in lexicographical order with indices in $j_k(i)$. For instance, at level $1$ we obtain $2$ sub-intervals $I_1^1=[0,\ 1/2]$, $I_1^2=[1/2,\ 1]$, and at level $2$ we obtain $8$ sub-intervals: $4$ associated with $I_1^1$
$$I_2^{11}=[0,\ 1/8],\  I_2^{12}=[1/8,\ 1/4],\  I_2^{13}=[1/4,\ 3/8],\  I_2^{14}=[3/8,\ 1/2],$$
and other $4$ associated with $I_1^2$
$$I_2^{21}=[1/2,\ 5/8],\  I_2^{22}=[5/8,\ 3/4],\  I_2^{23}=[3/4,\ 7/8],\  I_2^{24}=[7/8,\ 1].$$
We proceed in an analogous way at higher levels. Note that $I_{k+1}^j\subset I_k^i$ for every $j\in j_k(i)$ and $| I_k^i|=\frac{1}{n_k}$ since there are exactly $n_k$ intervals $I_k^i$ with $i\in \Words_k$.

If we had only one function $g_n$, the proof would be straightforward: we would just perform a decreasing (or increasing) re-arrangement to obtain a $BV$ function. The main idea of the proof is to try to mimic this by re-arranging the level sets of  $g_n$ to decrease their oscillations. Because there are a countable number of $g_n$, this is complicated. We introduce a specific decomposition of the $g_n$ along hierarchical level sets through the $O_k^i$ below. At each step, we want to cut the level set into two pieces that have same mass.
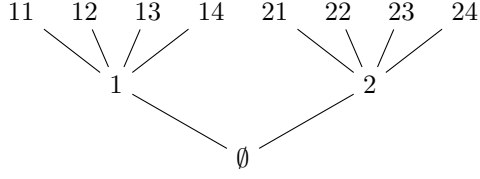
\begin{figure}[t]
\centering
\begin{forest}
for tree={grow=north,edge={-}}
[$\emptyset$ [2 [24] [23] [22] [21]] [1 [14] [13] [12] [11]]]
\end{forest}
\caption{Words in $\Words_0$, $\Words_1$ and $\Words_2$}
\label{fig:words}
\end{figure}

Now, we define new coverings $[0,\ 1]=\bigcup_{i\in \Words_k}O_k^i$ with measurable subsets $O_k^i\subset [0,\ 1]$ such that $O_{k+1}^j\subset O_k^i$ for every $j\in j_k(i)$ and $|O_k^i|=|I_k^i|=\frac{1}{n^k}$. Specifically, they will take the form
\begin{equation}\label{E-covering-O}
O_k^i:=\{x\in [0,\ 1]:\,s_{k,m}^i< g_m(x)\leq t_{k,m}^i,\ m=1,\ldots,k\},
\end{equation}
for some values $0\leq s_{k,m}^i<t_{k,m}^i\leq 2^{-m+1}$ with $m=1,\ldots,k$ to be determined below by recursion on $k$. Again, $O_0^\emptyset=[0,\ 1]$. We define $O_1^1$ and $O_1^2$ at level $1$ as follows. Since $t\mapsto |\{x\in [0,\ 1]:\,g_1(x)\geq t\}|$ is continuous (because $|\{x\in [0,\ 1]:\,g_1(x)=t\}|=0$) then there must exist some $\tilde t\in [0,\ 1]$ such that
$$|\{x\in [0,\ 1]:\,g_1(x)>\tilde t\}| = |\{x\in [0,\ 1]:\,g_1(x)\leq \tilde t\}| = \frac{1}{2}.$$
Therefore, we can set the values $s_{1,1}^1=0$, $t_{1,1}^1=\tilde t$  for $O_1^1$ and $s_{1,1}^2=\tilde t$, $t_{1,1}^2=1$ for $O_1^2$ so that
\begin{align*}
O_1^1&=\{x\in [0,\ 1]:\,s_{1,1}^1< g_1(x)\leq t_{1,1}^1\}=\{x\in [0,\ 1]:0< g_1(x)\leq \tilde t\},\\
O_1^2&=\{x\in [0,\ 1]:\,s_{1,1}^2< g_1(x)\leq t_{1,1}^2\}=\{x\in [0,\ 1]:\,\tilde t< g_1(x)\leq 1\}.
\end{align*}
So defined note that $[0,\ 1]=O_1^1\cup O_1^2$ thanks to the hypothesis $0< g_1\leq 1$, and $| O_1^1|=|O_1^2|=\frac{1}{2}$. Let us now assume that at level $k$ the values $0\leq s_{k,m}^i<t_{k,m}^i\leq 2^{-m+1}$ have been set so that the $O_k^i$ given in \eqref{E-covering-O} determine a covering $[0,\ 1]=\bigcup_{i\in \Words_k}O_k^i$ and $|O_k^i|=\frac{1}{n_k}$. Let us now proceed in the recursion by defining at level $k+1$ the values $0\leq s_{k+1,m}^j<t_{k+1,m}^j\leq 2^{-m+1}$ which will determine the $O_{k+1}^j$. 

For any fixed $i\in \Words_k$ we will built the $O_{k+1}^j$ with $j\in j_k(i)$ by appropriately decomposing $O_k^i$ into $2^{k+1}$ measurable pieces with identical measure. In this construction we will use words $\widetilde \Words_m$ containing $m$ letters with only allowed values $\{1,2\}$, {\em i.e.},
$$\widetilde \Words_m:=\{i_1i_2\cdots i_m:\,i_l\in \{1,2\},\,l=1,\ldots,m\},$$
for any $m=0,\ldots,k+1$. By a new recursion on $m$, we shall proceed by defining a family of auxiliary subsets $\widetilde O_m^\alpha$ for $\alpha\in \widetilde \Words_m$ and $m=0,1,\ldots,k+1$ verifying $\widetilde O_m^\alpha=\widetilde O_{m+1}^{\alpha 1}\cup \widetilde O_{m+1}^{\alpha 2}$ and $|\widetilde O_{m+1}^{\alpha 1}|=|\widetilde O_{m+1}^{\alpha 2}|=\frac{1}{2}|\widetilde O_m^\alpha|$ starting at $\widetilde O_0^\emptyset=O_k^i$. For $m=1$ we note that $t\mapsto |\{x\in O_k^i:\,t<g_1(x)\leq t_{k,1}^i\}|$ is continuous. Then, there must exist some $\tilde t_0^\emptyset\in (s_{k,1}^i,\ t_{k,1}^i)$ such that
$$|\{x\in O_k^i:\,s_{k,1}^i<g_1(x)\leq \tilde t_0^\emptyset\}|=|\{x\in O_k^i:\,\tilde t_0^\emptyset<g_1(x)\leq t_{k,1}^i\}|=\frac{1}{2}|O_k^i|.$$
Then, we set
\begin{align*}
\widetilde O_1^1&:=\{x\in O_k^i:\,s_{k,1}^i<g_1(x)\leq \tilde t_0^\emptyset\},\\
\widetilde O_1^2&:=\{x\in O_k^i:\,\tilde t_0^\emptyset<g_1(x)\leq t_{k,1}^i\}.
\end{align*}
Note that in $\widetilde O_1^1$ and $\widetilde O_1^2$ we have simply split the interval $[s_{k,1}^i,\ t_{k,1}^i]$ for the $g_1$ into $[s_{k,1}^i,\ \tilde t_0^\emptyset]$ and $[\tilde t_0^\emptyset,\ t_{k,1}^i]$ respectively to divide $O_k^i$ into two pieces with exactly half the mass. Assume that $\widetilde O_m^\alpha$ have been defined for $\alpha\in \widetilde\Words_m$ and $m<k$ and let us define $\widetilde O_{m+1}^{\alpha 1}$ and $\widetilde O_{m+1}^{\alpha 2}$. The argument is similar: since $t\mapsto |\{x\in \widetilde O_m^{\alpha}:\,t<g_{m+1}(x)\leq t_{k,m+1}^i\}|$ is continuous, there exists $\tilde t_m^\alpha\in (s_{k,{m+1}}^i,\ t_{k,m+1}^i)$ with
$$ |\{x\in \widetilde O_m^{\alpha}:\,s_{k,m+1}^i<g_{m+1}(x)\leq \tilde t_m^\alpha\}|=|\{x\in \widetilde O_m^{\alpha}:\,\tilde t_m^\alpha<g_{m+1}(x)\leq t_{k,m+1}^i\}|=\frac{1}{2}|\widetilde O_m^{\alpha}|.$$
Then, we define
\begin{align*}
\widetilde O_{m+1}^{\alpha 1}&=\{x\in \widetilde O_m^{\alpha}:\,s_{k,m+1}^i<g_{m+1}(x)\leq \tilde t_m^\alpha\},\\
\widetilde O_{m+1}^{\alpha 2}&=\{x\in \widetilde O_m^{\alpha}:\,\tilde t_m^\alpha<g_{m+1}(x)\leq t_{k,m+1}^i\}.
\end{align*}
Again note that in $\widetilde O_{m+1}^{\alpha 1}$ and $\widetilde O_{m+1}^{\alpha 2}$ we have simply split the interval $[s_{k,m}^i,\ t_{k,m}^i]$ for the $g_m$ into two pieces $[s_{k,m}^i,\ \tilde t_m^\alpha]$ and $[\tilde t_m^\alpha,\  t_{k,m}^i]$ respectively to divide $\widetilde O_m^{\alpha}$ into pieces with exactly half the mass. Once the auxiliary family has been set up to $m=k$, note that we can repeat the above argument once more for each $\alpha\in \widetilde\Words_k$ and find $\tilde t_k^\alpha\in (0,\ 2^{-k})$ (by the bound $0<g_{k+1}\leq 2^{-k}$) such that the following subsets divide $\widetilde O_k^{\alpha}$ into two halves with exact mass
\begin{align*}
\widetilde O_{k+1}^{\alpha 1}&:=\{x\in \widetilde O_k^{\alpha}:\,0<g_{k+1}(x)\leq \tilde t_k^\alpha\},\\
\widetilde O_{k+1}^{\alpha 2}&:=\{x\in \widetilde O_k^{\alpha}:\,\tilde t_k^\alpha<g_{k+1}(x)\leq 2^{-k}\}.
\end{align*}
By construction the last level $\{\widetilde O_{k+1}^{\alpha}:\,\alpha\in \widetilde\Words_{k+1}\}$ contain exactly $2^{k+1}$ pieces with same mass so that $O_k^i=\bigcup_{\alpha\in \widetilde \Words_{k+1}}\widetilde O_{k+1}^{\alpha}$. We then define $O_{k+1}^j$ with $j\in j_k(i)$ by labeling them in lexicographical order and we set the resulting values $0\leq s_{k+1,m}^j<t_{k+1,m}^j\leq 2^{-m+1}$ accordingly.

From this construction, we easily infer that fixing any $m\leq k$ and any $i\in \Words_k$, then at least half of the new intervals $[s_{k+1,m}^j,\ t_{k+1,m}^j]$ with $j\in j_k(i)$ must have at most half the length of $[s_{k,m}^i,\ t_{k,m}^i]$, {\em i.e.}
\begin{equation}
  \#\left\{j\in j_k(i):\ \left(t_{k+1,m}^j-s_{k+1,m}^j\right)>\frac{1}{2}\,\left(t_{k,m}^i-s_{k,m}^i\right)\right\}<2^k.
  \label{mostlyhalf}
\end{equation}
We can even be more precise. Note that going from $[s_{k,m}^i,\ t_{k,m}^{i}]$ to the next level $k+1$, we introduced $2^{k+1}$ new intervals $[s_{k+1,m}^j,\ t_{k+1,m}^{j}]$. They correspond to either of the two pieces $[s_{k,m}^i,\ \tilde t_{m-1}^\alpha]$ or $[\tilde t_{m-1}^\alpha,\ t_{k,m}^i]$ for $\alpha\in \widetilde\Words_{m-1}$ in which we split the initial interval. By construction note that each of the pieces is repeated in exactly $2^{k-m+1}$ of the $O_{k+1}^j$ with $j\in j_k(i)$. Therefore, we obtain
\begin{align*}
\frac{1}{2^{k+1}}&\,\sum_{j\in j_k(i)} (t_{k+1,m}^j-s_{k+1,m}^j)\\
&=\frac{1}{2^{k+1}}\,2^{k-m+1}\sum_{\alpha\in \widetilde \Words_{m-1}}\left((t_{k,m}^i-\tilde t_{m-1}^\alpha)+(\tilde t_{m-1}^\alpha-s_{k,m}^i)\right)\\
&=\frac{1}{2^m}\sum_{\alpha\in \widetilde\Words_{m-1}}(t_{k,m}^i-s_{k,m}^i)=\frac{1}{2}(t_{k,m}^i-s_{k,m}^i).
\end{align*}
Summing over $i\in \Words_k$, dividing by $n_k$ and recalling \eqref{inductionnk} yields
\[
\frac{1}{n_{k+1}}\,\sum_{j\in \Words_{k+1}} (t_{k+1,m}^{j}-s_{k+1,m}^{j})=\frac{1}{2\,n_{k}}\,\sum_{i\in \Words_k}  (t_{k,m}^{i}-s_{k,m}^{i}).
\]
for any $k\geq m$. By induction we then infer that
\[
\frac{1}{n_{k}}\,\sum_{i\in \Words_k}  (t_{k,m}^{i}-s_{k,m}^{i})\leq \frac{1}{2^{k-m}}\,\frac{1}{n_m}\,\sum_{i\in \Words_m} (t_{m,m}^i-s_{m,m}^i),
\]
for any $k\geq m$. Notice that by construction and our assumption on the $g_m$ we obtain that the intervals $(s_{m,m}^i,t_{m,m}^i]$ they all reduce to $(0,\frac{1}{2^{m-1}}]$. Therefore, we obtain
\begin{equation}
\frac{1}{n_{k}}\,\sum_{i\in \Words_k}  (t_{k,m}^{i}-s_{k,m}^{i})\leq \frac{1}{2^k},\label{sizedecrease}
  \end{equation}
for any $k\in \mathbb{N}$ and every $m\leq k$. Note that the fact that $0<g_m\leq \frac{1}{2^m}$ has crucially used in the above cancellations, thus leading to a uniform in $m$ bound.

\medskip

$\diamond$ {\em Step~2: The measure-preserving map $\Phi$.}

\medskip

We begin by defining $\Psi: [0,\ 1]\to [0,\ 1]$, which will correspond to the inverse of $\Phi:[0,\ 1]\to [0,\ 1]$. Set any point $x\in [0,\ 1]$ and use the previous covering by the disjoint $\{O_k^i\}_{i \in \Words_k}$ to find a unique nested sequence $i_{k+1}(x)\in j_k(i_k(x))$ such that $x\in \bigcap_{k\in \mathbb{N}} O_k^{i_k(x)}$. Since $\{I_k^{i_k(x)}\}_{k\in \mathbb{N}}$ is a nested sequence of compact sets, by Cantor's intersection theorem it is evident that their intersection is a non-empty compact interval again. In fact, since $|I_k^{i_k(x)}|\rightarrow 0$ as $k\rightarrow \infty$ then $\bigcap_{k\in \mathbb{N}} I_k^{i_k(x)}$ must consist in a singleton $\{y\}$, which we use to define $\Psi(x):=y$.

Now we prove that $\Psi$ is a measure-preserving map. Specifically, we shall prove that for every Borel $B\subset [0,\ 1]$ we have that $\Psi^{-1}(B)$ is measurable and $|\Psi^{-1}(B)|=|B|$. First, we argue for $B=I_k^i$ for any fixed $k\in \mathbb{N}$ and $i\in \Words_k$. Note that the inclusion $\Psi^{-1}(I_k^i)\supset O_k^i$ is clear by definition, but the converse does not necessarily hold. Indeed, we have that $\Psi^{-1}(I_k^i)=O_k^i\cup N_k^i$, where
\[
N_k^i=\big(\bigcap_{m\in \mathbb{N}}O_m^{l_m}\big)\bigcup \big(\bigcap_{m\in \mathbb{N}}O_k^{r_m}\big).
\]
Here, $l_{m+1}\in j_m(l_m)$ and $r_{m+1}\in j_m(r_m)$ describe the only two nested sequences of intervals $\{I_m^{l_m}\}_{m\in \mathbb{N}}$ and $\{I_m^{r_m}\}_{m\in \mathbb{N}}$ which stay always adjacent to (and respectively at the left or the right of) $I_k^i$ for $m\geq k$. Obviously, if $I_k^i$ is the first or last interval at level $k$ in lexicographical order, then there is only one such sequence. In particular, $\Psi^{-1}(I_k^i)$ is measurable because so are $O_k^i$ and $N_k^i$ (as countable union and intersection of measurable sets). In fact, $N_k^i\subset [0,\ 1]$ is negligible (as the union of two negligible subsets), then $|\Psi^{-1}(I_k^i)|=|O_k^i|=|I_k^i|$. To extend the above property to general Borel sets $B$, we note that it is enough to verify the property for any semialgebra generating the Borel $\sigma$-algebra ({\em cf.} \cite[Theorem 1.1]{W-82}), {\em e.g.}, the algebra of intervals with endpoints in the dyadic rationals $J$:
\[
\mathcal{A}:=\left\{[a,\ b):\,a,b\in J,\ 0\leq a\leq b\leq 1\right\}\cup\left\{[a,\ 1]:\,a\in J,\ 0\leq a< 1\right\}.
\]
Set any such $B\in \mathcal{A}$, for instance $B=[a,\ b)$ for $a,b\in J$ and $0\leq a< b\leq 1$ (the other case follows similarly). Then, there exists some (countable) subset $\mathcal{I}$ of indices $(k,l)$ so that $[a,\ b)=\bigcup_{(k,l)\in \mathcal{I}}I_k^i$ because the $I_k^i$ are a basis of neighborhoods of $[0,\ 1]$. By appropriately removing eventual nested elements in $\mathcal{I}$ to avoid redundant information, we may assume that all the involved $I_k^i$ are not contained in each other so that they can only intersect at points $J$ of their boundaries. Therefore, we obtain
$$\Psi^{-1}(B)=\bigcup_{(k,i)\in \mathcal{I}}\Psi^{-1}(I_k^i)=\bigcup_{(k,i)\in \mathcal{I}}(O_k^i\cup N_k^i).$$
so that $\Psi^{-1}(B)$ is measurable. Since $N_k^i$ are negligible and all the involved $O_k^i$ are disjoint then
\[
\vert \Psi^{-1}(B)\vert =\sum_{(k,i)\in \mathcal{I}}\vert O_k^i\vert=\sum_{(k,i)\in \mathcal{I}}\vert I_k^i\vert= \vert B\vert,
\]
where in the last equality we have used again that the $I_k^i$ only intersect at most at points $J$ of their boundaries and this has a null contribution to the sum.

Let us now study the injectivity and surjectivity {\em a.e.} of $\Psi$. Consider any point $y\in [0,\ 1]\setminus J$ and set unique nested sequence $i_{k+1}'(y)\in j_k(i_k'(y))$ so that $\{y\}= \bigcap_{k\in \mathbb{N}} I_k^{i_k'(y)}$. We can then consider the associated negligible measurable set $\bigcap_{k\in \mathbb{N}}O_k^{i_k'(y)}$. Notice that the {\em a.e.} surjectivity of $\Psi$ amounts to proving that the latter intersection is non-empty for {\em a.e.} $y$, whilst the {\em a.e.} injectivity of $\Psi$ amounts to proving that the intersection contains at most one point for {\em a.e.} $y$. For the injectivity, let us assume that $g_{1}$ is strictly decreasing and continuous (by adding it to the sequence $\{g_{n}\}_{n\in \mathbb{N}}$ if needed). Then note that
\[
\bigcap_{k\in \mathbb{N}}O_k^{i_k'(y)}\subset \bigcap_{k\in \mathbb{N}}g_1^{-1}\left([s_{k,1}^{i_k'(y)},t_{k,1}^{i_k'(y)}]\right)=g_1^{-1}\left(\bigcap_{k\in \mathbb{N}}[s_{k,1}^{i_k'(y)},t_{k,1}^{i_k'(y)}]\right),
\]
for each $y\in [0,\ 1]\setminus J$. Since the right hand side consists in an intersection of a nested sequence of compact intervals, then it reduces to a compact interval again. The condition for it to actually contain a single point is that $t_{k,1}^{i_k'(y)}-s_{k,1}^{i_k'(y)}\to 0$ when $k\rightarrow \infty$. Although it does not necessarily happen for each $y\in [0,\ 1]\setminus J$, we shall prove that it does happen for {\em a.e.} $y$. Our argument relies heavily on \eqref{mostlyhalf} or, its more quantitative version \eqref{sizedecrease} above. Namely, for any $m\in \mathbb{N}$ let us define the sets
\begin{equation}
J_{k,m}:=\left\{i\in \Words_k:\, t_{k,m}^i-s_{k,m}^i\geq \frac{2^k}{3^k}\right\}.\label{badindices}
\end{equation}
for any level $k\geq m$. As it will be observed below, the choice of $\frac{2}{3}$ in $J_{k,m}$ is rather arbitrary and could be replaced by any fixed rate $r\in (\frac{1}{2},1)$. Therefore,
\[
\frac{1}{n_k}\sum_{i\in \Words_k} \left(t_{k,m}^i-s_{k,m}^i \right)\geq \frac{1}{n_k}\sum_{i\in J_{k,m}} \left(t_{k,m}^i-s_{k,m}^i \right)\geq \frac{2^k\,|J_{k,m}|}{3^k\,n_k}.
\]
Hence, using \eqref{sizedecrease} implies
\begin{equation}
|J_{k,m}|\leq \frac{3^k}{4^k}\,n_k,\label{boundJkm}
\end{equation}
for any $k\geq m\geq 1$. In particular, if one defines the poor sets
\begin{equation}\label{poorset}
P_m:=\{y\in [0,\ 1]\setminus J:\, t_{k,m}^{i_k'(y)}-s_{k,m}^{i_k'(y)}\not \to 0\quad \mbox{when}\quad k\rightarrow \infty\},
\end{equation}
for any $m\in \mathbb{N}$, then one has that 
\begin{equation}
P_m\subset \bigcap_{k_0\in \mathbb{N}} \bigcup_{k\geq k_0} \bigcup_{j\in J_{k,m}} I_k^j.\label{boundPm}
\end{equation}
Since $|I_k^j|=\frac{1}{n_k}$, using \eqref{boundJkm} and \eqref{boundPm} we find that $|P_m|=0$ since
\[
|P_m|\leq \sum_{k\geq k_0} \frac{3^k}{4^k},
\]
for any $k_0\in \mathbb{N}$. Define the associated global poor set
\begin{equation}\label{poorset-global}
P:=\bigcup_{m\in \mathbb{N}}P_m.
\end{equation}
which is again negligible, and consider pair of full measure sets
\[
F_0':=[0,\ 1]\setminus (J\cup P),\quad F:=\Psi^{-1}(F_0').
\]
Then, it is clear that $\Psi: F\rightarrow [0,\ 1]$ becomes an injective (measurable) measure-preserving map. Since we are dealing with the (completed) Lebesgue $\sigma$-algebra, which determines a standard probability space, then $\Psi(F)$ is measurable ({\em cf.} \cite[Theorem 3-2]{dLR}). Unfortunately, we only have the inclusion $\Psi(F)\subset F_0'$. Hence, we define the (eventually) smaller full measure set 
\[
F':=\Psi(F).
\]
Therefore, $\Psi: F\rightarrow F'$ becomes indeed a bijective (measurable) measure-preserving map. Therefore, $\Psi^{-1}:F'\rightarrow F$ is also a measure-preserving map. Extending $\Psi^{-1}$ to all $[0,\ 1]$ we obtain our measure-preserving map $\Phi$. In particular, note that by construction $\Phi$ and $\Psi$ are characterized by the fact that 
\begin{equation}\label{E-Phi-image-dyadic}
\Phi(I_k^i\cap F')=O_k^i\cap F,\quad \Psi(O_k^i\cap F)=I_k^i\cap F',
\end{equation}
for each $k\in \mathbb{N}$ and $i\in \Words_k$.

\medskip

$\diamond$ {\em Step 3: $L^1$ modulus of continuity of $\{g_n\circ \Phi\}_{n\in \mathbb{N}}$.}

\medskip

For simplicity of notation, we define $h_{n}:=g_{n}\circ \Phi$ for any $n\in \mathbb{N}$. First, we notice that for any $y\in F'\subset F_0'$ the value $h_{m}(y)$ is determined by the sequence $t_{k,m}^{i_k'(y)}$ (or $s_{k,m}^{i_k'(y)}$), where $i_{k+1}'(y)\in j_k(i_k'(y))$ is defined as before as the unique nested sequence with $\{y\}=\bigcap_{k\in \mathbb{N}} I_k^{i_k'(y)}$. Specifically, we shall prove
\begin{equation}\label{hm-limit}
h_m(y)=\lim_{k\rightarrow \infty} t_{k,m}^{i_k'(y)}=\lim_{k\rightarrow \infty} s_{k,m}^{i_k'(y)},
\end{equation}
for each $y\in F'\subset F_0'$. Indeed, by \eqref{E-Phi-image-dyadic} and \eqref{E-covering-O} note that
\[
s_{k,m}^{i_k'(y)}< h_m(y)\leq t_{k,m}^{i_k'(y)},
\]
for any $k\geq m$ and any $y\in F'\subset F_0'$. By the definitions \eqref{poorset} and \eqref{poorset-global} of $P_m$ and $P$ the nested intervals $[s_{k,m}^{i_k'(y)},t_{k,m}^{i_k'(y)}]$ verify $t_{k,m}^{i_k'(y)}-s_{k,m}^{i_k'(y)}\to 0$ when $k\to \infty$ so that we conclude \eqref{hm-limit}.

To finish the proof, we need to show that 
\[
 \int_0^1 |h_{m}(y)-h_{m}(y+\tau)|\,dy\leq 2^{-C\,\sqrt{\log\frac{1}{|\tau|}}},
\]
for any $0<|\tau|<1$ and some universal constant $C$. Let us fix any $0<\tau <1$ (a similar argument yields $-1<\tau<0$) and any level $k\in \mathbb{N}$ so that $\tau<\frac{1}{n_k}$. That is, the size $\tau$ of the shift is smaller than the length of the dyadic intervals $I_k^i$ at the level $k$. Then, we can split each interval $I_k^i$ as follows
\[
I_k^i\cap F'=G_{k,\tau}^i\cup B_{k,\tau}^i,
\]
for every $i\in \Words_k$, where $G_{k,\tau}^i$ consist of the points in $I_k^i$ which stay in $I_k^i$ up to a shift of size $\tau$ and $B_{k\tau}^i$ are all the other points in $I_k^i$. Specifically,
\begin{align*}
G_{k,\tau}^i&:=\{y\in I_k^i\cap F':\,y+\tau\in I_k^i\cap F'\},\\
B_{k,\tau}^i&:=\{y\in I_k^i\cap F':\,y+\tau\notin I_k^i\cap F'\}.
\end{align*}
Note that a point $y\in B_k^i$ must stay within $\tau$ of the right boundary of $I_k^i$. Therefore, we infer
\begin{equation}\label{E-Iki-tau-split}
|G_{k,\tau}^i|=\frac{1}{n_k}-\tau\leq \frac{1}{n_k},\quad |B_{k,\tau}^i|=\tau.
\end{equation}
Using such a decomposition, we can split the above integrals as follows
\begin{align}\label{E-integral-hm-split}
\begin{aligned}
\int_0^1|h_m(y)-h_m(y+\tau)|\,dy&=\sum_{i\in \Words_k}\int_{G_{k,\tau}^i}|h_m(y)-h_m(y+\tau)|\,dy\\
&+\sum_{i\in \Words_k}\int_{B_{k,\tau}^i}|h_m(y)-h_m(y+\tau)|\,dy.
\end{aligned}
\end{align}
For points $y\in G_{k,\tau}^i$, since $y,y+\tau\in I_k^i\cap F'$ then $\Phi(y),\,\Phi(y+\tau)\in O_k^i$ so that
\[
|h_{m}(y)-h_{m}(y+\tau)|\leq t_{k,m}^i-s_{k,m}^i.
\]
For all other points $y\in B_{k,\tau}^i$ we only have that
\[
|h_m(y)-h_m(y+\tau)|\leq |h_m(y)|+|h_m(y+h)|\leq \frac{2}{2^{m-1}}\leq 2.
\]
Using the above bounds along with \eqref{E-Iki-tau-split} in \eqref{E-integral-hm-split}, which exploit the decay of $g_m$, we obtain
\begin{equation}\label{E-integral-hm}
\int_0^1|h_m(y)-h_m(y+\tau)|\,dy\leq \frac{1}{n_k}\sum_{i\in \Words_k}(t_{k,m}^i-s_{k,m}^i)+\sum_{i\in \Words_k} 2\tau\leq \frac{1}{2^k}+2\tau n_k.
\end{equation}

\medskip

Here, $k\in \mathbb{N}$ has only been set under the restriction $\tau<\frac{1}{n_k}$ but we could optimize our choice of $k$. Specifically, let us set $k=k_\tau$ more precisely so that $\frac{1}{n_{k+1}^2}<\tau\leq \frac{1}{n_k^2}$. Then, we obtain that
\[
\frac{1}{2^k}+2\tau n_k\leq \frac{3}{2^k}\leq 2^{-C\sqrt{\log\frac{1}{\tau}}},
\]
for universal $C$, which by \eqref{E-integral-hm} yields our result. For the last inequality, note that by our choice of $k$ in terms of $\tau$ we get $\frac{1}{\tau}\leq n_{k+1}^2=2^{(k+1)(k+2)}\leq 2^{6k^2}$ so that $\sqrt{\log \frac{1}{\tau}}\leq \sqrt{6\log 2}\,k$.
\end{proof}
%
We note that the dyadic character of the hierarchical construction in the proof of Lemma \ref{kl} required a specific geometric uniform decay $0\leq g_n\leq 2^{-n+1}$ in order for the $L^1$ modulus of continuity of $g_n\circ \Phi$ to be independent of $k$. However, applying the above to the rescaled sequence
\[
\tilde g_n:= \frac{g_n+\Vert g_n\Vert_{L^\infty}}{2^n\Vert g_n\Vert_{L^\infty}},
\]
allows considering a generic sequence $g_n$ in $L^\infty([0,\ 1])$, and obtaining a $n$-dependent $L^1$ modulus of continuity (eventually growing with $n$) as follows.

\begin{cor}\label{kc}
Consider any sequence $g_n$ in $L^\infty([0,\ 1])$. Then, there exists $\Phi: [0,\ 1]\to [0,\ 1]$, {\it a.e.} injective, measure-preserving, such that the following estimate is verified
      \[
\int_0^{1} |(g_n\circ\Phi)(\xi)-(g_n\circ\Phi)(\xi+h)|\,d\xi\leq 2^n\,\Vert g_n\Vert_{L^\infty}\,2^{-C\,\sqrt{\log\frac{1}{\vert h\vert}}},
      \]
for any $0<\vert h\vert <1$, each $n\in \mathbb{N}$ and some universal constant $C$.
\end{cor}
 
As a further remark, let us highlight the relationship of this lemma with the classical results of regularity on graphons.  If $w$ is a graphon, then the Lemma \ref{kl} implies the classical regularity lemma by Lov\'asz-Szegedy \cite{LS}, and in turn the famous regularity lemma on graphs by Szemer\'edi.

  \begin{lem}[Regularity lemma on graphons \cite{LS}]
    If $w\in L^\infty([0,\ 1]^2)$ is symmetric with $0\leq w\leq 1$, then there exists  $\Phi: [0,\ 1]\to [0,\ 1]$,  measure-preserving, such that in the  cut-distance
      \[
\delta_{\square}(w(\Phi(\cdot),\Phi(\cdot)),w(\Phi(\cdot+h),\Phi(\cdot)))\leq \frac{C}{\sqrt{\log \frac{1}{h}}},
\]
    where the cut-distance is given by
      \[\begin{split}
      \delta_{\square}(w,\tilde w)&=\sup_{\Phi,\,\tilde\Phi} \ \sup_{S,T} \ \left|\int_{S\times T} (w(\Phi(\xi),\Phi(\zeta))-\tilde w(\tilde\Phi(\xi),\tilde\Phi(\zeta))) \,d\xi\,d\zeta\right|,
      \end{split}
      \]
where the supremum is taken over any pair $\Phi, \, \tilde\Phi:[0,\ 1]\rightarrow [0,\ 1]$ of measure-preserving maps, and any pair $S, \, T\subset [0,\ 1]$ of measurable subsets.
    \end{lem}
  
  We once again emphasize that in our case $w$ is not necessarily symmetric nor a bounded function.
  
  \subsection{Proof of Theorem \ref{lg}}
We are now ready to prove the main theorem of this section. By hypothesis $(i)$ and $(ii)$  in Theorem \ref{lg} we have that there exist $C_1,C_2\in \mathbb{R}_+$ so that
\[
\Vert w_N\Vert_{L^\infty_\xi \mathcal{M}_\zeta \cap L^\infty_\zeta \mathcal{M}_\xi}\leq C_1,\quad \ \Vert f_N\Vert_{L^\infty_\xi (W^{1,1}_x\cap W^{1,\infty}_x)}\leq C_2,
\]
for any $N\in \mathbb{N}$. For each fixed $T\in \Tree$, by Lemma~\ref{L-tau-bound}, we have that $\tau(T,w_N,f_N)$ is uniformly bounded in $W^{1,1}\cap W^{1,\infty}$ and is hence locally compact in $L^p_x$ for all $p<\infty$. By a standard diagonal extraction procedure, we may hence assume that for some subsequence of $N$ (still denoted as $N$ for simplicity)
  \[
\lim_{N\to\infty} \tau(T,w_N, f_N)\quad\mbox{exists in}\quad L^p_{loc}, \quad \forall p<\infty,\quad \forall\,T\in \Tree.
  \]
  Our goal is to identify the above limits for each $T\in \Tree$.
  
We also recall the above countable algebras $M(w_N,f_N)$ in Definition \ref{defalgebra} endowed with the family of transforms $\mathcal{T}$ according to the rules $(i)$, $(ii)$ and $(iii)$. From Lemma~\ref{lemtreeF}, we already know that this algebra represents the hierarchy accurately. The rest of the proof is handled in several steps. 

\medskip

$\diamond$ {\em Step 1: Stability by re-arrangements.}\\
In order to apply our compactness Lemma \ref{kl} note that we have to perform appropriate re-arrangements in the $\xi$ variable by measure-preserving maps. The following result show that the representation through the algebra $M(w_N,f_N)$ is stable under re-arrangements.

\begin{lem}
  Consider any  $w\in L^\infty_\xi {\mathcal M}_\zeta\cap L^\infty_\zeta {\mathcal M}_\xi$, $f\in L^\infty_{\xi}(L^1_x\cap L^\infty_x)$ and any {\it a.e.} injective  measure-preserving map $\Phi: [0,\ 1]\to [0,\ 1]$. Define the re-arranged objects 
\begin{equation}\label{eq:rearranged-objects}
\tilde w(\xi,d\zeta):=\Phi^{-1}_\#w(\Phi(\xi),\cdot)(d\zeta),\quad \tilde f(x,\xi):=f(x,\Phi({\xi})).
\end{equation}
Then we have that
  \[
F(\tilde w,\tilde f)(\xi,x_1,\ldots,x_m)=F(w,f)(\Phi({\xi}),x_1,\ldots,x_m),
\]
for any transform $F\in {\mathcal T}$, and for {\it a.e.} $\xi\in [0,\ 1]$ and each $x_1,\ldots, x_m\in \mathbb{R}^d$, where $\Phi^{-1}$ is any {\it a.e.} defined left inverse of $\Phi$. Moreover, $\tau(T,\tilde w,\tilde f)=\tau(T,w,f)$ for any tree $T$.
\label{proprearrange} 
\end{lem}
If $w\in L^\infty_\xi L^1_\zeta\cap L^\infty_\zeta L^1_\xi$, then the re-arranged object $\tilde w$ in \eqref{eq:rearranged-objects} clearly belongs to $L^\infty_\xi L^1_\zeta\cap L^\infty_\zeta L^1_\xi$ and we have the straightforward formula
\[
\tilde w(\xi,\zeta)=w(\Phi(\xi),\Phi(\zeta)).
\]
For general $w\in L^\infty_\xi\mathcal{M}_\zeta\cap L^\infty_\zeta\mathcal{M}_\xi$, we have that $\tilde w$ is defined as the pull-back of $w$ through $\Phi$. Since $\Phi$ is injective,  the pull-back though $\Phi$ agrees with the push-forward trough $\Phi^{-1}$.  We also emphasize that the definition is not sensitive to modifications of $\Phi^{-1}$ over the atoms of $w(\Phi(\xi),d\zeta)$ within a Lebesgue-negligible set, as we show in the following straightforward argument based on Lemma \ref{bilinearbound}. Assume that $\Phi_1^{-1}$ and $\Phi_2^{-1}$ are two different choices of the {\it a.e.} left inverse of $\Phi$ ({\it i.e.}, $\Phi_1^{-1}\circ \Phi =Id=\Phi_2^{-1}\circ\Phi$ {\it a.e.}) and define 
\begin{align*}
\begin{aligned}
&\tilde w_1(\xi,d\zeta):=\Phi_{1\#}^{-1} w(\Phi(\xi),\cdot)(d\zeta), & & \tilde w_2(\xi,d\zeta):=\Phi_{2\#}^{-1}\tilde w(\Phi(\xi),\cdot)(d\zeta),\\
& \phi_1(\zeta):=\phi(\Phi_1^{-1}(\zeta)), & & \phi_2(\zeta):=\phi(\Phi_2^{-1}(\zeta)),
\end{aligned}
\end{align*}
for any $\phi\in L^\infty([0,\ 1])$. Then, by estimate $\eqref{wfL1Linfty}_2$ we have
\[
\begin{split}
  \left\Vert \int_0^1 \phi(\zeta)\,(\tilde w_1(\xi,d\zeta)-\tilde w_2(\xi,d\zeta))\right\Vert_{L^\infty_\zeta} & =\left\Vert \int_0^1 (\phi(\Phi_1^{-1}(\zeta))\,w(\Phi_1(\xi),d\zeta)-\phi(\Phi_2^{-1}(\zeta))\, w(\Phi_2(\xi),d\zeta)) \right\Vert_{L^\infty_\zeta}\\
  & \leq \Vert w\Vert_{L^\infty_\xi \mathcal{M}_\zeta}\Vert \phi_1-\phi_2\Vert_{L^\infty}=0,
\end{split}
\]
where we have used that $\phi_1=\phi_2$ {\it a.e.} by definition. Since $\phi\in L^\infty([0,\ 1])$ is arbitrary, we can restrict to $\phi\in C([0,\ 1])$ and by a straightforward separability argument we conclude that $\tilde w_1=\tilde w_2$.
%
%
\begin{proof}[Proof of Lemma \ref{proprearrange}]
  The proof is done by induction on the number of operations performed by $F$. Of course, for the basic transformation $F_0$ in item $(i)$ of Definition \ref{defalgebra}, the property is obvious since 
\[
F_0(\tilde w,\tilde f)(\xi,x)=\tilde f(\xi,x)=f(\Phi(\xi),x)=F_0(w,f)(\Phi(\xi),x),
\]
by definition \eqref{eq:rearranged-objects} of the re-arrangements. Now consider any $F\in \mathcal{T}$ with at least one operation, either of the type $(ii)$ or $(iii)$ in Definition \ref{defalgebra}. On the one hand, let us first assume that $F=F_1\otimes F_2$ is obtained through operations of type $(ii)$ by taking the tensor product of some $F_1$ and $F_2$. Since both the transforms $F_1$ and $F_2$ must contain at least one less operation each, by the induction hypothesis we have
  \[
  F_i(\tilde w,\tilde f)(\xi,x_1,\ldots,x_{m_i})=F_i(w,f)(\Phi(\xi),x_1,\ldots,x_{m_i}),
  \]
for $i=1,2$. Then, we obviously have
\begin{align*}
F(\tilde f,\tilde w)&(\xi,x_1,\ldots,x_{m_1+m_2})\\
&=F_1(\tilde w,\tilde f)(\xi,x_1,\ldots,x_{m_1})\,F_2(\tilde w,\tilde f)(\xi,x_{m_1+1},\ldots,x_{m_1+m_2})\\
&=F_1(w,f)(\Phi(\xi),x_1,\ldots,x_{m_1})\,F_2(w, f)(\Phi(\xi),x_{m_1+1},\ldots,x_{m_1+m_2})\\
&=F(w, f)(\Phi(\xi),x_1,\ldots,x_{m_1+m_2}). 
\end{align*}
On the other hand, assume now that the identity holds for some $F\in \mathcal{T}$ and let us consider $F^\star$ via an operation of $(iii)$, {\em i.e.},
\[
F^\star(w,f)(\xi,x_1,\ldots,x_{m})=\int_0^1 F(w,f)(\zeta,x_1,\ldots,x_{m})\,w(\xi,d\zeta).
\]
Then,
we have again that
\[\begin{split}
F^\star(\tilde w,\tilde f)(\xi,x_1,\ldots,x_{m})&=\int_0^1 F(\tilde w,\tilde f)(\zeta,x_1,\ldots,x_{m})\,\tilde w(\xi,d\zeta)\\
&=\int_0^1 F(w, f)(\Phi(\zeta),x_1,\ldots,x_{m})\,\tilde w(\xi,d\zeta)\\
&=\int_0^1 F(w, f)(\zeta,x_1,\ldots,x_{m})\, w(\Phi(\xi),d\zeta),
\end{split}
\]
where in the second line we have used that $F(\tilde w,\tilde f)(\xi,x_1,\ldots,x_m)=F(f,w)(\Phi(\xi),x_1,\ldots,x_m)$ and in the third line we have used the definition of $\tilde w$ as a push-forward measure.
Hence, we find
\[
F^\star(\tilde w,\tilde f)(\xi,x_1,\ldots,x_m)=F^\star(w,f)(\Phi(\xi),x_1,\ldots,x_m).
\]
Finally,  by Lemma \ref{lemtreeF}, for any $T\in \Tree$ there exists $F\in {\mathcal T}$ such that  
\begin{align*}
\tau(T,w,f)(x_1,\ldots,x_{|T|})&=\int_0^1 F(w,f)(\xi,x_1,\ldots,x_{|T|})\,d\xi,\\
\tau(T,\tilde w,\tilde f)(x_1,\ldots,x_{|T|})&=\int_0^1 F(\tilde w,\tilde f)(\xi,x_1,\ldots,x_{|T|})\,d\xi.
\end{align*}
Since $\Phi$ is measure-preserving, by the theorem of change of variables we obtain
\[\begin{split}
\tau(T,\tilde w,\tilde f)&=\int_0^1 F(\tilde w,\tilde f)(\xi,x_1,\ldots,x_{|T|})\,d\xi=\int_0^1 F(w,f)(\Phi(\xi),x_1,\ldots,x_{|T|})\,d\xi\\
&=\int_0^1 F(w, f)(\xi,x_1,\ldots,x_{|T|})\,d\xi=\tau(T,w,f),
\end{split}
\]
which concludes the proof.
\end{proof}

\medskip

$\diamond$ {\em Step~2: Obtaining compactness.}\\
 Since $\mathcal T$ is countable, we may label its elements with countable indices as follows $\mathcal{T}=\{F_k:\,k\in \mathbb{N}\}$, where we choose as the first element $F_0$ the elementary transform in rule $(i)$ of Definition \ref{defalgebra}. Assume that $F_k$ involves $m_k$ variables $x_1,\ldots,x_{m_k}$ and denote $n_k\in \mathbb{N}$ the amount of operations of type $(iii)$ involved in the transform $F_k$. Then, by Definition \ref{defalgebra} of the algebra $M(w_N,f_N)$ and arguing like in Lemma~\ref{L-tau-bound} by induction on the number of operations in $F_k$ implies
\begin{equation}\label{eq:uniform-bound-transformations}
\|F_k(w_N,f_N)\|_{L^\infty_\xi (W^{1,1}\cap W^{1,\infty})}\leq \|w_N\|^{n_k}_{L^\infty_\xi L^1_{\zeta}}\,\|f_N\|^{m_k}_{L^\infty_\xi (W^{1,1}_x\cap W^{1,\infty}_x)}\leq C_1^{m_k}\,C_2^{n_k},
\end{equation}
for any $N\in \mathbb{N}$ and each $k\in \mathbb{N}$. We remark that by our hypothesis, the above estimate is independent on $N$. Our goal in this part is to use Lemma~\ref{kl} (more specifically its Corollary \ref{kc}) to obtain compactness with respect to $N$ after suitable re-arrangements of the sequence $F_k(w_N,f_N)$ for any $k\in \mathbb{N}$. Unfortunately, $F_k(w_N,f_N)$ does not only depend on $\xi$ but also on the extra variables $x_1,\ldots,x_{m_k}$, so that it is not totally clear how compactness in the joint variables $(\xi,x_1,\ldots,x_{m_k})$ arises from Corollary \ref{kc}. Note thought that dependency on $x_1,\ldots,x_{m_k}$ is actually smooth.

\begin{lem}\label{lem:compactness-limits}
Under the assumptions in Theorem \ref{lg}, there exist $\Phi_N:[0,\ 1]\rightarrow [0,\ 1]$  measure-preserving maps for each $N\in \mathbb{N}$ and $\phi_k(\xi,x_1,\ldots,x_{m_k})\in L^\infty_\xi (W^{1,1}_x\cap W^{1,\infty}_x)$ for each $k\in \mathbb{N}$ so that we have
\[
F_k(w_N,f_N)(\Phi_N(\xi),x_1,\ldots,x_{m_k})\rightarrow \phi_k(\xi,x_1,\ldots,x_{m_k})\quad \mbox{in}\quad L^p_{loc}([0,\ 1]\times \mathbb{R}^{d\,m_k}),
\]
when $N\rightarrow \infty$ (up to a subsequence on $N$) for each $k\in \mathbb{N}$, and any $1\leq p<\infty$.
\end{lem}
%
\begin{proof}
For every $k\in \mathbb{N}$, let $\mathcal{D}_k=\{(x_1^l,\ldots,x_{m_k}^l):\,l\in \mathbb{N}\}\subset \mathbb{R}^{d\,m_k}$ be any countable dense subset ({\em e.g.}, $\mathcal{D}_k=\mathbb{Q}^{d\,m_k}$), and define the functions $g_{k,l}^N(\xi):=F_k(w_N,f_N)(\xi,x_1^l,\ldots,x_{m_k}^l)$ for each $\xi\in [0,\ 1]$ and any $k,\,l,\,N\in \mathbb{N}$. Let us fix $N$ and apply Corollary \ref{kc} to the countable family $\{g_{k,l}^N\}_{k,l\in \mathbb{N}}\subset L^\infty([0,\ 1])$: that is the role of $n$ in Corollary \ref{kc} is played by $(k,l)$. We obtain that for each $N\in \mathbb{N}$ there exists and a measure-preserving map $\Phi_N:[0,\ 1]\longrightarrow [0,\ 1]$ so that the re-arrangements $\tilde g_{k,l}^N=g_{k,l}^N\circ \Phi_N$ fulfill the following estimate
\[
\sup_{N\in \mathbb{N}}\int_0^1 \vert \tilde g_{k,l}^N(\xi+h) -\tilde g_{k,l}^N(\xi)\vert \,d\xi\leq C_{k,l}\,2^{-C\sqrt{\log \frac{1}{|h|}}},
\]
for any $0<|h|<1$, an universal constant $C\in \mathbb{R}_+$, and some constants $C_{k,l}\in \mathbb{R}_+$. We remark that the constants $C_{k,l}$ are indeed independent of $N$ thanks to the above uniform bound \eqref{eq:uniform-bound-transformations}. By the Fr\'echet-Kolmogorov theorem and using a diagonal extraction there exists some subsequence of $N$'s (which we still denote $N$ for simplicity) and there exists $\tilde g_{k,l}\in L^1([0,\ 1])$ so that $\tilde g_{k,l}^N\rightarrow \tilde g_{k,l}$ in $L^1([0,\ 1])$ as $N\rightarrow \infty$ for any $k,l\in \mathbb{N}$. Our final step will be to lift the above convergence on $\mathcal{D}_k$ to all $\mathbb{R}^{d\,m_k}$. To such an end, we define the functions $\tilde g_k:[0,\ 1]\times \mathcal{D}_k\longrightarrow \mathbb{R}$ by $\tilde g_k(\xi,x_1^l,\ldots,x_{m_k}^l):=\tilde g_{k,l} (\xi)$. Since $L^1$ convergence implies convergence {\em a.e.} of an appropriate subsequence, then the uniform Lipschitz bounds on \eqref{eq:uniform-bound-transformations} imply
\begin{align*}
&\vert \tilde g_k(\xi,x_1^l,\ldots,x_{m_k}^l)\vert \leq C_1^{m_k}\,C_2^{n_k},\\
&\vert \tilde g_k(\xi,x_1^{l_1},\ldots,x_{m_k}^{l_1})-\tilde g_k(\xi,x_1^{l_2},\ldots,x_{m_k}^{l_2})\vert \leq C_1^{m_k}\,C_2^{n_k}\,\vert (x_1^{l_1},\ldots,x_{m_k}^{l_1})-(x_1^{l_2},\ldots,x_{m_k}^{l_2})\vert,
\end{align*}
for any $k,\,l,\,l_1,\,l_2\in \mathbb{N}$ and {\em a.e.} $\xi\in [0,\ 1]$. Then, $\tilde g_k:[0,\ 1]\times \mathcal{D}_k\longrightarrow \mathbb{R}$ can be extended by continuity in a unique way into a function $\phi_k\in L^\infty_\xi W^{1,\infty}_{x_1,\ldots,x_{m_k}}$ with $\Vert \phi_k\Vert_{L^\infty_\xi W^{1,\infty}_{x_1,\ldots,x_{m_k}}}\leq C_1^{m_k}\,C_2^{n_k}$. From the uniform Lipschitz bounds on the $\tilde g_{k,l}^N$ and $\phi_k$, we obtain that
\[F_k(w_N,f_N)(\Phi_N(\xi),x_1,\ldots,x_{m_k})\rightarrow \phi_k(\xi,x_1,\ldots,x_{m_k}),\]
as $N\rightarrow \infty$, for each $k\in \mathbb{N}$, $a.e.\;\xi\in [0,\ 1]$, and any $x_1,\ldots,x_{m_k}\in \mathbb{R}^d$.

Note we also know that $F_k(w_N,f_N)(\cdot, \Phi_N(\cdot),\cdot,\ldots,\cdot)$ and $\phi_k(\cdot,\cdot,\ldots,\cdot)$ are bounded in $L^\infty$ uniformly with respect to $N$ thanks to the uniform bounds \eqref{eq:uniform-bound-transformations}. Therefore, the above pointwise convergence can indeed be improved into $L^p_{loc}$ convergence for any $1\leq p<\infty$, on any compact set, through the dominated convergence theorem. Indeed, since $F_k(w_N,f_N)(\cdot, \Phi_N(\cdot),\cdot,\ldots,\cdot)$ are also uniformly bounded in $L^\infty_\xi W^{1,1}_{x_1,\ldots,x_{m_k}}\subseteq L^\infty_\xi BV_{x_1,\ldots,x_{m_k}}$ with respect to $N$ by \eqref{eq:uniform-bound-transformations}, then we can also take the subsequence of $N$'s so that it converges locally weakly-star in $L^\infty_\xi BV_{x_1,\ldots,x_{m_k}}$, thus guaranteeing that $\phi_k\in L^\infty_\xi W^{1,1}_{x_1,\ldots,x_{m_k}}$.
Note we also know that $F_k(w_N,f_N)(\cdot, \Phi_N(\cdot),\cdot,\ldots,\cdot)$ and $\phi_k(\cdot,\cdot,\ldots,\cdot)$ are bounded in $L^\infty$ uniformly with respect to $N$ thanks to the uniform bounds \eqref{eq:uniform-bound-transformations}. Therefore, the above pointwise convergence can indeed be improved into $L^p_{loc}$ convergence for any $1\leq p<\infty$, on any compact set, through the dominated convergence theorem. Indeed, since $F_k(w_N,f_N)(\cdot, \Phi_N(\cdot),\cdot,\ldots,\cdot)$ are also uniformly bounded in $L^\infty_\xi W^{1,1}_{x_1,\ldots,x_{m_k}}\subseteq L^\infty_\xi BV_{x_1,\ldots,x_{m_k}}$ with respect to $N$ by \eqref{eq:uniform-bound-transformations}, then we can also take the subsequence of $N$'s so that it converges locally weakly-star in $L^\infty_\xi BV_{x_1,\ldots,x_{m_k}}$, thus guaranteeing that $\phi_k\in L^\infty_\xi W^{1,1}_{x_1,\ldots,x_{m_k}}$.
\end{proof}
%

\begin{rem}\label{rem:compactness-limits}
Lemma \ref{lem:compactness-limits} could also be extended to time-dependent objects, but we had omitted this full generality of the proof in a first approach to avoid complicating the notation. More specifically, assume that $\{w_N\}_{N\in \mathbb{N}}$ and $\{f_N\}_{N\in \mathbb{N}}$ (now time-dependent) verify the following hypothesis
    \begin{enumerate}[label=(\roman*)]
    \item $\quad \displaystyle \sup_{N\in \mathbb{N}} \sup_{\xi\in [0,\ 1]} \int_0^1 |w_N(\xi,\zeta)|\,d\zeta\, <\infty, \quad \sup_{N\in \mathbb{N}} \sup_{\zeta\in [0,\ 1]} \int_0^1 |w_N(\xi,\zeta)|\,d\xi\, <\infty,$\\
    \item $\quad \displaystyle\sup_{N\in \mathbb{N}} \|f_N\|_{L^\infty_\xi (W^{1,1}_{t,x}\cap W^{1,\infty}_{t,x})}<\infty$.\\
   \end{enumerate}
Then, there exist $\Phi_N:[0,\ 1]\rightarrow [0,\ 1]$  measure-preserving maps for each $N\in \mathbb{N}$, the same $\Phi_N$ for all $t$, and there exist $\phi_k(t, \xi,x_1,\ldots,x_{m_k})\in L^\infty_\xi (W^{1,1}_{t,x}\cap W^{1,\infty}_{t,x})$ for each $k\in \mathbb{N}$ so that we have
\[
F_k(w_N,f_N)(t, \Phi_N(\xi),x_1,\ldots,x_{m_k})\rightarrow \phi_k(t, \xi,x_1,\ldots,x_{m_k})\quad \mbox{in}\quad L^p_{loc}([0, \ t^*] \times [0,\ 1]\times \mathbb{R}^{d\,m_k}),
\]
when $N\rightarrow \infty$ (up to a subsequence on $N$) for each $k\in \mathbb{N}$, any $t^*>0$ and any $1\leq p<\infty$.

\medskip

Indeed, the proof of the time-dependent version is identical to the one of Lemma \ref{lem:compactness-limits} and simply considers the time variable $t$ along with the space variables $x_1, \ldots, x_{m_k}$ that are already included. In the proof, instead of discretizing in space we can also discretize in time by setting for every $k\in \mathbb{N}$ a countable dense set $\mathcal{D}_k=\{(t^l,x_1^l,\ldots,x_{m_k}^l):\,l\in \mathbb{N}\}\subset [0,\ t_*]\times \mathbb{R}^{dm_k}$ and defining $g_{k,l}^N(\xi):=F_k(w_N,f_N)(t^l,\xi,x_1^l,\ldots,x_{m_k}^l)$ for each $\xi\in [0,\ 1]$ and any $k,l,N\in \mathbb{N}$. Applying Corollary \ref{kc} to the countable family $\{g_{k,l}^N\}_{k,l\in \mathbb{N}}\subset L^\infty([0,\ 1])$ we obtain time-independent measure-preserving maps $\Phi_N:[0,\ 1]\longrightarrow [0,\ 1]$ such that for all $k,l\in \mathbb{N}$ the rearranged $\tilde g_{k,l}^N:=g_{k,l}^N\circ \Phi_N$ converge when $N\to\infty$ (up-to subsequence of $N$) toward some limiting $\tilde g_{k,l}$ in $L^1([0,\ 1])$. The rest of the proof follows the same train of thoughts in order to build the extension by continuity of this discrete (in time and space) limit to $\phi_k\in L^\infty_\xi (W^{1,\infty}_{t,x_1,\ldots,x_{m_k}}\cap W^{1,1}_{t,x_1,\ldots,x_{m_k}})$ so that the above convergence of the observables takes place in $L^p_{\rm loc}([0,\ t_*]\times [0,\ 1]\times \mathbb{R}^{dm_k})$.
\end{rem}

%
Before entering into the last step in the proof of Theorem \ref{lg} (namely, the identification of the limits $\phi_k$ in Lemma \ref{lem:compactness-limits}),  let us first note that at this point we already have the necessary information to prove the following result.
%
  \begin{theo}\label{th:tildeNeq}
Let $(X_1,\ldots,X_N)$ be solution to \eqref{eq1} with $K\in W^{1,1}\cap W^{1,\infty}$ and consider the associated laws $f_i(t,\cdot):=\Law (X_i(t))$. Assume that $X_i^0$ are independent with $\mathbb{E}[\vert X_i^0\vert^2]<\infty$ and
\[
\sup_{1\leq i\leq N}\sqrt{\mathbb{E}\vert X_i^0\vert^2}\leq M,\quad \sup_{1\leq i\leq N}\Vert f_i^0\Vert_{W^{1,1}\cap W^{1,\infty}}\leq L,\quad \sup_{1\leq i\leq N} \sum_{j=1}^N |w_{ij}|\leq C,\quad \sup_{1\leq j\leq N} \sum_{i=1}^N |w_{ij}|\leq C,
\]
for every $N\in \mathbb{N}$ and appropriate $M,L,\,C\in \mathbb{R}_+$. Then, there exists $\tilde w_N\in L^\infty_\xi L^1_\zeta\cap L^\infty_\zeta L^1_\xi$ and $\tilde f_N\in L^\infty_\xi (W^{1,1}_x\cap W^{1,\infty}_x)$ for every $N\in \mathbb{N}$ satisfying the graphon-type Vlasov PDE \eqref{independ2}, {\em i.e.},
\begin{equation}\label{tildefNeq}
\partial_t \tilde f_N(t,x,\xi)+\divop_x\,\left(\tilde f_N(t,x,\xi)\,\int_0^1\int_{\mathbb{R}^d} K(x-y)\, \tilde w_N(\xi,\zeta)\,\tilde f_N(t,y,\zeta)\,dy\,d\zeta\right)=0,
\end{equation}
in the distributional sense. Moreover, for any finite time interval $[0,\ t_*]$ and each compact set $\Omega\subset \mathbb{R}^d$ we have that $\tilde w_N$ and $\tilde f_N$ verify
\begin{align}\label{eq:independence+compactness}
\begin{aligned}
&\int_0^{t_*}\int_{\Omega} \int_0^1 \left|\int_0^1 (\tilde w_N(\xi,\zeta)-\tilde w_N(\xi+h,\zeta))\,\tilde f_N(t,x,\zeta)\,d\zeta\right|\,d\xi\,dx\,dt \leq \varepsilon(|h|),\\
&\sup_{t\in [0,\ t_*]}W_1\left(\mathbb{E}\,\mu_N(t,\cdot), \int_0^1 \tilde f_N(t,\cdot,\xi)\,d\xi\right) \leq  \,\tilde C\sup_{1\leq i,j\leq N} |w_{ij}|^{1/2},
\end{aligned}
\end{align}
for any $h\in \mathbb{R}$, each $N\in \mathbb{N}$, some constant $\tilde C\in \mathbb{R}_+$ and some continuous function $\varepsilon(|h|)$ with $\varepsilon(0)=0$. Both $\tilde C$ and $\varepsilon(h)$ are independent on $N$: $\tilde C$ only depends on $C$, $\Vert K\Vert_{W^{1,\infty}}$ and $t_*$, and $\varepsilon(|h|)$ only depends on $C$, $L$, $t_*$ and $\Omega$. Here, $\mu_N$ is the empirical measure $\mu_N(t, {x}):=\frac{1}{N}\sum_{i=1}^N\delta_{{X}_i(t)}(x)$.
\end{theo}
%
\begin{rem}
By interpolation with Lemma~\ref{kl}, the function $\varepsilon(h)$ can be calculated explicitly by considering an explicit choice of a dense sequence of points $x_1^l$, \ldots, $x^l_{m_k}$ in the argument above (by constructing a dyadic grid for example), as we briefly explain in the proof. 
\end{rem}
%
\begin{proof}[Proof of Theorem \ref{th:tildeNeq}]
First, by Proposition \ref{propindependence} we know that the laws $f_i$ and the solution $\bar f_i$ to the coupled system \eqref{independ} with the same initial data $\bar f_i^0= f_i^0$ satisfy the estimate \eqref{independ-error}, {\em i.e.},
\[
W_1(f_i(t,\cdot),\bar f_i(t,\cdot))\leq \tilde C\sup_{1\leq i,j\leq N}\vert w_{ij}\vert^{1/2},
\]
for each $t\in [0,\ t_*]$, any $i=1,\ldots,N$, and the constant $\tilde C:=\sqrt{2/C}\left(e^{2C\,t_*\Vert K\Vert_{W^{1,\infty}}}-1\right)$. Let us set $f_N$ and $w_N$ according to the graphon-type representation in Definition \ref{D-graphon-representation} associated with $\bar f_i$ and $w_{ij}$. Therefore, noting that $\mathbb{E}\,\mu_N(t,\cdot)=\frac{1}{N}\sum_{i=1}^N f_i(t,\cdot)$ and $\int_0^1 f_N(t,x,\xi)\,d\xi=\frac{1}{N}\sum_{i=1}^N \bar f_i(t,x,\xi)$ we infer 
\begin{equation}\label{eq:independence+compactness_pre}
\sup_{t\in [0,\ t_*]}W_1\left(\mathbb{E}\,\mu_N(t,\cdot), \int_0^1 f_N(t,\cdot,\xi)\,d\xi\right) \leq  \,\tilde C\sup_{1\leq i,j\leq N} |w_{ij}|^{1/2}.
\end{equation}

Second, as already studied in Section \ref{sec:hierarchy}, we recall that $f_N$ solves the same transport equations as in \eqref{tildefNeq} in the sense of distribution.  Arguing like in the proof of Proposition \ref{existenceweak} we can propagate the initial $L^\infty_\xi(W^{1,1}_x\cap W^{1,\infty}_x)$ norms of the solution.  This  implies that on any bounded time interval $[0,\ t_*]$, we have that $f_N\in L^\infty_\xi W^{1,1}_{t,x}\cap W^{1,\infty}_{t,x}$ and, in addition, $\Vert f_N\Vert_{L^\infty_\xi W^{1,1}_{t,x}\cap W^{1,\infty}_{t,x}}$ is bounded uniformly with respect to $N$ in terms of $\Vert K\Vert_{W^{1,\infty}}$, the constants $L,\,C$ and $t_*$. 

We note that Lemma \ref{lem:compactness-limits} was initially proved for time-independent objects. However, as mentioned in Remark \ref{lem:compactness-limits} since $f_N$ have appropriate Lipschitz dependence in the joint variable $(t,x)$ and the norms are bounded uniformly on $N$, then a similar result holds true for the time dependent objects. Specifically, there exists a measure-preserving map $\Phi_N:[0,\ 1]\longrightarrow [0,\ 1]$ for every $N\in \mathbb{N}$ such that the sequence $\{F_k(w_N,f_N)(\Phi_N(\cdot),\cdot,\ldots,\cdot)\}_{N\in \mathbb{N}}$ is compact in $L^1_{loc}([0,\ t_*]\times[0,\ 1]\times \mathbb{R}^{d\,m_k})$ for each $k\in \mathbb{N}$. For simplicity of notation, we denote again the re-arranged objects 
\[
\tilde w_N(\xi,\zeta):=w_N(\Phi_N(\xi),\Phi_N(\zeta)),\quad \tilde f_N(t,x,\xi):=f_N(t,x,\Phi_N(\xi)),
\]
as given in \eqref{eq:rearranged-objects}. In the sequel, we shall restrict to the special transform $F_0^\star\in \mathcal{T}$ consisting of only one operation of the type $(iii)$ on the elementary transform $F_0$ of rule $(i)$ in Definition \ref{defalgebra}. Specifically, we choose $F_0^\star(w_N,f_N)(t,\xi,x)=\int_0^1 w_N(\xi,\zeta) f_N(t,x,\zeta)\,d\zeta$. Indeed, by Lemma \ref{proprearrange} we can reformulate it as
\[
F_0^\star(w_N,f_N)(t,\Phi_N(\xi),x)=F_0^\star(\tilde w_N,\tilde f_N)(t,\xi,x)=\int_0^1 \tilde w_N(\xi,\zeta)\,\tilde f_N(t,x,\zeta)\,d\zeta,
\]
which is compact in $L^1_{loc}([0,\ t_*]\times [0,\ 1]\times \mathbb{R}^d)$ in particular. Then, by the Fr\'echet-Kolmogorov theorem, for every compact set $\Omega\subset \mathbb{R}^d$ there must exist then some $N$-independent continuous function $\varepsilon_{\Omega}(|h|)$ with $\varepsilon_{\Omega}(|0|)=0$  such that
\[
\int_0^1\int_0^{t_*}\int_{{\Omega}} \left|\int_0^1 (\tilde w_N(\xi,\zeta)\,\tilde f_N(t,x,\zeta)-\tilde w_N(\xi+h,\zeta)\,\tilde f_N(t+h,x+h\,v,\zeta))\,d\zeta\right|\,d\xi\,dx\,dt\leq \varepsilon_{{\Omega}}(|h|),
\]
for every $h\in \mathbb{R}$, each $v\in \mathbb{S}^{d-1}$ and each $N\in \mathbb{N}$.

Using the Lipschitz-continuity of $f_N$ (thus $\tilde f_N$) on the joint variables $(t,x)$ implies
\[
\int _0^1\int_0^{t_*}\int _{\Omega}\left|\int_0^1 (\tilde w_N(\xi,\zeta)-\tilde w_N(\xi+h,\zeta))\,\tilde f_N(t,x,\zeta)\,d\zeta\right|\,d\xi\,dx\,dt \leq \varepsilon_{{\Omega}}(|h|)+C\,t_*\,\vert \Omega\vert\,\Vert \tilde f_N\Vert_{L^\infty_\xi W^{1,\infty}_{t,x}}\vert h\vert,
\]
for every $h\in \mathbb{R}$ and any $N\in \mathbb{N}$. Since $f_N$ (thus $\tilde f_N$) are bounded in $L^\infty_\xi W^{1,\infty}_{t,x}$ uniformly in $N$, then the above estimate implies $\eqref{eq:independence+compactness}_1$. Finally, note that \eqref{tildefNeq} and the above estimate \eqref{eq:independence+compactness_pre} for $f_N$ are stable under re-arrangements on $f_N$. Hence, we conclude that $\tilde f_N$ also satisfies \eqref{tildefNeq} and $\eqref{eq:independence+compactness}_2$ by a simple change of variables.

We observe here that a more intricate interpolation argument can show that $\varepsilon$ depend only on the uniform bounds on $w_N$ and $f_N$, and on $\Omega$ and could even be made explicit. As in the proof of Lemma~\ref{lem:compactness-limits}, we construct a dense sequence $(t_l,x_l)\in [0,\ t_*]\times \Omega$, for example through dyadic grid such that for some $\theta>0$, we have that there exists $l\leq C_{\Omega}\,\varepsilon^{-\theta}$ with $|(t,x)-(t_l,x_l)|\leq \varepsilon$, for any $(t,x)$ and any $\varepsilon>0$. Now Corollary~\ref{kc} implies that 
\[
\sup_{N\in\mathbb{N}} \int_0^1 |F_0^\star(w_N,f_N)(t_l,\Phi_N(\xi+h),x_l)- F_0^\star(w_N,f_N)(t_l,\Phi_N(\xi),x_l)|\,d\xi\leq C_{{\Omega}}\,2^{l}\,2^{-C\,\sqrt{\log \frac{1}{|h|}}},
\]
for some constant $C_{\Omega}$. Thanks to the Lipschitz bound in $t$ and $x$, we also have
\[
\begin{split}
  &\sup_{N\in\mathbb{N}} \int_0^1 |F_0^\star(w_N,f_N)(t,\Phi_N(\xi+h),x)- F_0^\star(w_N,f_N)(t,\Phi_N(\xi),x)|\,d\xi\\
  &\qquad\leq \sup_{N\in\mathbb{N}} \int_0^1 |F_0^\star(w_N,f_N)(t_l,\Phi_N(\xi+h),x_l)- F_0^\star(w_N,f_N)(t_l,\Phi_N(\xi),x_l)|\,d\xi\\
  &\qquad+C\,t_*\,|\Omega|\,\Vert \tilde f_N\Vert_{L^\infty_\xi W^{1,\infty}_{t,x}}\,|(t,x)-(t_l,x_l)|,
\end{split}
\]
for any $(t,x)$ and any $(t_l,x_l)$. 
Recalling the property on the grid $(t_l,x_l)$, we find that
\[
\begin{split}
  &
  \sup_{N\in \mathbb{N}} \int_0^1 |F_0^\star(w_N,f_N)(t,\Phi_N(\xi+h),x)- F_0^\star(w_N,f_N)(t,\Phi_N(\xi),x)|\,d\xi\\
  &\qquad\leq \inf_l\left(C_{\Omega}\,2^{l}\,2^{-C\,\sqrt{\log \frac{1}{|h|}}}+C\,t_*\,|\Omega|\,\Vert \tilde f_N\Vert_{L^\infty_\xi W^{1,\infty}_{t,x}}\,l^{1/\theta}\right),
\end{split}
\]
which yields an explicit $\varepsilon_{{\Omega}}(|h|)$, indeed depending only on $\Omega$, $t_*$ and $\Vert \tilde f_N\Vert_{L^\infty_\xi W^{1,\infty}_{t,x}}$.
  \end{proof}

\medskip
$\diamond$ {\em Step~3: Identifying limits}.\\
We are finally ready to identify the limits obtained by the compactness result in Lemma \ref{lem:compactness-limits}. We recall that, according to such a result, there exists a measure-preserving map $\Phi_N:[0,\ 1]\longrightarrow [0,\ 1]$ for every $N\in \mathbb{N}$ and $\phi_k(\xi,x_1,\ldots,x_{m_k})\in L^\infty_\xi (W^{1,1}_x\cap W^{1,\infty}_x)$ for every $k\in \mathbb{N}$ so that for an appropriate subsequence in $N$ (still denoted by $N$ for simplicity) we obtain
\[
F_k(w_N,f_N)(\Phi_N(\xi),x_1,\ldots,x_{m_k})\rightarrow \phi_k(\xi,x_1,\ldots,x_{m_k})\quad \mbox{in}\quad L^p_{loc}([0,\ 1]\times \mathbb{R}^{d\,m_k}),
\]
as $N\rightarrow\infty$ for all $k\in \mathbb{N}$ and each $1\leq p<\infty$. For simplicity of notation, we shall denote again the re-arranged objects by $\tilde f_N$ and $\tilde w_N$ as given in \eqref{eq:rearranged-objects}. Using the stability under re-arrangements in Lemma \ref{proprearrange}, the above convergence can then be simply  reformulated as follows:
\begin{equation}\label{eq:convergence-Fk}
F_k(\tilde w_N,\tilde f_N)\rightarrow \phi_k\quad \mbox{in}\quad L^p_{loc}([0,\ 1]\times \mathbb{R}^{d\,m_k}),
\end{equation}
as $N\rightarrow \infty$ for all $k\in \mathbb{N}$ and each $1\leq p<\infty$. As mentioned above, our goal here reduces to identifying $\phi_k$ as $F(\tilde w, \tilde f)$ for some appropriate limiting objects $\tilde f$ and $\tilde w$ of the re-arranged objects $\tilde f_N$ and $\tilde w_N$.

On the one hand, since $\tilde f_N$ is the first element in the algebra $M(\tilde w_N,\tilde f_N)$, the above implies that there exists $\tilde f\in L^\infty_\xi (W^{1,1}_x\cap W^{1,\infty}_x$) (which is indeed the first of the above $\phi_k$) so that we have
\begin{equation}\label{eq:convergence-fN}
\tilde f_N\rightarrow \tilde f\quad \mbox{in}\quad L^p_{loc}([0,\ 1]\times \mathbb{R}^d),
\end{equation}
when $N\rightarrow \infty$ for each $1\leq p<\infty$. On the other hand, by the uniform bound of $w_N$ in item $(i)$ of Theorem \ref{lg}, we readily obtain that $\Vert \tilde w_N\Vert_{L^\infty_\xi L^1_\zeta}\leq C$ and $\Vert \tilde w_N\Vert_{L^\infty_\zeta L^1_\xi}\leq C$ for every $N\in \mathbb{N}$. Since we have (see \cite{IT-77})
\[
L^\infty([0,\ 1],\,L^1(0,1))\subset L^\infty([0,\ 1],\,\mathcal{M}([0,\ 1]))=(L^1([0,\ 1],\,C([0,\ 1])))^*,
\] 
then the Alaoglu-Bourbaki theorem claims that there exist a subsequence on $N$ (again denoted by $N$ for simplicity) and some $\tilde w\in L^\infty_\xi\mathcal{M}_\zeta$ (which is in fact exchangeable because so are the $\tilde w_N$) so that
\begin{equation}\label{eq:convergence-wN}
\tilde w_N\overset{*}{\rightharpoonup}\tilde w\quad \mbox{in}\quad L^\infty_\xi\mathcal{M}_\zeta\cap L^\infty_\zeta\mathcal{M}_\xi,
\end{equation}
when $N\rightarrow \infty$. Our last step is to prove that indeed $\phi_k=F_k(\tilde w,\tilde f)$ for every $k\in \mathbb{N}$. Again, we proceed by induction on the number of operations in $F_k$.

First, let us assume that $F_k$ is built via an operation of type $(ii)$ in Definition \ref{defalgebra}. Specifically, assume that $F_k=F_{k_1}\otimes F_{k_2}$. Since each $F_{k_1}$ and $F_{k_2}$ must contain at least one less operation that $F_k$, then the induction hypothesis implies
\[
F_{k_i}(\tilde w,\tilde f)=\phi_{k_i}=\lim_{N\rightarrow\infty} F_{k_i}(\tilde w_N,\tilde f_N), 
\]
for $i=1,\,2$. We recall that the above is a limit in $L^p_{loc}$ for any $1\leq p<\infty$. Therefore,
\begin{align*}
\phi_k(\xi,x_1,\ldots,x_{m_k})&=\lim_{N\rightarrow \infty} F_k(\tilde w_N,\tilde f_N)(\xi,x_1,\ldots,x_{m_k})\\
&=\lim_{N\rightarrow \infty} F_{k_1}(\tilde w_N,\tilde f_N)(\xi,x_1,\ldots,x_{m_{k_1}})\,\lim_{N\rightarrow\infty} F_{k_2}(\tilde w_N,\tilde f_N)(\xi,x_{m_{k_1}+1},\ldots,x_{m_{k}})\\
&=\phi_{k_1}(\xi,x_1,\ldots,x_{m_{k_1}})\,\phi_{k_2}(\xi,x_{m_{k_1}+1},\ldots,x_{m_{k}})\\
&=F_{k_1}(\tilde w,\tilde f)(\xi,x_1,\ldots,x_{m_{k_1}})\,F_{k_2}(\tilde w,\tilde f)(\xi,x_{m_{k_1}+1},\ldots,x_{m_{k}})\\
&=F_k(\tilde w,\tilde f) (\xi,x_1,\ldots,x_{m_k}).
\end{align*}
Second, let us assume that $F_k$ is obtained through an operation of type $(iii)$ in Definition \ref{defalgebra}. Specifically, we have there exists some transform $F_{k'}\in \mathcal{T}$ with $m_k=m_k'$ such that $F_k=F_{k'}^\star$, that is
\begin{equation}\label{eq:decomposition-Fk0}
F_k(\tilde w_N,\tilde f_N)(\xi,x_1,\ldots,x_{m_k})=\int_0^1 \tilde w_N(\xi,\zeta)\,F_{k'}(\tilde w_N,\tilde f_N)(\zeta,x_1,\ldots,x_{m_k})\,d\zeta,
\end{equation}
holds. Since $F_{k'}$ contains at least one operation less than $F_k$, then the induction hypothesis shows that
\begin{equation}\label{eq:convergence-Fk0}
F_{k'}(\tilde w,\tilde f)=\phi_{k'}=\lim_{N\rightarrow\infty} F_{k'}(\tilde w_N,\tilde f_N),
\end{equation}
again in $L^p_{loc}$ for any $1\leq p<\infty$. Take any $\varphi\in C_c([0,\ 1]\times \mathbb{R}^{d\,m_k})$ and define
\begin{align*}
\psi_N(\xi,\zeta)&:=\int_{\mathbb{R}^{d\,m_k}}\varphi(\xi,x_1,\ldots,x_{m_k})\,F_{k'}(\tilde w_N,\tilde f_N)(\zeta,x_1,\ldots,x_{m_k})\,dx_1\ldots\,dx_{m_k},\\
\psi(\xi,\zeta)&:=\int_{\mathbb{R}^{d\,m_k}}\varphi(\xi,x_1,\ldots,x_{m_k})\,F_{k'}(\tilde w,\tilde f)(\zeta,x_1,\ldots,x_{m_k})\,dx_1\ldots\,dx_{m_k}.
\end{align*}
Then, by \eqref{eq:convergence-Fk0} it is clear that $\psi_N(\xi,\zeta)\rightarrow \psi(\xi,\zeta)$ in $L^1_{\zeta} C_{\xi}$. Multiplying \eqref{eq:decomposition-Fk0} by $\varphi$ and integrating note that we can write
\begin{align*}
\int_0^1\int_{\mathbb{R}^{d\,m_k}} &\varphi(\xi,x_1,\ldots,x_{m_k})\,F_k(\tilde w_N,\tilde f_N)(\xi,x_1,\ldots,x_{m_k})\,dx_1\ldots\,dx_{m_k}\,d\xi\\
&=\int_0^1\int_0^1 \tilde w_N(\xi,\zeta) \psi_N(\xi,\zeta)\,d\xi\,d\zeta.
\end{align*}
Using that $\psi_N(\xi,\zeta)\rightarrow \psi(\xi,\zeta)$ in $L^1_{\zeta} C_{\xi}$ and the weak-* convergence of $\tilde w_N$ in $L^\infty_\zeta\mathcal{M}_\xi$ as given in \eqref{eq:convergence-wN} imply that
\begin{align*}
\lim_{N\rightarrow \infty}\int_0^1\int_{\mathbb{R}^{d\,m_k}} &\varphi(\xi,x_1,\ldots,x_{m_k})\,F_k(\tilde f_N,\tilde w_N)(\xi,x_1,\ldots,x_{m_k})\,dx_1\ldots\,dx_{m_k}\,d\xi\\
&=\int_0^1\int_0 ^1 \tilde w(\xi,d\zeta)\,\psi(\xi,\zeta)\,d\xi,
\end{align*}
where in the last step we have used Remark~\ref{R-exchangeable-case}. Undoing the definition of $\psi$ we have proven (by arbitrariness of $\varphi\in C_c([0,\ 1]\times \mathbb{R}^{d\,m_k})$) that
\[
F_k(\tilde w_N,\tilde f_N)\overset{*}{\rightharpoonup} F_k(\tilde w,\tilde f)\quad \mbox{in}\quad \mathcal{M}([0,\ 1]\times \mathbb{R}^{d\,m_k}),
\]
as $N\rightarrow\infty$. Since $F_k(\tilde w_N,\tilde f_N)\rightarrow \phi_k$ in $L^p_{loc}$ by \eqref{eq:convergence-Fk}, then we conclude that $\phi_k=F_k(\tilde w,\tilde f)$.

Finally,  we conclude the identifications of the limits of $\tau(T,w_N,f_N)$ by using Lemmas \ref{lemtreeF} \ref{proprearrange}. Specifically, for any $T\in \Tree$, there exists some $F\in {\mathcal T}$ such that 
\[
\tau(T,w_N,f_N)(x_1,\ldots,x_{|T|})=\tau(T,\tilde w_N,\tilde f_N)(x_1,\ldots,x_{|T|})=\int_0^1 F(\tilde w_N,\tilde f_N)(\xi,x_1,\ldots,x_{|T|})\,d\xi.
\]
 We recall the above convergence of the $F(\tilde f_N,\tilde w_N)$ to $F(\tilde f,\tilde w)$ in $L^p_{loc}$ as $N\rightarrow\infty$. Then, we can pass to the limit in the right hand side of the above relation to achieve
\[
\tau(T,w_N,f_N)(x_1,\ldots,x_{|T|})\rightarrow \int_0^1 F(\tilde f,\tilde w)(\xi,x_1,\ldots,x_{|T|})\,d\xi=\tau(T,\tilde w,\tilde f),
\]
as $N\rightarrow\infty$ in $L^p_{loc}(\mathbb{R}^{d\,|T|})$, for any $1\leq p<\infty$ and each $T\in \Tree$,  where in the last step we have used Definition \ref{D-tau-extended} for the extension of the operator $\tau$.
 \subsection{An alternative formulation of the proof}

We mention here a possible variant of the proof above, which consists in re-arranging  the arguments in our proof along the following 3 lemmas (separated to make the discussion easier here)

\begin{lem}
For any sequence $w_N$ uniformly bounded in $L^\infty_\xi \mathcal{M}_\zeta \cap L^\infty_\zeta \mathcal{M}_\xi$ and any sequence $f_N$ uniformly bounded in $L^\infty_\xi W^{1,\infty}_{t,x}$, there exists a measure-preserving map $\Phi_N$ such that $\tilde f_N(t,x,\xi)=f_N(t,x,\Phi_N(\xi))$ converges strongly in $L^1_{t,x,\xi}$, $\tilde w_N(\xi,\zeta)=w_N(\Phi_N(\xi),\Phi_N(\zeta)$ converges weakly to $w \in L^\infty_\xi \mathcal{M}_\zeta \cap L^\infty_\zeta \mathcal{M}_\xi$, and
\[ 
\int^1_0
(\tilde w_{N}(\xi, d\zeta) - w(\xi, d\zeta))\, h(\zeta) \to 0
\]
strongly in $L^1_\xi$ for all $h$ in a countable dense family of $C([0,\ 1])$. 
\label{lemmanew1}
\end{lem}

The proof of Lemma~\ref{lemmanew1} is a straightforward consequence of Lemma~\ref{kl} and Corollary~\ref{kc} for $\tilde w_N$. For $f_N$, it follows from a time-dependent version o Lemma~\ref{lem:compactness-limits} as mentioned in Remark~\ref{rem:compactness-limits}. We remark here that the proof of Lemma~\ref{lem:compactness-limits} could equivalently be performed by using a basis of the function space.

The next step is as follows:

\begin{lem}
Consider a sequence $\tilde w_N$ uniformly bounded in $L^\infty_\xi \mathcal{M}_\zeta \cap L^\infty_\zeta \mathcal{M}_\xi$. Assume that $\tilde w_N(\xi,\zeta)$ converges weakly to $L^\infty_\xi \mathcal{M}_\zeta \cap L^\infty_\zeta \mathcal{M}_\xi$, and
\[ 
\int^1_0
(\tilde w_{N}(\xi, d\zeta) - \tilde w(\xi, d\zeta)) h(\zeta) \to 0
\]
strongly in $L^1_\xi$ for all $h$ in a countable dense family of $C([0,\ 1])$. Then for any function $g_N(t,x_1,\ldots,x_k,\xi)$ uniformly bounded in $L^\infty$ and converging to $g$ strongly in $L^1$, we have that
\[
\int^1_0
(\tilde w_{N}(\xi, d\zeta) - \tilde w(\xi, d\zeta)) g_N(t,x_1,\ldots,x_k,\zeta) \to 0
\]
strongly in $L^1_{t,x_1,\ldots,x_d,\xi}$.
\label{lemmanew2}
\end{lem}
Lemma~\ref{lemmanew2} follows easily from a density argument, together with the estimate in Lemma~\ref{bilinearboundBx} (notice that this is already implicitly used in step 3 of the proof of Theorem~\ref{lg}).

The final step is as follows
\begin{lem}
Consider any sequence $\tilde w_N$ uniformly bounded in $L^\infty_\xi \mathcal{M}_\zeta \cap L^\infty_\zeta \mathcal{M}_\xi$ and any sequence $\tilde f_N$ uniformly bounded in $L^\infty_\xi W^{1,\infty}_{t,x}$,  such that $\tilde f_N(t,x,\xi)=f_N(t,x,\Phi_N(\xi))$ converges strongly in $L^1_{t,x,\xi}$, $\tilde w_N(\xi,\zeta)=w_N(\Phi_N(\xi),\Phi_N(\zeta)$ converges weakly to $w \in L^\infty_\xi \mathcal{M}_\zeta \cap L^\infty_\zeta \mathcal{M}_\xi$, and assume that for any function $g_N(t,x_1,\ldots,x_k,\xi)$ uniformly bounded in $L^\infty$ and converging to $g$ strongly in $L^1$, we have that
\[
\int^1_0
(\tilde w_{N}(\xi, d\zeta) - \tilde w(\xi, d\zeta)) g_N(t,x_1,\ldots,x_k,\zeta) \to 0
\]
strongly in $L^1_{t,x_1,\ldots,x_d,\xi}$. Then for all transform $F$ (or equivalently all multilinear operators), we have that $F(\tilde w_N, \tilde f_N)$ converges strongly in $L^1$ to $F(w,f)$.
\label{lemmanew3}
\end{lem}
The proof of Lemma~\ref{lemmanew3} is exactly the induction argument in step 3 of the proof of Theorem~\ref{lg}.

This approach skips Theorem~\ref{th:tildeNeq}, which is however interesting on itself. But it does have some advantage and in particular, it allows to disentangle more the limits of $f_N$ and of $w_N$ through Lemma~\ref{lemmanew1}.

\section{Proof of the main result: Theorem~\ref{maintheorem}}\label{sec:proof-maintheorem}
We are now ready to perform the full proof of Theorem~\ref{maintheorem}.

First of all, denote by $f_i^0$ the density of each $X_i^0$. Then, we define $w_N(\xi,\zeta)$ and $f^0_N(x,\xi)$ according to equation~\eqref{wf} in  Definition~\ref{D-graphon-representation}, which we recall
\begin{align*}
\begin{aligned}
  &w_N(\xi,\zeta)=\sum_{i,j=1}^N N\,w_{ij}\,\mathbb{I}_{[\frac{i-1}{N},\frac{i}{N})}(\xi)\, \mathbb{I}_{[\frac{j-1}{N},\frac{j}{N})}(\zeta), & & \xi,\zeta\in [0,\ 1],\\
  &f_N^0(x,\xi)=\sum_{i=1}^N f_i^0(x)\,\mathbb{I}_{[\frac{i-1}{N},\frac{i}{N})}(\zeta), & & x\in \R^d,\xi\in [0,\ 1].
  \end{aligned}
\end{align*}

From the uniform bounds on  $f_i^0$ and $w_{ij}$, it is immediate to check that $w_N$ and $f_N^0$ satisfy the assumptions $(i)-(ii)$ of Theorem~\ref{lg}. We note as well that $\tau(T,w_N,f_N)$ is uniformly bounded in $W^{1,\infty}$, for any fixed $T$. Since trees are countable, by the usual diagonal extraction technique, we can obtain a subsequence (still denoted by $N$) such that  $\tau(T,w_N,f_N)$ converges strongly in $L^\infty$, for every $T$. Hence,  this subsequence also satisfies assumption $(iii)$ of Theorem~\ref{lg}.

Applying now Theorem~\ref{lg}, we find $w\in L^\infty_\xi \mathcal{M}_\zeta\cap L^\infty_\zeta \mathcal{M}_\xi$ and $f^0\in L^\infty_\xi (W^{1,1}_x\cap W^{1,\infty}_x)$ such that
\[
\tau(T,w,f^0)=\lim_{N \to \infty} \tau(T,w_N,f_N^0),\quad\forall T,
\]
in any $L^p_{loc}$ with $p<\infty$. Furthermore, we note that, for each $T$, $\tau(T,w,f^0)$ is uniformly bounded in $L^1_x\cap L^\infty_x$. Fixing now any $R>0$, a straightforward extension of Lemma~\ref{L-tau-bound} leads to
\[
\|\tau(T,w,f^0_N)\|_{L^1(\R^{d\,|T|})\setminus B(0,R)^{|T|}}\leq |T|\,\|w\|_{L^\infty_\xi M_\zeta}^{|T|-1}\,\|f^0_N\|_{L^\infty_\xi L^1(\R^{d})}^{|T|-1}\,\|f^0_N\|_{L^\infty_\xi L^1(\R^{d}\setminus B(0,R))}.
\]
On the other hand, from the assumptions the main theorem, we have that
\[
\int_{\R^d} |x|^2\,f_i^0(x)\,dx\leq C,
\]
for some constant $C$ independent of $N$ and $i$. By the definition of $f^0_N$, this implies as well that
\[
\int_{\R^d} |x|^2\,f_N^0(x)\,dx\leq C.
\]
This shows that
\[
\|\tau(T,w,f^0_N)\|_{L^1(\R^{d\,|T|})\setminus B(0,R)^{|T|}}\leq \frac{C}{R^2}\,|T|\,\|w\|_{L^\infty_\xi M_\zeta}^{|T|-1}\,\|f^0_N\|_{L^\infty_\xi L^1(\R^{d})}^{|T|-1}.
\]
Combined with the local convergence in $L^p$, it implies that $\tau(T,w_N,f_N^0)$ converges strongly to $\tau(T,w_N,f^0)$ in every $L^p(\R^{d\,|T|})$ and in particular in $L^2(\R^{d\,|T|})$.

\medskip

We further note that also by Corollary~\ref{cor-L-tau-bound}, we have that
\[
\|\tau(T,w,f^0)\|_{L^2_x}\leq \|w\|_{L^\infty_\xi \mathcal{M}_\zeta}^{|T|-1}\,\|f^0\|_{L^\infty_\xi L^2_x}^{|T|}.
\]
Consequently, there exists some $\lambda>0$ small enough such that
\begin{equation}
\|\tau(\cdot,w,f^0)-\tau(\cdot,w_N,f^0_N)\|_{\lambda}\to 0,\quad\mbox{as}\ N\to\infty.\label{convergencetau}
  \end{equation}
By Proposition~\ref{existenceweak}, we have existence of weak solution $f\in L^\infty_{t,\xi} ( W^{1,1}_x\cap W^{1,\infty}_x)$ to \eqref{independ2}  associated with $w$ in the sense of Definition~\ref{D-weak-solution-independ2},  namely
\[
\left\{\begin{array}{l}
\displaystyle\partial_t f(t,x,\xi)+\divop_x\left(f(t,x,\xi) \,\int_0^1 w(\xi,d\zeta)\int_{\R^d} K(x-y)\,f(t,y,\zeta)\,dy\right)=0,\\
\displaystyle f(0,\cdot,\cdot)=f^0.
\end{array}\right.
\]
Of course, $f$ is the function that we will use to compare to the empirical measure, meaning that we have to prove that
\begin{equation}
\mathbb{E}\,W_1\left(\int_0^1 f(t,\cdot,\xi)\,d\xi,\;\frac{1}{N}\,\sum_{i=1}^N \delta_{X_i(t)}\right)\to 0,\quad \mbox{as}\ N\to \infty.\label{desiredresult}
\end{equation}
First of all, we apply Proposition~\ref{propindependence}: Since $\sup_{1\leq i,j\leq N}\vert w_{ij}\vert\rightarrow 0$ as $N\rightarrow\infty$, then this shows that the solutions $\bar f_i$ to \eqref{independ} satisfy that
\begin{equation}
\mathbb{E}\,W_1\left( \frac{1}{N}\sum_{i=1}^N \bar f_i(t,\cdot),\;\frac{1}{N}\,\sum_{i=1}^N \delta_{X_i(t)}\right)\to 0,\quad \mbox{as}\ N\to \infty.\label{empiricalbarfi} 
\end{equation}
Then, we define $f_N$ as per \eqref{wf}, namely
\[
f_N(t,x,\xi)=\sum_{i=1}^N \bar f_i(t,x)\,\mathbb{I}_{[\frac{i-1}{N},\frac{i}{N})}(\xi),\qquad t\in\R_+,\;x\in \R^d,\;\xi\in [0,\ 1].
\]
We recall that $(w_N,f_N)$ is also a weak solution to \eqref{independ2}, with $f_N\in L^\infty_{t,\xi} (W^{1,1}_x\cap W^{1,\infty}_x)$. Therefore, both $f_N$ and $f$ satisfy the assumptions of Theorem~\ref{stabilitynonviscous}. Consequently, using \eqref{convergencetau}, we find
\[
\left\|\int_0^1 (f-f_N)(t,\cdot,\xi)\,d\xi\right\|_{L^2_x}\leq \frac{C(t)}{{(\log |\log \|\tau(\cdot,w_N,f^0_N)-\tau(\cdot,w,f^0)\|_{\lambda}|)^{1/2}_+}}\to 0,\quad \mbox{as}\ N\to \infty.
\]
Since the $L^2$ norm locally dominates the Wasserstein distance and we have propagation of any moments, this implies that
\[
W_1\left(\int_0^1 f_N(t,\cdot,\xi)\,d\xi,\int_0^1 f(t,\cdot,\xi)\,d\xi\right)\to 0,\quad \mbox{as}\ N\to \infty.
\]
Recalling that $\int_0^1 f_N(t,x,\xi)\,d\xi=\frac{1}{N}\,\sum_{i=1}^N \bar f_i(t,x)$, and combining this convergence with \eqref{empiricalbarfi}, finally we obtain  \eqref{desiredresult}, which therefore concludes the proof of Theorem~\ref{maintheorem}. 

\appendix
\section{Models of biological neurons \label{neurons-a}}
\subsection{Smooth neuron dynamics of the Fitzhugh-Nagumo or Hodgkin-Huxley type}
Models of biological neurons offer a natural example which requires the more complex structure of the connectivities that we handle in this article. The purpose of this appendix is to give a very brief overview of some of the models, how they motivate our current study and what are the remaining open questions. There exists an extensive literature on the subject that we cannot do full justice to in this appendix. We first briefly refer to~\cite{GerKis, sporns},  for a comprehensive point of view and to~\cite{CoBrGoWa,Compte} for the cortical network.

Typical models of neuron dynamics include the now famous Fitzhugh-Nagumo \cite{FitzHugh,Nagumo}, or Hodgkin-Huxley \cite{HH}, which naturally fit in the framework of \eqref{eq1}. In the case of Hodgkin-Huxley for example, one would take $X_i(t)=(V_i(t),n_i(t), m_i(t),h_i(t))$ where $V_i$ is the membrane potential, $(n_i,m_i)$ are connected to the activation for the potassium and sodium channels and $h_i$ is connected to the inactivation for the Sodium channel. $(n_i(t),m_i(t),h_i(t))$ solve uncoupled ODE's that can be represented by adding some self-interaction in our model, for example
\begin{equation}\label{eqnmh}
\left\{
\begin{array}{l}
\displaystyle\frac{d}{dt} n_i(t)=\alpha_n(V_i(t))\,(1-n_i(t))-\beta_n(V_i(t))\,n_i(t),\\
\displaystyle\frac{d}{dt} m_i(t)=\alpha_m(V_i(t))\,(1-m_i(t))-\beta_m(V_i(t))\,m_i(t),\\
\displaystyle\frac{d}{dt} h_i(t)=\alpha_h(V_i(t))\,(1-h_i(t))-\beta_h(V_i(t))\,h_i(t),
\end{array}
\right.
\end{equation}
for some given functions $\alpha_{n,m,h}$ and $\beta_{n,m,h}$.

The neurons are coupled through their membrane potentials solving 
\begin{equation}
  C_m\,\frac{dV_i}{dt}+g_K\,n_i^4\,(V_i-V_K)+g_{Na}\,m_i^3\,h_i(V_i-V_{Na}) +g_l\,(V_i-V_l)+\sum_{j\neq i} w_{ij} (V_i-V_j)=0,\label{eqVi}
\end{equation}
for given constants $C_m,\;g_K,\;g_{Na},\;g_l,\;V_K,\;V_{Na},\;V_l$.

The coupled system composed of \eqref{eqnmh}-\eqref{eqVi} obviously fits within the general framework of systems like~\eqref{eq1}. Provided all coefficients are smooth, it also satisfies the assumptions of Theorem~\ref{maintheorem} so that all our results are immediately applicable.

The simulation of such large systems is challenging, and mean-field limits are again an attractive alternative; see~\cite{MG1,O,MG2,RBW}. For identical weights $w_{ij}=1/N$, the mean-field limit can be rigorously derived by classical methods and has proved useful in understanding some of the large scale behavior of such systems, see for example~\cite{DeInRuTa,DeInRuTa2,DemGalLocPre,FPZ,FouLoc} on the stochastic side or~\cite{CCP,CP,CaGoGuSc,CPSS,FPZ,PaPeSa,PS} on the deterministic side.

Those studies do not account for the structure of connectivity in the neural network, whose effect has been shown to be critical, as in \cite{PhPaChVi,PJ, PiShPaShLiChSi,W-S,WC,WJ}. Of course, a first practical modeling issue, when considering more complex networks, simply is that the connectivity map $w_{ij}$  has long remain mysterious. Random weights $w_{ij}$ are popular for this reason and naturally lead to random networks, see~\cite{bara}. The corresponding graph of interactions is typically dense and some notion of mean-field limit, based on graphons, has been derived for example in~\cite{CMM,CM1,CM2} for smooth interaction kernels $K$, and actually for the so-called Kuramoto~\cite{Ku1,Ku2} model without learning (see~\eqref{KM1} below with $\eta=0$).

Those follow an interesting approach to the problem when the interacting network  is described as a random graph, usually based on  symmetric Erd\H os-R\'enyi type  graphs. The family of Erd\H os-R\'eenyi graphs  is one example of a convergent family of random graphs. The limiting behavior of such families is determined by a symmetric measurable function on
the unit square: a classical graphon. In the case of the Erd\H os-R\'eenyi graphs, the limiting graphon is even a constant function. This also leads to investigating bifurcations in the Kuramoto models  on a variety of graphs ranging from symmetric and random small-world and power law graphs. The graph structure plays a key role in the transition to synchronization in the model. In \cite{CMM}, Kuramoto models on Erd\H{o}s-R\'enyi and Paley graphs have been proven to have the same critical value  that defines synchronization, an the authors explain  the  onset  of  synchronization  in  the Kuramoto  model  on  a  broad  class  of  dense  and  sparse random  graphs by establishing an explicit link between the network structure and the onset of synchronization. 
 Furthermore, the mean-field limits of the Kuramoto models on these graphs can be proved to coincide with that for the Kuramoto models on weighted complete graphs \cite{CM1,CM2}.

However, from a scaling point of view, it is expected that the {\em graph of connection maps} $w_{ij}$ should be {\em sparse} for applications in  neurosciences. Mammalian brains contain between {$10^8 - 10^{11}$ neurons} ($86\cdot 10^9$ actually in humans). Although, of course, the models generally apply only to subdomains or certain neural clusters in charge of some specific mission, the dimensions of these clusters is very large {$ N \gg 1 $} and even these clusters tend to be correlated with others, which also implies very large dimensions.
For example, in the human brain each neuron has on average $ 7000 $ synaptic connections to other neurons. This implies that synaptic weights must have a scale according to this number of the order of { $\sup_{i,j} w_{ij}\sim 10^{-3}-10^{-4}\ll 1$}. 

We also emphasize the importance of not imposing symmetry conditions to keep $w_{ij}\neq w_{ji}$ in general. In the neuronal framework, the breakdown in symmetry of network connectivity can be  correlated with a time scale hierarchy in some special situations. Small deviations from symmetric connectivity will cause small deviations from the invariant dynamics and introduce a temporal hierarchy, ensuring the consequence of an emergent separation on the time scale. Exploring network mechanisms due to asymmetric couplings to create more complex processes suggests that asymmetric coupling is not always a variation that needs to be averaged, since complex dynamics can be formed by simpler dynamics through a combination of asymmetric couplings. For example, stimulating a local part of neural tissue can trigger temporarily extended paths of motion by inciting the spread and sustained activation of motor networks. Asymmetries in connectivity can reshape the emerging dynamics of a neural network, see \cite{PJ,WJ}.

The derivation of mean-field limits for neuron models with sparse, non-symmetric connection maps had remained fully open and is critical for biological applications. The question was made even more pressing by recent progress that have lead to the first fully detailed representation of the connections between neurons, for some simple animals such as the {\it Drosophila}, see for example~\cite{HulHabFraTurTakal}.
The present article is the first to allow handling of arbitrary connection maps. However our investigations still leave a certain number of open questions which we briefly discuss below. 

\subsection{Integrate and fire models}
Many variants of models~\eqref{eqnmh}-\eqref{eqVi} have been studied to make the individual dynamics of a neuron more accurate. Those include adding noise to the system in various forms that often still fully compatible with our results.

However, so-called integrate and fire models are considered to better capture the dynamics of large networks; see for example~\cite{Bru,Bur} or \cite{Abb} for a historical perspective. Among many elaborate variants, a basic model is described here for simplicity. The state of each neuron is again represented by its membrane potential $V_i(t)$ which follows a simple SDE at almost all time,
\begin{equation}
dV_i(t)=b(V_i(t))\,dt+\sigma_N\,dW_i,\label{integratefire}
  \end{equation}
but with a probability to spike
\begin{equation}
 \mathbb{P}(\mbox{spike between}\ t\ \mbox{and}\ t+dt)\approx f(V(t))\,dt.\label{integratefire2}  
\end{equation}
In the case of a spike at time $t$, $V_i$ is reset at $0$ while all connected neurons are activated or inhibited
\begin{equation}
V_i(t_+)=0,\quad V_j(t_+)=V_j(t_-)+\alpha\,w_{ij}\quad\forall j\neq i. \label{integratefire3} 
\end{equation}
There exist many variations around those models, for example to include more complex equations instead of~\eqref{integratefire}, or deterministic spiking when $V_i$ reaches a given threshold. While \eqref{integratefire}-\eqref{integratefire3} would formally fit in a modified version of \eqref{eq1}, it includes {\em singular interactions} that prevent our results to apply. The nature of this singularity is however different from many of the other applications in this proposal, as it comes from the jumps in the $V_i$ at spiking times, instead of unbounded interactions at $X_i=X_j$. 

Still a full result on the derivation of integrate and fire models for general connectivities has been performed in~\cite{J-Z}, by building on the approach introduced here.
\subsection{Open problem: Learning mechanisms in neuron dynamics}
Long-term perspectives also include understanding how to incorporate learning mechanisms into system~\eqref{eq1}. This is a key question for biological and artificial neuron dynamics as learning is of course a critical function of the network.

Learning can simply translate in having time dependent connection weights $w_{ij}(t)$ that evolve according to the dynamics of the agents themselves, see~\cite{A-A} for instance; this corresponds for example to the well-known Hebb rule that neurons wire together if they fire together. 
  Considering again $N$ neurons {with activation $X_i$ and time-dependent connectivity or synaptic weights $w_{ij}(t)$ between neuron $i$ and $j$, a straightforward extension of~\eqref{eq1} simply reads
    \begin{equation}\label{N1}
   \left\{ \begin{array}{lll}
      &\displaystyle \frac{dX_i}{dt}=\sum_{j=1}^N w_{ij}(t)\, K(X_i,X_j)& \mbox{(+noise)},\\
      & \displaystyle \frac{dw_{ij}}{dt}=\eta\,(\Gamma(X_i,X_j)-w_{ij}) & \mbox{(+noise)},\\
      & \displaystyle X_i(0)=X_i^0,\quad w_{ij}(0)=w_{ij}^0, &
      \end{array} 
      \right.
    \end{equation}
for some plasticity function $\Gamma$, where $\eta$ is regarded as the learning rate parameter. 

How the mechanisms of learning can be reinforced or relaxed was specified in the neuronal context by Hebb \cite{H}. Hebb  proposed a theory of the adaptation of neurons in the brain during the learning process. The basic idea of his rule is based on the fact that neurons wire together if they fire together. Such a mechanism is based on the hypothesis that synchronous activation of cells (firing of neurons) leads to selectively pronounced increases in synaptic strength between these cells. From this process emerges the self-organized patterns. These postulates provide the neural basis for unsupervised complex learning of cognitive function in neural networks and may explain some processes that occur in the development of the nervous system.

The modeling is however made especially difficult in the present setting: changes in the $w_{ij}(t)$ need to be able to altering the structure of sparse networks while preserving bounds such as~\eqref{meanfieldscaling}. If all $w_{ij}$ are of the same order, which is typical of dense graphs, then system~\eqref{N1} likely has a similar behavior to~\eqref{eq1}. However for sparse graphs, we expect the values of the $w_{ij}$ to vary wildly and a straightforward system like \eqref{N1} just cannot correctly capture very different orders of magnitude in connectivities.

On the analytical side,  weights $w_{ij}$ that adapt in time can lead to stronger synchronization between agents, strengthening correlations between them and putting in doubt the propagation of independence. Useful hints can be derived from studies on the Kuramoto model~\cite{Ku1,Ku2} of coupled oscillators that was briefly mentioned above. The system is posed on the frequencies $\theta_i $ and can read, with learning,
\begin{equation}\label{KM1}
  \left\{
  \begin{array}{lll}
\displaystyle\frac{d \theta_i }{dt} = \Omega_i + \frac{1}{N}\sum_{j=1}^N  w_{ij} \sin(\theta_j - \theta_i),\\
\displaystyle\frac{dw_{ij}}{dt}=\eta\,(\Gamma(\theta_i,\theta_j)-w_{ij})\  \mbox{(+noise)},\\
\displaystyle\theta_i(0)=\theta_i^0,\quad w_{ij}(0)=w_{ij}^0,
  \end{array}\right.
\end{equation}
where the $\Omega_i$ are the natural frequencies of oscillators. 

For large $\eta$, the other variable in the model that is the choice of $ \Gamma $ takes on great relevance since, at the limit $\eta\to\infty$, we have that $w_{ij}(t)=\Gamma(\theta_i,\theta_j)$.  If for example $\Gamma$ is assumed to be $\Gamma(\theta_i,\theta_j) = \cos (\theta_i - \theta_j)$, the attraction between near oscillators is reinforced whereas repulsive interaction arises between apart phases. On the other hand if one takes~$\Gamma(\theta_i, \theta_j) = C|\sin (\theta_i - \theta_j)|$,  synchronization only slowly emerges due to the reduction in connectivities for nearby oscillators.

Hebbian-type dynamics can be considered for the Kuramoto model. This learning law ensures that the weight of the adaptive coupling increases if the phases of the oscillators are close. One possibility discussed in \cite{PPS} is to assume that the Hebbian-like plasticity function $\Gamma $ is given by
\begin{eqnarray}\label{C-2}
\Gamma(\theta,\theta'):=\frac{1}{c_{\alpha,\zeta}^\alpha |\theta-\theta'|_o^{2\alpha}},
\end{eqnarray}
where $|\theta-\theta'|_o$ is the orthodromic distance (to zero) over the unit circle and the parameter $c_{\alpha,\zeta}:=1-\zeta^{-1/\alpha}$, $\zeta\in (0,1]$ , has been chosen so that whenever two phases $\theta_i$ and $\theta_j$ stay at orthodromic distance $\sigma$ or larger, then the adaptive function $\Gamma$ predicts a maximum degree of connectedness not larger than $\zeta$ between such oscillators. This is another way to create and amplify correlations between neurons.

However, at this stage, time-evolving connectivities do not fit within our framework. At least formally, note that when $t\gg 1$, weights behave as $w_{ij}\approx \Gamma(\theta_j,\theta_j)$. If $\Gamma$ were bounded, it appears that our extended mean-field scaling should be satisfied (note that~\eqref{KM1} contains weights $\frac{w_{ij}}{N}$ instead of $w_{ij}$); but that does not apply anymore when using unbounded $\Gamma$ as the Hebbian-like plasticity mentioned above.
  \section{Technical proofs from subsection~\ref{subsec:graphonrepresentation} \label{technicalappendix}}
For the sake of completeness, we collect here the more technical proofs from subsection~\ref{subsec:graphonrepresentation}, starting with the proof of Lemma~\ref{bilinearbound}.  
\begin{proof}[Proof of Lemma~\ref{bilinearbound}.]
Set any $w\in L^\infty_\xi\mathcal{M}_\zeta\cap L^\infty_\zeta\mathcal{M}_\xi$, any $\phi\in L^\infty([0,\ 1])$ and any couple of sequences $\{w_n\}_{n\in \mathbb{N}}$ and $\{\phi_n\}_{n\in \mathbb{N}}$ as in the statement of the lemma. On the one hand, since $\phi\in L^1([0,\ 1])$, then
\[
\int_0^1 \psi(\xi)\left(\int_0^1\phi(\zeta)\,w(\xi,d\zeta)\right)\,d\xi=\int_0^1 \phi(\zeta)\left(\int_0^1 \psi(\xi)\,w(d\xi,\zeta)\right)\,d\zeta\leq \Vert w\Vert_{L^\infty_\zeta \mathcal{M}_\xi}\Vert \phi\Vert_{L^1}\Vert \psi\Vert_{L^\infty},
\]
for any $\psi\in C([0,\ 1])$. This implies that $\int_0^1\phi(\zeta)\,w(\cdot,d\zeta)\in \mathcal{M}([0,\ 1])$ and
\begin{equation}\label{E-wfL1Linfty-1-pre}
\left\Vert \int_0^1 \phi(\zeta)\,w(\xi,d\zeta)\right\Vert_{\mathcal{M}_\xi}\leq \Vert w\Vert_{L^\infty_\zeta\mathcal{M}_\xi}\Vert \phi\Vert_{L^1}.
\end{equation}
Moreover, given any $\psi\in C([0,\ 1])$ we have that $\psi\otimes \phi_n\rightarrow \psi\otimes \phi$ strongly in $L^1_\zeta C_\xi$. Since in addition $w_n\overset{*}{\rightharpoonup}w$ weakly-star in $L^\infty_\zeta\mathcal{M}_\xi=(L^1_\zeta C_\xi)^*$, then we also obtain
\[
\int_0^1 \psi(\xi)\left(\int_0^1\phi_n(\zeta)\,w_n(\xi,d\zeta)\right)\,d\xi\rightarrow \int_0^1\psi(\xi)\left(\int_0^1\phi(\zeta)\,w(\xi,d\zeta)\right)\,d\xi.
\]
Due to the arbitrariness of $\psi$, the above amounts to
\begin{equation}\label{E-wfweakcont-pre}
\int_0^1\phi_n(\zeta)\,w_n(\cdot,d\zeta)\overset{*}{\rightharpoonup} \int_0^1\phi(\zeta)\,w(\cdot,d\zeta),
\end{equation}
weakly-star in $\mathcal{M}([0,\ 1])$. On the other hand, set any $\psi\in C([0,\ 1])$ and another sequence $\{\tilde\phi_n\}_{n\in \mathbb{N}}$ as above consisting of continuous functions. It clearly exists by density of $C([0,\ 1])$ in $L^1([0,\ 1])$ and we can indeed assume that $\{\tilde \phi_n\}_{n\in \mathbb{N}}\in C([0,\ 1])$, $\tilde \phi_n\rightarrow \phi$ strongly in $L^1$ and $\Vert \tilde \phi_n\Vert_{L^\infty}\leq \Vert \phi\Vert_{L^\infty}$. Then,
\[
\int_0^1\psi(\xi)\left(\int_0^1 \tilde\phi_n(\zeta)\,w(\xi,d\zeta)\right)\,d\xi\leq \Vert w\Vert_{L^\infty_\xi\mathcal{M}_\zeta}\Vert \psi\Vert_{L^1}\Vert \tilde \phi_n\Vert_{L^\infty}\leq \Vert w\Vert_{L^\infty_\xi\mathcal{M}_\zeta}\Vert \psi\Vert_{L^1}\Vert \phi\Vert_{L^\infty}.
\]
By the above weak-* convergence, we obtain
\[
\int_0^1 \psi(\xi)\left(\int_0^1\phi(\zeta)\,w(\xi,d\zeta)\right)\,d\xi\leq \Vert w\Vert_{L^\infty_\xi\mathcal{M}_\zeta}\Vert \psi\Vert_{L^1}\Vert \phi\Vert_{L^\infty},
\]
for every $\psi\in C([0,\ 1])$. Noting that $C([0,\ 1])$ is dense in $L^1([0,\ 1])$, we find that the above Radon measure $\int_0^1\phi(\zeta)\,w(\cdot,d\zeta)$ actually belongs to $L^1([0,\ 1])^*=L^\infty([0,\ 1])$ and
\[
\left\Vert \int_0^1 \phi(\zeta)\,w(\xi,d\zeta)\right\Vert_{L^\infty_\xi}\leq \Vert w\Vert_{L^\infty_\xi\mathcal{M}_\zeta}\Vert \phi\Vert_{L^\infty},
\]
which implies $\eqref{wfL1Linfty}_2$. Using \eqref{E-wfL1Linfty-1-pre} we also infer $\eqref{wfL1Linfty}_1$ again because the Radon measure $\int_0^1 \phi(\zeta)\,w(\cdot,d\zeta)$ is absolutely continuous with respect to the Lebesgue measure and we already had an estimate in the total variation norm. To end the proof of \eqref{wfweakcont} note that $\int_0^1 \phi_n(\zeta)\,w_n(\cdot,d\zeta)$ are uniformly bounded in $L^1$ and $L^\infty$ thanks to the already proven estimates \eqref{wfL1Linfty}. Since we already had that they converge weakly-star in $\mathcal{M}([0,\ 1])$  to $\int_0^1 \phi(\zeta)\,w(\cdot,d\zeta)$ by \eqref{E-wfweakcont-pre}, we conclude that they also converge weakly-star in $L^\infty$.
  \end{proof}
We now turn to the proof of Proposition~\ref{existenceweak}.
\begin{proof}[Proof of Proposition~\ref{existenceweak}.]
The argument relies on a direct Picard fixed point. Specifically, let us define
\[
E:=L^\infty([0,\ t_*]\times [0,\ 1],\, W^{1,1}\cap W^{1,\infty}(\mathbb{R}^d)),\quad E_R:=\left\{f\in E:\,\Vert f\Vert_{L^\infty_{t,\xi}W^{1,1}_x\cap W^{1,\infty}_x}<R\right\},
\]
for any $R>0$ to be determined later. It is straightforward to check that $E_R$ is a closed subset of $L^1([0,\ t_*]\times [0,\ 1]\times \mathbb{R}^d)$ under the $L^1$ norm in all variables $t$, $\xi$ and $x$. Then, endowed with the $L^1$ norm, $E_R$ becomes a complete and convex metric space, which we shall use in the fixed point argument below.

Notice that the nonlinear PDE \eqref{independ2} can be reformulated as a fixed point equation in the usual way. Specifically, for any $f\in E$ let us set $\mathcal{L}f\in E$ to be the unique solution of the linear Cauchy problem
\begin{equation}\label{independ2-linear}
\left\{
\begin{array}{l}
\displaystyle \partial_t \mathcal{L}f(t,x,\xi)+\divop_x\left(\mathcal{L}f(t,x,\xi)\,\int_0^1 w(\xi,d\zeta) \int_{\R^d} K(x-y)\, f(t,y,\zeta)\,dy\right)=\nu\,\Delta_x \mathcal{L}f(t,x,\xi),\\
\mathcal{L}f(0,\cdot,\cdot)=f^0.
\end{array}
\right.
\end{equation}
Hence, the fixed points of the operator $\mathcal{L}$ amount to weak solutions of \eqref{independ2}. The rest of the proof will focus on proving that the operator $\mathcal{L}:E\longrightarrow E$ is well defined, and we have that $\mathcal{L}(E_R)\subset E_R$ and it is contractive in the $L^1$ norm for appropriate $R>0$.

\medskip

$\diamond$ {\it Step~1: Well-posedness of \eqref{independ2-linear}}.\\
We note that \eqref{independ2-linear} is a linear transport equation in $\mathcal{L}f$ parametrized by $\xi$ and whose velocity field $\mathcal{V}_f$ is given in Definition \ref{D-velocity-field}. Notice that integrating by parts we have
\[\nabla_x \mathcal{V}_f(t,x,\xi)=\int_0^1 w(\xi,d\zeta)\int_{\mathbb{R}^d}K(x-y)\otimes\nabla_y f(t,y,\zeta)\,dy,
\]
which is bounded uniformly in all variables by the regularity of $f\in E$. Therefore, the velocity field verifies $\mathcal{V}_f\in L^\infty([0,\ t_*]\times [0,\ 1],\,W^{1,\infty}(\mathbb{R}^d,\mathbb{R}^d))$. For $\nu=0$ and for {\it a.e.} value $\xi\in [0,\ 1]$ we can associate a unique and well-defined flow solving the characteristic system
\begin{equation}\label{E-characteristic-system}
\left\{
\begin{array}{l} 
\displaystyle \frac{dX_f}{dt}(t,s,x;\xi)=\mathcal{V}_f(t,X_f(t,s,x;\xi);\xi),\\
X_f(s,s,x;\xi)=x,
\end{array}
\right.
\end{equation}
in the sense of Caratheodory. Hence, the classical results ensure that the unique solution of \eqref{independ2-linear} can be obtained for example by the method of characteristics, namely, $\mathcal{L}f(t,\cdot,\xi)=X_f(t,0,\cdot\,;\xi)_{\#}f^0(\cdot,\xi)$ for $t\in [0,\ t_*]$ and {\it a.e.} $\xi\in [0,\ 1]$. In other words, we have that
\begin{equation}\label{E-independ2-linear-solution}
\mathcal{L}f(t,x,\xi)=f_0(X_f(0,t,x;\xi))\,J_f(0,t,x;\xi),
\end{equation}
where $J_f(t,s,x;\xi):=\det (\nabla_x X_f(t,s,x;\xi))$ denotes the Jacobian associated to the flow.

For $\nu>0$, we may instead use the heat kernel $G_\nu(t,x)$ and we can again classically obtain a solution through the mild formulation
\begin{equation}\label{mildformulation}
  \begin{split}
    &\mathcal{L}f(t,x,\xi)=G_\nu(t,.)\star_x f^0 (x,\xi)\\
    &\qquad+\int_0^t\int_{\R^d} G_\nu(t-s,x-y)\,\divop_y\left(\mathcal{L}f(s,y,\xi)\,\int_0^1 w(\xi,d\zeta) \int_{\R^d} K(y-z)\, f(s,z,\zeta)\,dz\right)\,ds\,dy.
\end{split}
  \end{equation}

\medskip

$\diamond$ {\it Step~2: A priori estimates of \eqref{independ2-linear}}.\\
We first recall classical estimates for advection and advection-diffusion equations. Consider any $v(t,x)\in L^\infty([0,\ t^*],\;W^{1,\infty}(\R^d))$, any weak solution $u$ in $L^1_{loc}$ to
\begin{equation}
\left\{\begin{array}{l}
\partial_t u+\divop_x(u\,v)=\nu\,\Delta_x u+R(t,x),\\
u(0,x)=u^0(x).
\end{array}\right.
\label{advectiondiffusion}\end{equation}
Then if $R\in L^1_{t,x}$, we have that $u\in L^\infty_t L^1_x$ with
\begin{equation}
\|u(t,\cdot)\|_{L^1_x}\leq \|u^0\|_{L^1_x}+\int_0^t \|R(s,\cdot)\|_{L^1_x}\,ds.\label{propL1}
\end{equation}
If $R\in L^1_t L^\infty_x$, then $u\in L^\infty_{t,x}$ with
\begin{equation}
\begin{split}
  \|u(t,\cdot)\|_{L^\infty_x}\leq &\|u^0\|_{L^\infty_x}\,\exp\left(t\, \|\divop_x v\|_{L^\infty_{t,x}}\right)\\
  &+\int_0^t \|R(s,\cdot)\|_{L^\infty_x}\,\,\exp\left((t-s)\, \|\divop_x v\|_{L^\infty_{t,x}}\right)\,ds.\label{propLinfty}
\end{split}
\end{equation}
If additional derivatives are available on $u^0$, $v$ and $R$, Sobolev bounds are also propagated. Provided that $v$ is $C^1$, $u$ is automatically a classical solution to \eqref{advectiondiffusion} so that
\[
\partial_t \nabla_x u+\divop(v\,\nabla_x u)+\divop(\nabla_x v\,u)=\nu\,\Delta_x \nabla_x u+\nabla_x R.
\]
In particular, we may apply estimates similar to the previous ones as $\nabla_x u$ solves \eqref{advectiondiffusion} with $\tilde R=\nabla_x R-\divop(\nabla_x v\,u)$. We consequently find that, if $u^0\in W^{1,1}_x$ and $R=0$ for simplicity, then
\[\begin{split}
\|\tilde R(t,\cdot)\|_{L^1_x}&\leq \|\divop v(t,\cdot)\|_{W^{1,\infty}_x}\,\|u(t,\cdot)\|_{L^1_x}+\|v(t,\cdot)\|_{W^{1,\infty}_x}\,\|\nabla_x u(t,\cdot)\|_{L^1_x}\\
&\leq \|\divop v(t,\cdot)\|_{W^{1,\infty}_x}\,\|u^0\|_{L^1_x}+\|v(t,\cdot)\|_{W^{1,\infty}_x}\,\|\nabla_x u(t,\cdot)\|_{L^1_x}.  
\end{split}
\]
Hence we obtain by Gronwall's lemma that
\begin{equation}
  \begin{split}
    \|\nabla_x u(t,\cdot)\|_{L^1_x}\leq &\|\nabla_x u^0\|_{L^1_x}\,\exp\left(t\,\|\divop_x v\|_{L^\infty_{t,x}}\right)\\
        &+\|u^0\|_{L^1}\,\int_0^t \|\divop v(s,\cdot)\|_{W^{1,\infty}_x}\,\exp\left((t-s)\,\| v\|_{L^\infty_{t} W^{1,\infty}_x}\right)\,ds.
\end{split}\label{propW11}
  \end{equation}
Similarly if $u^0\in W^{1,\infty}_x$ and $R=0$, then 
\[\begin{split}
\|\tilde R(t,\cdot)\|_{L^\infty_x}&\leq \|\divop v(t,\cdot)\|_{W^{1,\infty}_x}\,\|u(t,\cdot)\|_{L^\infty_x}+\|v(t,\cdot)\|_{W^{1,\infty}_x}\,\|\nabla_x u(t,\cdot)\|_{L^\infty_x}\\
&\leq \|\divop v(t,\cdot)\|_{W^{1,\infty}_x}\,\|u^0\|_{L^\infty_x}\,\exp\left(t\,\|\divop_x v\|_{L^\infty_{t,x}}\right)+\|v(t,\cdot)\|_{W^{1,\infty}_x}\,\|\nabla_x u(t,\cdot)\|_{L^\infty_x},  
\end{split}
\]
and one has that, still by Gronwall's lemma,
\begin{equation}
  \begin{split}
  \|\nabla_x u(t,.)\|_{L^\infty_x}\leq &\|\nabla_x u^0\|_{L^\infty_x}\,\exp\left(t\,\|\divop_x v\|_{L^\infty_{t,x}}\right)\\
        &+\|u^0\|_{L^\infty}\,\int_0^t \|\divop v(s,.)\|_{W^{1,\infty}_x}\,\exp\left((t-s)\,(\| v\|_{L^\infty_{t} W^{1,\infty}_x}+\|\divop_x v\|_{L^\infty_{t,x}})\right)\,ds.
\end{split}\label{propW1infty}
\end{equation}

We may now propagate estimates to show that the solution $\mathcal{L}f$ lies in $E$, and indeed $\mathcal{L}(E_R)\subset E_R$ for an appropriate choice of $R>0$. We note that Equation~\eqref{independ2-linear} is immediately under the form \eqref{advectiondiffusion} where $\xi$ is only a parameter and
\[
v_\xi(t,x)=V_f(t,x,\xi)=\int_0^1 w(\xi,d\zeta) \int_{\R^d} K(x-y)\, f(t,y,\zeta)\,dy.
\]
From here, since $f^0\in L^\infty_\xi L^1_x$ we obtain that
\begin{align}\label{E-independ2-linear-estimate1}
\begin{aligned}
&\Vert \mathcal{L}f(t,\cdot,\xi)\Vert_{L^1_x}=\Vert f^0(\cdot,\xi)\Vert_{L^1_x},\\
  &\Vert \mathcal{L}f(t,\cdot,\xi)\Vert_{L^\infty_x}\leq \Vert f^0(\cdot,\xi)\Vert_{L^\infty}\exp\left(t_*\,\Vert \divop_x \mathcal{V}_f\Vert_{L^\infty_{t,x,\xi}}\right),
\end{aligned}
\end{align}
for $t\in [0,\ t_*]$ and {\it a.e.} $\xi\in [0,\ 1]$.

From the definition of $V_f$, we also directly obtain that
\begin{align}\label{E-independ2-linear-estimate3}
\begin{aligned}
  &\Vert \divop_x V_f\Vert_{L^\infty_{t,x,\xi}}\leq \Vert K\Vert_{L^\infty}\Vert w\Vert_{L^\infty_\xi\mathcal{M}_\zeta\cap L^\infty_\zeta\mathcal{M}_\xi}\Vert \nabla_x f\Vert_{L^\infty_{t,\xi}L^1_x}\\
  &\Vert V_f\Vert_{L^\infty_{t,\xi} W^{1,\infty}_{x}}\leq \Vert K\Vert_{L^\infty}\,\Vert w\Vert_{L^\infty_\xi\mathcal{M}_\zeta\cap L^\infty_\zeta\mathcal{M}_\xi}\,\Vert \nabla_x f\Vert_{L^\infty_{t,\xi}L^1_x}\\
&\|\divop V_f\Vert_{L^\infty_{t,\xi} W^{1,\infty}_x}\leq \Vert \divop K\Vert_{L^1}\,\Vert w\Vert_{L^\infty_\xi\mathcal{M}_\zeta\cap L^\infty_\zeta\mathcal{M}_\xi}\,\Vert \nabla_x f\Vert_{L^\infty_{t,\xi}L^\infty_x}.\\
\end{aligned}
\end{align}

Putting \eqref{E-independ2-linear-estimate1}, and \eqref{E-independ2-linear-estimate3} together into \eqref{propW11} and \eqref{propW1infty} implies that
\begin{align}\label{E-independ2-linear-estimate}
\begin{aligned}
&\Vert \mathcal{L}f\Vert_{L^\infty_{t,\xi}W^{1,1}_x\cap W^{1,\infty}_x}\leq C_1\Vert f^0\Vert_{L^\infty_\xi W^{1,1}_x\cap W^{1,\infty}}\\
&\hskip3.4cm\times \exp\left(C_2\,t_*\,(\Vert K\Vert_{L^\infty}+\Vert \divop K\Vert_{L^1})\Vert w\Vert_{L^\infty_\xi\mathcal{M}_\zeta\cap L^\infty_\zeta\mathcal{M}_\xi}\Vert f\Vert_{L^\infty_{t,\xi}W^{1,1}_x\cap W^{1,\infty}_x}\right),
\end{aligned}
\end{align}
for some universal $C_1,C_2\geq 1$. In particular, the above estimate show that $\mathcal{L}f\in E$ for every $f\in E$.

\medskip

$\diamond$ {\it Step~3: Invariance of $E_R$ and contractivity of $\mathcal{L}$}.\\
Set now any $R>C_1\Vert f^0\Vert_{L^\infty_\xi W^{1,1}_x W^{1,\infty}_x}$ and note that the above estimates also show that $\mathcal{L}(E_R)\subset E_R$ if the maximal time $t_*$ is taken sufficiently small, namely,
\[
0<t_*\leq t_1(R):=\frac{1}{C_2(\Vert K\Vert_{L^\infty}+\Vert \divop K\Vert_{L^1})\Vert w\Vert_{L^\infty_\xi\mathcal{M}_\zeta\cap L^\infty_\zeta\mathcal{M}_\xi} R}\log\left(\frac{R}{C_1\Vert f^0\Vert_{L^\infty_\xi W^{1,1}_x\cap W^{1,\infty}_x}}\right).
\]
Our next goal is to show that $\mathcal{L}$ is contractive under the $L^1$ norm. To do so, let us consider $f,g\in E_R$ and note that $f-g$ follows the following linear transport equation with a source term
\[
\left\{
\begin{array}{l}
\displaystyle\partial_t(\mathcal{L}f-\mathcal{L}g)+\divop_x(\mathcal{V}_g\,(\mathcal{L}f-\mathcal{L}g))+\divop_x((\mathcal{V}_f-\mathcal{V}_g)\,\mathcal{L}f)=\nu\Delta_x (\mathcal{L}f-\mathcal{L}g),\\
\displaystyle (\mathcal{L}f-\mathcal{L}g)(0,\cdot,\cdot)=0.
\end{array}\right.
\]
Taking $L^1$ norms with respect to $x$ and $\xi$ yields
\begin{align}\label{E-independ2-linear-contraction1}
\begin{aligned}
\Vert (\mathcal{L}f-\mathcal{L}g)(t,\cdot,\cdot)\Vert_{L^1_{x,\xi}}&\leq \Vert \divop_x((\mathcal{V}_f-\mathcal{V}_g)\,\mathcal{L}f)\Vert_{L^1_{t,x,\xi}}\\
&\leq \Vert \divop_x(\mathcal{V}_f-\mathcal{V}_g)\Vert_{L^\infty_{t,\xi}L^1_x}\Vert \mathcal{L}f\Vert_{L^\infty_{t,x,\xi}}+\Vert \mathcal{V}_f-\mathcal{V}_g\Vert_{L^\infty_{t,\xi} L^1_x}\Vert \nabla_x\mathcal{L}f\Vert_{L^\infty_{t,x,\xi}},
\end{aligned}
\end{align}
for $t\in [0,\ t_*]$. By the estimate $\eqref{wfL1LinftyBx}_1$ in Lemma \ref{bilinearboundBx} we have
\begin{align}\label{E-independ2-linear-contraction2}
\begin{aligned}
&\Vert \mathcal{V}_f-\mathcal{V}_g\Vert_{L^1_{t,x,\xi}}\leq \Vert K\Vert_{L^1}\Vert w\Vert_{L^\infty_\xi\mathcal{M}_\zeta\cap L^\infty_\zeta\mathcal{M}_\xi} \Vert f-g\Vert_{L^1_{t,x,\xi}},\\
&\Vert \divop_x(\mathcal{V}_f-\mathcal{V}_g)\Vert_{L^1_{t,x,\xi}}\leq \Vert \divop K\Vert_{L^1}\Vert w\Vert_{L^\infty_\xi\mathcal{M}_\zeta\cap L^\infty_\zeta\mathcal{M}_\xi}\Vert f-g\Vert_{L^1_{t,x,\xi}}.
\end{aligned}
\end{align}
Then, putting \eqref{E-independ2-linear-contraction1} and \eqref{E-independ2-linear-contraction2} together, and using that $\mathcal{L}f\in E_R$ by the above step, we have that
\begin{equation}\label{E-independ2-linear}
\Vert \mathcal{L}f-\mathcal{L}g\Vert_{L^1_{t,x,\xi}}\leq R\,t_*(\Vert K\Vert_{L^1}+\Vert \divop K\Vert_{L^1})\Vert w\Vert_{L^\infty_\xi\mathcal{M}_\zeta\cap L^\infty_\zeta\mathcal{M}_\xi}\Vert f-g\Vert_{L^1_{t,x,\xi}}.
\end{equation}
Take $t_*$ sufficiently small so that $t_*\leq t_1(R)$ as above and, in addition
\[
0<t_*< t_2(R):=\frac{1}{R\,(\Vert K\Vert_{L^1}+\Vert \divop K\Vert_{L^1})\Vert w\Vert_{L^\infty_\xi\mathcal{M}_\zeta\cap L^\infty_\zeta\mathcal{M}_\xi}}.
\]
Hence $\mathcal{L}$ is contractive on $E_R$ under the $L^1$ norm by the stability estimate \eqref{E-independ2-linear}. By the Banach fixed point theorem there must exist a (unique) fixed point $f\in E_R$ of $\mathcal{L}$, leading to a weak solution of \eqref{independ2} in $[0,\ t_*]$. Note that $t_*$ must be taken small enough so that $0<t_*<\min\{t_1(R),t_2(R)\}$. However, restarting the argument from $t_*$ allows extending this weak solution solution to any finite time interval as usual, since the above a priori estimate \eqref{E-independ2-linear-estimate} above does not blow up in time.
  \end{proof}

We finish this technical appendix with the proof of Lemma~\ref{gooddensity}.

\begin{proof}[Proof of Lemma~\ref{gooddensity}]
  As usual, one uses an approximation by convolution. Because $w$ is defined on $[0,\ 1]^2$, we first extend  $w$ as a periodic function on with period $1$ in both variables. Similarly we naturally consider any function on $[0,\ 1]$ as periodic, allowing to naturally use convolutions.  Then, choose any convolution kernel $L_n\geq 0$ with $L_n\in C^{\infty}_c(\R)$ for any fixed $n$ and define $w_n=L_n\otimes L_n\star_{\xi,\zeta} w$. Since $w\in \mathcal{M}_{\xi,\zeta}$, we immediately note that $w_n\in L^\infty_{\xi,\zeta}$.

  Next, for any $f\in L^1_\xi C_\zeta$, we trivially have that $\|L_n\otimes L_n \star f\|_{L^1_\xi L^\infty_\zeta}\leq \|f\|_{L^1_\xi L^\infty_\zeta}$. By duality, we deduce that $w_n$ is uniformly bounded in $L^\infty_\xi \mathcal{M}_\zeta\cap L^\infty_\zeta \mathcal{M}_\xi$ and, hence, in $L^\infty_\xi L^1_\zeta\cap L^\infty_\zeta L^1_\xi$, since $w_n\in L^\infty_{\xi,\zeta}$.

  Consider now any $\phi_n$ converging weakly to $\phi$ in $H^1([0,\ 1])$, we have that 
  \[
  \begin{split}
\int_0^1 \phi_n(\zeta)\,L_n\star_\zeta w(\xi,\zeta)\,d\zeta=
\int_0^1 \tilde L_n\star_\zeta \phi_n(\zeta)\,w(\xi,\zeta)\,d\zeta,
  \end{split}
  \]
  with $\tilde L_n(\xi)=L_n(-\xi)$. By Lemma~\ref{bilinearbound},
  for any $\psi\in L^\infty([0,\ 1])$, we deduce that
  \[
\int_0^1 \psi(\xi)\,\left(\int_0^1 \phi_n(\zeta)\,L_n\star_\zeta w(\xi,\zeta)\,d\zeta-\int_0^1 \phi(\zeta)\, w(\xi,d\zeta)\right)\,d\xi\leq \|\psi\|_{L^\infty}\,\|w\|_{L^\infty_\xi \mathcal{M}_\zeta\cap L^\infty_\zeta \mathcal{M}_\xi}\,\|\phi-\tilde L_n\star\phi_n\|_{L^1}.
\]
Since $H^1$ is compactly embedded in $L^1$, this shows that
\[
\left\|\int_0^1 \phi_n(\zeta)\,L_n\star_\zeta w(\xi,\zeta)\,d\zeta-\int_0^1 \phi(\zeta)\, w(\xi,d\zeta)\right\|_{L^1_\xi}\to 0,\quad\mbox{as}\ n\to\infty.
\]
Still by Lemma~\ref{bilinearbound} since $H^1\subset L^\infty$, $\int_0^1 \phi(\zeta)\, w(\xi,d\zeta)$ is a fixed function in $L^\infty$ and consequently its convolution by $L_n$ converges in $L^1$ (and in all $L^p$, $p<\infty$). Therefore
\[
\begin{split}
  &\left\|\int_0^1 \phi_n(\zeta)\,w_n(\xi,\zeta)\,d\zeta-\int_0^1 \phi(\zeta)\, w(\xi,d\zeta)\right\|_{L^1_\xi}\\
  &\quad\leq \left\|L_n\star_\xi\left(\int_0^1 \phi_n(\zeta)\,L_n\star_\zeta w(\xi,\zeta)\,d\zeta-\int_0^1 \phi(\zeta)\, w(\xi,d\zeta)\right)\right\|_{L^1_\xi}\\
&\qquad+\left\|L_n\star_\xi\int_0^1 \phi_n(\zeta)\,w(\xi,\zeta)\,d\zeta-\int_0^1 \phi(\zeta)\, w(\xi,d\zeta)\right\|_{L^1_\xi}\to 0,
\end{split}
\]
proving the convergence of $w_n$ in $L^1_\xi H^{-1}_\zeta$.
\end{proof}



\begin{thebibliography}{99}
\bibitem{Abb} L. F. Abbott, {\em Lapicque's introduction of the integrate-and-fire model neuron (1907)}, Brain Res. Bull. {\bf 50(5-6)} (1999), 303--4. 

\bibitem{A-A} T. Aoki, T. Aoyagi, {\em Co-evolution of Phases and Connection Strengths in a Network of Phase Oscillators}, Phys. Rev. Lett. \textbf{102} (2009), 034101.

\bibitem{BacSze} A. Backhausz, B. Szegedy, {\em Action convergence of operators and graphs}, Canad. J. Math. (2020), 1--50.
  
\bibitem{bara} A. L. Barab\'asi, R. Albert, {\it Emergence of scaling in random networks}, Science {\bf 286} (1999), 509--512. 
  
\bibitem{Boi11} E. Boissard, {\em Probl\`emes d'interaction discret-continu et distances de Wasserstein.} PhD thesis, Universit\'e de Toulouse III, 2011.

\bibitem{BT} S. Bochner and A. E. Taylor, {\em Linear Functionals on Certain Spaces of Abstractly-Valued Functions}, Ann. of Math. {\bf 39(4)} (1938), 913--944.

\bibitem{BCCZ1} C. Borgs, J. T. Chayes, H. Cohn, Y. Zhao, {\em An $L^p$ theory of sparse graph convergence I: limits, sparse random graph models, and power law distributions}, Trans. Am. Math. Soc. {\bf 372(5)} (2019), 3019--3062.

\bibitem{BCCZ2} C. Borgs, J. T. Chayes, H. Cohn, Y. Zhao, {\em An $L^p$ theory of sparse graph convergence II: LD convergence, quotients, and right convergence}, Ann. Probab. {\bf 46}(1) (2018), 337--396.

\bibitem{BCKL} C. Borgs, J. T. Chayes, J. Kahn, L. Lov{\'a}sz, {\it Left and right convergence of graphs with bounded degree}, Random Struct. Algorithms {\bf 42} (2013), 1--28.

\bibitem{BCLSV06} C. Borgs, J. T. Chayes, L. Lov\'{a}sz, V. S\'{o}s, K. Vesztergombi, {\em Counting graph homomorphisms}. In	Topics in discrete mathematics (eds.\ M. Klazar, J. Kratochv\'{\i }l, M. Loebl, J. Matou\v {s}ek, R. Thomas, and P. Valtr), {S}pringer, 2006, {315--371}.
  
\bibitem{BCLSV2} C. Borgs, J. T. Chayes, L. Lov\'{a}sz, V. S\'{o}s,  K. Vesztergombi,  {\em Convergent graph sequences {I}: subgraph frequencies, metric properties, and testing}, Adv. Math. {\bf 219} (2008), 1801--1851.

\bibitem{BCLSV3} C. Borgs, J. T. Chayes, L. Lov\'{a}sz, V. S\'{o}s, K. Vesztergombi,  {\em Convergent graph sequences II: multiway cuts and statistical physics}, Ann. Math. {\bf 176} (2012), 151--219.
  
\bibitem{BH} W. Braun, K. Hepp, {\em The Vlasov dynamics and its fluctuations in the 1/N limit of interacting classical particles}, Comm. Math. Phys. {\bf 56(2)} (1977),101--113.

\bibitem{Bru} N. Brunel, {\em Dynamics of sparsely connected networks of excitatory and inhibitory spiking neurons}, J. Comput. Neurosci. {\bf 8(3)} (2000), 183--208.

\bibitem{Bur} A. N. Burkitt, {\it A review of the integrate-and-fire neuron model: I. Homogeneous synaptic input}, Biol. Cybern. {\bf 95(1)} (2006), 1--19. 
    
\bibitem{CCP} M. J. C\'aceres, J. A. Carrillo, B. Perthame, {\em Analysis of nonlinear noisy integrate \& fire neuron models: blow-up and steady states}, J. Math. Neurosci. {\bf 1} (2011), Art. 7, 33 pp.

\bibitem{CP} M. J. C\'aceres, B. Perthame, {\em Beyond blow-up in excitatory integrate and fire neuronal networks: refractory period and spontaneous activity}, J. Theor. Biol. {\bf 350} (2014), 81--89.

\bibitem{CaGoGuSc} J. A. Carrillo, M. d. M. Gonzalez, M. P. Gualdani, M. E. Schonbek, {\em Classical solutions for a nonlinear Fokker-Planck equation arising in computational neuroscience}, Comm. Part. Differ. Equat. {\bf 38(3)} (2013), 385--409.
    
\bibitem{CPSS} J. A. Carrillo, B. Perthame, D. Salort, D. Smets,  {\em Qualitative Properties of Solutions for the Noisy Integrate \& Fire model in Computational Neuroscience}, Nonlinearity {\bf 28 (9)} (2015), 3365.

\bibitem{CM1} H. Chiba, G. S. Medvedev, {\em The mean field analysis for the Kuramoto model on graphs I. The mean field equation and transition point formulas}, Discrete Contin. Dyn. Syst. Ser. A  {\bf 39(1)} (2019), 131--155.

\bibitem{CM2} H. Chiba, G. S. Medvedev, {\em The mean field analysis of the Kuramoto model on graphs II. Asymptotic stability of the incoherent state, center manifold reduction, and bifurcations}, Discrete Contin. Dyn. Syst. Ser. A  {\bf 39(7)} (2019), 3897--3921.

\bibitem{CMM} H. Chiba, G. S. Medvedev, M. S. Muzuhara, {\em Bifurcations in the Kuramoto model on graphs}, Chaos {\bf 28} (2018), 073109.

 \bibitem{CoBrGoWa} A. Compte, N. Brunel, P. S. Goldman-Rakic, X.-J. Wang, {\em Synaptic mechanisms and network dynamics underlying spatial working memory in a cortical network model}, Cereb. Cortex {\bf 10} (2000), 910--923.
 
\bibitem{Compte} A. Compte, M. V. Sanchez-Vives, D. A. McCormick, X.-J. Wang, {\em Cellular and network mechanisms of slow oscillatory activity ($<$1 Hz) and wave propagations in a cortical network model}, J. Neurophysiol. {\bf 89} (2003), 2707--2725.

\bibitem{Cop22}  F. Coppini, {\em Long time dynamics for interacting oscillators on graphs}, The Annals of Applied
Probability, {\bf 32} (1) 2022.

\bibitem{CDG} F. Coppini, H. Dietert, G. Giacomin, {\em A law of large numbers and large deviations for interacting diffusions on Erd\"os-R\'enyi graphs}, Stoch. Dyn. {\bf 20} (2020), 2050010.

\bibitem{CLP23}  F. Coppini, E. Lu\c con, C. Poquet, {\em Central limit theorems for global and local empirical measures of diffusions on Erd\"os-R\'enyi graphs}, Electronic Journal of Probability,
{\bf 28} (2023).

\bibitem{dLR} T. de La Rue, {\em Espaces de Lebesgue}, S\'eminaire de probabilit\'es de Strasbourg, Tome 27 (1993), pp. 15-21.

\bibitem{DeInRuTa2} F. Delarue, J. Inglis, S. Rubenthaler, E. Tanr\'e, {\em Particle systems with a singular mean-field self-excitation. Application to neuronal networks}, Stoch. Process. Their Appl. 125 (2015) 2451--2492. 

\bibitem{DeInRuTa} F. Delarue, J. Inglis, S. Rubenthaler, E. Tanr\'e, et al, {\em Global solvability of a networked integrate-and-fire model of McKean-Vlasov type}, Ann. Appl. Probab. {\bf 25(4)} (2015), 2096--2133.
  
\bibitem{DGL} S. Delattre, G. Giacomin, E. Lu\c con, {\em A note on dynamical models on random graphs and Fokker-Planck equations}, J. Stat. Phys. {\bf 165(4)} (2016), 785--798. 

\bibitem{DemGalLocPre} A. De Masi, A. Galves, E. L\"ocherbach, E. Presutti, {\em Hydrodynamic limit for interacting neurons}, J. Stat. Phys. {\bf 158} (2015), 866--902.

\bibitem{DU} J. Diestel and J. J. Uhl Jr, {\em Vector Measures}, Mathematical Surveys and Monographs, vol. 15, American Mathematical Society, 1977.
  
\bibitem{DY95} V. Dobri\'c, J. E. Yukich, {\em Asymptotics for transportation cost in high dimensions}, J. Theoret. Probab. {\bf 8(1)} (1995), 97--118.  
  
\bibitem{Do} R. L. Dobrushin, {\em Vlasov equations},  Funct. Anal. its Appl. {\bf 13(2)} (1979), 115--123.

\bibitem{FitzHugh} R. FitzHugh, {\em Impulses and physiological states in theoretical models of nerve membrane}, Biophysical J. {\bf 1} (1961), 445--466.

\bibitem{FPZ} F. Flandoli, E. Priola, G. Zanco, {\em A mean-field model with discontinuous coefficients for neurons with spatial interaction},  Discrete Contin. Dyn. Syst. Ser. A {\bf 39} (2019), 3037--3067.

\bibitem{FouGui} N. Fournier, A. Guillin, {\em On the rate of convergence in Wasserstein distance of the empirical measure}, Probab. Theory Relat. Fields {\bf 162} (2015), 707--738.  
  
\bibitem{FouLoc} N. Fournier, E. L\"ocherbach, {\em On a toy model of interacting neurons}, Ann. Inst. H. Poincar\'e Probab. Statist. {\bf 52(4)} (2016), 1844--1876.

\bibitem{GerKis} W. Gerstner, W. M. Kistler, {\em Spiking neuron models: Single neurons, populations, plasticity}. Cambridge university press, 2002.

\bibitem{GkoKueh} M. A. Gkogkas, C. Kuehn, {\em Graphop Mean-Field Limits for Kuramoto-Type Models},  SIAM J. Appl. Dyn. Syst. {\bf 21(1)} (2022), 248--283.
  
\bibitem{golse} F. Golse, {\em On the dynamics of large particle systems in the mean field limit Macroscopic and Large Scale Phenomena: Coarse Graining, Mean Field Limits and Ergodicity}, Springer, Berlin 2016, 1--144.

\bibitem{H-J} M. Hauray, P. E. Jabin, {\em Particles approximations of Vlasov equations with singular forces: Propagation of chaos}, Ann. Sci. Ec. Norm. Super. \textbf{48(4)}, 891--940 (2015).

\bibitem{H} D.O. Hebb, {\em The Organization of Behavior}, Wiley, New York, 1949.

\bibitem{HH} A. L. Hodgkin, A. F. Huxley, {\em A quantitative description of membrane current and its application to conduction and excitation in nerve}, J. Physiol. {\bf 117} (1952), 500--544.

 \bibitem{HulHabFraTurTakal} B. K. Hulse, H. Haberkern, R. Franconville, D. B. Turner-Evans, S. Y. Takemura, T. Wolff, M. Noorman, M. Dreher, C. Dan, R. Parekh, A. M. Hermundstad, G. M. Rubin, V. Jayaraman, {\em A connectome of the Drosophila central complex reveals network motifs suitable for flexible navigation and context-dependent action selection}, eLife {\bf 10} (2021), e66039.
 
\bibitem{IT-77} A. Ionescu Tulcea, C. Ionescu Tulcea, {\em Topics in the theory of lifting}. Ergebnisse der Mathematik und ihrer Grenzgebiete. 2. Folge, vol. 48, Springer-Verlag, Berlin, Heidelberg, 1977.

\bibitem{Jab} P. E. Jabin, {\em A review of the mean field limits for Vlasov equations}, Kinet. Relat. Mod. {\bf 7} (2014), 661--711.
  
\bibitem{J-W} P. E. Jabin, Z. Wang, {\em Mean field limit and propagation of chaos for Vlasov systems with bounded forces}, J. Funct. Anal. \textbf{271(12)} (2016), 3588--3627.

\bibitem{JW2} P.-E. Jabin, Z. Wang, {\em Mean field limit for stochastic particle systems}. In: Active Particles, Volume 1, Theory, Models, Applications, pp 379--402.  Birkhauser-Springer, Boston  (2017). 

\bibitem{J-W-3} P.-E. Jabin, Z. Wang, {\em Quantitative estimates of propagation of chaos for stochastic systems with $W^{-1,\infty}$ kernels}, Invent. Math. \textbf{214(1)} (2018), 523--591.

 \bibitem{J-Z} P.-E. Jabin, D. Zhou, {\em The mean-field Limit of sparse networks of integrate and fire neurons}, Preprint arXiv:2309.04046.

\bibitem{KueXu} C. Kuehn, C. Xu, {\em Vlasov equations on digraph measures}, {\bf 339} (2022),  J. Differ. Equ. 261--349.

\bibitem{KalMed}  D. Kaliuzhnyi-Verbovetskyi, G.S. Medvedev, {\em The Mean Field Equation for the Kuramoto Model on Graph Sequences with Non-Lipschitz Limit}, SIAM J. Math. Anal., 50(3), 2441–2465.
  
\bibitem{KunLovSze} D. Kunszenti-Kov\'acs, L. Lov\'asz, B. Szegedy, {\em Measures on the square as sparse graph limits}, J. Comb. Theory. Ser. B {\bf 138} (2019), 1--40.
  
\bibitem{Ku1} Y. Kuramoto,  {\em Chemical Oscillations, waves and turbulence.} Springer-Verlag, Berlin, 1984.

\bibitem{Ku2} Y. Kuramoto,  {\em International symposium on mathematical problems in mathematical physics}, Lecture Notes in Theoretical Physics. \textbf{30}  (1975), 420

\bibitem{Lacker} D. Lacker, {\em Hierarchies, entropy and quantitative propagation of chaos for mean field diffusions}. Preprint arXiv:2105.02983.

\bibitem{LRW23}  D. Lacker, K. Ramanan, R. Wu, {\em Local weak convergence for sparse networks of interacting processes}, The Annals of Applied Probability, {\bf 33} (2), 2023.

\bibitem{L} C. Lancellotti,  {\em On the Vlasov limit for systems of nonlinearly coupled oscillators without noise}. Transport Theor. Stat. Phys. \textbf{34} (2005) 523--535.

\bibitem{Le-bas} J. M. Levine, J. Bascompte, P. B. Adler, S. Allesina, {\em Beyond pairwise coexistence: biodiversity maintenance in complex communities}, Nature {\bf 456} (2017), 56--64.

\bibitem{Lo} L. Lov\'{a}sz, {\em Large Networks and Graph Limits}, AMS Colloquium Publications
Volume: 60, 475 pp, 2012.

\bibitem{LS} L. Lov\'{a}sz, B. Szegedy, {\em Limits of dense graph sequences}, J. Comb. Theory. Ser. B {\bf 96} (2006), 933--957.

\bibitem{Lucon} E. Lu\c{c}on, {\em Quenched asymptotics for interacting diffusions on inhomogeneous random graphs},  Stoch. Process. Their Appl. {\bf 130(11)} (2020), 6783--6842.
  
\bibitem{MG1} M. Mattia, P. Del Giudice, {\em Efficient event-driven simulation of large networks of spiking neurons and dynamical synapses}, Neural Comput. {\bf 12(10)} (2000), 2305--2329.

\bibitem{Medvedev2} G. Medvedev, {\em The Nonlinear Heat Equation on Dense Graphs and Graph Limits}, SIAM J. Math. Anal. {\bf 46(4)} (2014), 883--898.

\bibitem{Medvedev} G. Medvedev, {\em The continuum limit of the Kuramoto model on sparse random graphs}, Commun. Math. Sci. {\bf 17(4)} (2019), 883--898.
    
\bibitem{Nagumo} J. Nagumo, S. Arimoto, S. Yoshizawa, {\em An active pulse transmission line simulating nerve axon}, Proc. IRE. {\bf 50} (1962), 2061--2070.

\bibitem{N} H. Neunzert, {\em An introduction to the nonlinear Boltzmann--Vlasov equation}. In Kinetic Theories and the Boltzmann Equation, Lecture Notes in Mathematics vol. \textbf{1048}. Springer-Verlag, 1984.

\bibitem{NW} H. Neunzert, J. Wick, {\em Die Approximation der Lsungvon Integro-Differentialgleichungen durch endliche Punktmengen}, In R. Ansorge and W. T\"ornig, editors, Numerische Behandlung nichtlinearer Integrodifferential - und Differentialgleichun-gen, volume 395 of Lecture Notes in Mathematics, pages 275--290. Springer, Berlin, Heidelberg, 1974.

\bibitem{ORS20} R.I. Oliveira, G.H. Reis, L. M. Stolerman, {\em Interacting diffusions on sparse
graphs: hydrodynamics from local weak limits}, Electronic Journal of Probability, {\bf 25}:135, 2020.

\bibitem{O} A. Omurtag, B.W. Knight, L. Sirovich, {\it On the simulation of large populations of neurons}, J. Comput. Neurosci. {\bf 8(1)} (2000), 51--63.

\bibitem{PaPeSa} K. Pakdaman, B. Perthame, D. Salort, {\em Dynamics of a structured neuron
population}, Nonlinearity, {\bf 23(1)} (2010), 55--75.

\bibitem{PPS} J. Park, D. Poyato, J. Soler, {\em Filippov trajectories and clustering in the Kuramoto model with singular couplings}, J. Eur. Math. Soc. {\bf 23} (2021), 3193--3278.

\bibitem{MG2} E. Pastorelli, C. Capone, F. Simula, M. V. Sanchez-Vives, P. Del Giudice, M. Mattia, P. S. Paolucci, {\em Scaling of a large-scale simulation of synchronous slow-wave and asynchronous awake-like activity of a cortical model with long-range interconnections}, Front. Syst. Neurosci. {\bf 13} (2019), 33.

\bibitem{PS} B. Perthame, D. Salort, {\em On a voltage-conductance kinetic system for integrate \& fireneural networks}, Kinet. Relat. Models {\bf 6} (2013),  841--864.

\bibitem{Phan} D. Pham, D. Karaboga, {\em Intelligent optimisation techniques: genetic algorithms, tabu search, simulated annealing and neural networks}, Springer Science \& Business Media, 2012.

\bibitem{PhPaChVi} J. Pham, K. Pakdaman, J. Champagnat, J.-F. Vibert, {\em Activity in sparsely connected excitatory neural networks: effect of connectivity}, Neural Netw. {\bf 11(3)} (1998),415--434.
  
\bibitem{PJ} A. S. Pillai, V. K. Jirsa, {\em Symmetry Breaking in Space-Time Hierarchies Shapes Brain Dynamics and Behavior}, Neuron {\bf 94(5)} (2017), 1010--1026.

\bibitem{PiShPaShLiChSi} J. W. Pillow, J. Shlens, L. Paninski, A. Sher, A. M. Litke, E. J. Chichilnisky, E. P. Simoncelli, {\em Spatio-temporal correlations and visual signalling in a complete neuronal population}, Nature {\bf 454} (2008), 995--999.

\bibitem{Po} D. Poyato, {\em Filippov flows and mean-field limits in the kinetic singular Kuramoto model}. Preprint arXiv:1903.01305.

\bibitem{P-S} D. Poyato, J. Soler, {\em Euler-type equations and commutators in singular and hyperbolic limits of kinetic Cucker-Smale models}, Math. Models Methods Appl. Sci. \textbf{27(6)} (2017), 1089--1152.

\bibitem{RachevR} S.T. Rachev and L. Ruschendorf. Mass Transportation Problems. Springer-Verlag, 1998.

\bibitem{RBW} A. Renart, N. Brunel, X. J. Wang, {\em Mean-field theory of irregularly spiking neuronal populations and working memory in recurrent cortical networks}, Computational neuroscience: A comprehensive approach, 431--490, J Feng Editor, CRC Press, 2003.

\bibitem{spohn} H. Spohn, {\em Dynamics of Charged Particles and their Radiation Field}, Cambridge University Press, 2004.

\bibitem{sporns} O. Sporns, {\em Networks of the Brain}, Cambridge, MA: MIT Press, 2010. 

\bibitem{sze} E. Szemer\'edi, {\em On sets of integers containing no k elements in arithmetic progression}, Acta Arith. {\bf 27} (1975), 299--345.

\bibitem{Sznit} A.-S. Sznitman, {\em Topics in propagation of chaos}, Ecole d'Et\'e de Probabilit\'es de Saint-Flour XIX -1989 (P. L. Hennequin, ed.), Lecture Notes in Mathematics, vol. 1464, Springer, Berlin, Heidelberg, 1991, pp. 165--251.

\bibitem{Vara} V. S. Varadarajan, {\em On the convergence of sample probability distributions}, Sankhya \textbf{19} (1958), 23--26.
  
\bibitem{W-82} P. Walters, {\em An Introduction to Ergodic Theory}, Graduate Texts in Mathematics, vol. 79, Springer, New-York, 1982.
  
\bibitem{W-S} D. J. Watts, S. H. Strogatz, {\em Collective dynamics of small-world networks}, Nature {\bf 393} (1998), 440--442.
  
\bibitem{WC} H. R. Wilson, J. D. Cowan, {\em Excitatory and inhibitory interactions in localized populations of model neurons}, Biophys. J. {\bf 12} (1972), 1--24.

\bibitem{WJ} M. M. Woodman, V. K. Jirsa, {\em Emergent Dynamics from Spiking Neuron Networks through Symmetry Breaking of Connectivity,} PLoS ONE {\bf 8(5)} (2013), e64339.
\end{thebibliography}
\end{document}